\documentclass[reqno]{amsart}
\usepackage[utf8]{inputenc}

\usepackage[backend=biber, giveninits=true, style=alphabetic, sorting=nyt,isbn=false, maxalphanames=5,url=false, doi=false, eprint=false, maxbibnames=99]{biblatex}
\usepackage{ma_styles}

\title[Local constancy of superdeterminants]{On the local constancy of regularized superdeterminants along special families of differential operators}

\author{Michele Schiavina}
\address{Department of Mathematics, University of Pavia, Via Ferrata 5, 27100 Pavia, Italy}
    \address{INFN Sezione di Pavia, via Bassi 6, 27100 Pavia, Italy}
    \email{michele.schiavina@unipv.it}
    
\author{Thomas Stucker}
\address{Department of Mathematics, ETH Zurich, R\"amistrasse 101, 8092, Z\"urich, Switzerland}
\email{thomas.stucker@math.ethz.ch}

\begin{document}

\begin{abstract}
    We consider the flat-regularized determinant of families of operators of the form $D_\tau=[\delta_\tau,d_\nabla]$, where $\tau\to\delta_\tau$ are families of degree $-1$ maps in the twisted de Rham complex $\left(\Omega^\bullet(M,E),d_\nabla\right)$ generalizing the (twisted) Hodge codifferential. We show that under suitable assumptions, both geometrical and analytical in nature, the flat-regularized determinant of $D_\tau$, restricted to the subspace $\im(\delta_\tau)$, is constant in $\tau$. The general result we present implies both local constancy of the Ray--Singer torsion and of the value at zero of the Ruelle zeta function for a contact Anosov flow, upon choosing $\delta_\tau = \delta_{g_\tau}$, the Hodge codifferential for a family of metrics, and $\delta_\tau=\iota_{X_\tau}$, the contraction along a family of (regular, contact) Anosov vector fields, respectively.
\end{abstract}

\maketitle

\tableofcontents

\section{Introduction}
Let $E\to M$ be a flat orthogonal vector bundle over a smooth manifold and $\tau_{R}(M,E)$ its associated Reidemeister torsion, a topological invariant \cite{Reidemeister}. It is a classic result of Cheeger, and independently M\"uller \cite{Cheeger77,Mueller78,Cheeger79,Mueller93} (see also \cite{BismutZhang}), that $\tau_{R}(M,E)$ can be computed by the analytic torsion $\tau_{A}(M,E)$ of Ray and Singer \cite{ray_singer}. A crucial aspect of this result is that, while to define the analytic torsion a choice of Riemannian metric on $M$ is required, the quantity $\tau_{A}(M,E)$ does not depend on that choice, as was already observed by Ray and Singer. This result can be rephrased by stating that the analytic torsion is a (locally) constant function on the space of metrics.

A decade later, Fried observed that certain features of geodesic behavior on hyperbolic manifolds are robust under deformations. He showed that the value at $\lambda=0$ of the dynamical, or Ruelle, zeta function\footnote{As introduced by Ruelle in \cite{ruelle}.} associated to the geodesic flow on $E\to M$ (lifted to the unit cotangent bundle $S^*M$, with $M$ hyperbolic) coincides with the analytic torsion of the underlying manifold \cite{fried}. This result spurred further interest in the properties of zeta functions for Anosov dynamical systems\footnote{The geodesic flow on a hyperbolic metric is a basic example of such (Anosov) dynamical systems, which are in turn all examples of ``Axiom A'' systems according to Smale's characterisation \cite{Smale_dynamical}, which exhibit chaotic behavior.}, and in particular their dependence on the choice of the underlying dynamical system itself, giving rise to what is now known as Fried's conjecture \cite{fried,FriedLefschetz}. In its ``strong'' form, this is the statement that the value at zero of Ruelle's zeta function for \emph{any} Anosov flow on a manifold $M$ (endowed with a flat orthogonal bundle $E$) returns the analytic, and thus Reidemeister, torsion, i.e.\ it is a topological invariant.

This conjecture has been the object of intense study, and a number of results have provided an answer in the positive for certain classes of manifolds and flows (see \cite{SurveyShen} for a survey), although a full proof is not yet available. Among these results we focus in particular on \cite{viet_dang_fried_conjecture}, where a local constancy statement is proven, showing that the value at zero of the Ruelle zeta function $\zeta_X(0)$, viewed as a function of the Anosov vector field, is constant in a neighborhood of a regular Anosov flow, in some sense mirroring what was proven about the analytic torsion by Ray and Singer. The similarity between the two local constancy results becomes even more striking by observing that both quantities can be thought of as regularized superdeterminants of appropriate differential operators on the space of $E$-valued differential forms, and their local constancy is formulated in terms of a smooth family of such operators (see also \cite{hadfield_kandel_schiavina}).

The relevant families of operators one considers in this context have a special form: in both cases, they are obtained by considering a smooth family of generalised ``codifferentials'', denoted $\delta_\tau$, for the twisted de Rham complex $(\Omega^\bullet(M,E),\d)$, and by building a \emph{characteristic operator} via the graded commutator: $D_\tau \doteq [\delta_\tau,\d]$. Indeed, in the case of the analytic torsion, one considers the codifferential for a smooth family of metrics $\delta_\tau = \d^{*_{g_\tau}}$, and in the Ruelle case one considers the vector field insertion of a smooth family of Anosov vector fields into differential forms $\delta_\tau = \iota_{X_\tau}$. More explicitly:
\begin{equation}
    \begin{cases}
        \delta_\tau=\d^{*_{g_\tau}} \leadsto D_\tau = [\d,\d^{*_{g_\tau}}] = \Delta_{g_\tau} & \text{analytic torsion case}\\
        \delta_\tau=\iota_{X_\tau} \leadsto D_\tau = [\d,\iota_{X_\tau}] = \L_{X_\tau} & \text{Ruelle zeta function case}
    \end{cases}
\end{equation}
and one has that, in the appropriate sense,\footnote{Here $\sdet$ denotes the flat regularised (super)deteminant of differential operators, in the sense of \cite{baladi}, extended to graded vector spaces.}
\begin{equation}
\begin{cases}
    T_{\rho,g_\tau}(M) = \sdet(\Delta_{g_\tau}|_{L_{g_\tau}})^{\frac{1}{2}}, & \frac{d}{d\tau}\sdet(\Delta_{g_\tau}|_{L_{g_\tau}}) = 0\\
    \zeta_{\rho,X_\tau}(0)= \sdet(\L_{X_\tau}|_{L_{X_\tau}})^{(-1)^{\frac{\mathrm{dim}(M)-1}{2}}}, & \frac{d}{d\tau}\sdet(\L_{X_\tau}|_{L_{X_\tau}}) = 0
\end{cases}
\end{equation}
where $L_{g_\tau}\doteq(\im(\d^{*_{g_\tau}}))$ and $L_{X_\tau}\doteq(\im(\iota_{X_\tau}))$ are vector subspaces of the space of differential forms with values in the flat Hemitian bundle $E\to M$ (Section \ref{section_geometric}), and $T_{\rho,g_\tau}$ computes the analytic torsion $\tau_A(M,E)$.

In this paper we show that there exists a larger class of general codifferentials (for the twisted de Rham complex) such that the associated characteristic operators have locally constant regularized superdeterminants (when considered for families with suitable regularity properties). Moreover, we show that both the codifferential $\d^{*_g}$ and the contraction operator along a regular contact Anosov vector field $\iota_X$ belong to this class, which gives us another proof of both the local constancy theorems of Ray and Singer as well as that of Dang, Guillarmou, Rivi\'ere and Shen, in the case of contact Anosov vector fields, as a byproduct (cf.\ \cite[Theorem 2]{viet_dang_fried_conjecture}). Namely, we prove:

\begin{maintheoremintro}
     Let $(-1,1)\mapsto \delta_\tau$ be a smooth family of degree $-1$ maps in $\Omega^\bullet(M,E)$ such that $\delta_\tau^2=0$ and $\delta_\tau$ is a regular general codifferential for all $\tau\in(-1,1)$ (Definitions \ref{def_general_codifferential} and \ref{def_codifferential_types}). Denote the associated characteristic operator by $D_\tau = [\delta_\tau,\d]$ and by $L_\tau \doteq \im(\delta_\tau)$. If $\delta_\tau$ is an inner variation of $\delta_{0}$ (Definition \ref{def_homotopy_structure}) and if the additional assumptions \ref{assumption_diff_semigroup}, \ref{assumption_diff_hormander}, \ref{assumption_smaller_wavefront}, \ref{assumption_G_convergencce} and \ref{assumption_analytic_continuation} are satisfied, then for all $\tau\in(-1,1)$ we have
     \[
     \sdet({D_\tau}|_{L_\tau}) = \sdet({D_0}|_{L_0}).
     \]
\end{maintheoremintro}

Which leads to 
\begin{maincor}\label{maincor_1}
    The analytic torsion for an acyclic twisted de Rham complex is locally constant, as a function on the space of metrics over $M$.
\end{maincor}

\begin{maincor}\label{maincor_2}
    The value at zero of the Ruelle zeta function for a regular contact Anosov vector field is locally constant, as a function on the space of regular contact Anosov vector fields.
\end{maincor}

The structure we employ to build paths of operators is inherited from, and inspired by, the mathematical physics literature, where the general codifferential plays the role of ``gauge fixing operator''. More generally, one looks for a Lagrangian subspace\footnote{This is defined as an isotropically complemented isotropic subspace of a vector space with a nondegenerate pairing.} of an appropriately defined space of fields. (These are often sections of graded vector bundles, and come equipped with a cohomology theory that encodes the physics of the problem.) In such scenario, a family of codifferentials is interpreted as a homotopy of ``gauge fixing Lagrangians'', and quantities of physical relevance are expected to be locally constant along such families. (A more careful formulation of this statement can be made into a theorem in the case of finite dimensional spaces of fields, see \cite{cattaneo-mnev-schiavina} and references therein). One quantity of particular interest is the partition function, which---in the case of a certain topological theory whose fields are differential forms---has been linked to both the analytic torsion and the value at zero of the Ruelle zeta function \cite{hadfield_kandel_schiavina,Schiavina_Stucker2} (see Section \ref{sec_Fried}, below). 

We conclude by mentioning that the assumptions we must impose for local constancy to hold have been found as natural requirements to make sense of the $\tau$ derivative of the regularized superdeterminant. It turns out that in special cases (or classes thereof) some of these assumptions are automatically satisfied as a byproduct of more general considerations on the nature of the characteristic operators (e.g.\ in the case of $D_\tau$ elliptic). Our work, however, aims to link operators that are potentially very different in nature, such as the second order elliptic Hodge Laplacian and the first order Lie derivative. Assuming that the characteristic operators $D_\tau$ belong to some class of differential operators requires a careful search for such a class---an attempt that would take our work closer to the theory of hypoelliptic Laplacians due to Bismut \cite{BismutSurvey}.\footnote{Note, in particular, that in \cite[Section 2]{BismutSurvey} some constructions inspired by Witten's Laplacian bring the two works close together. Naturally, Witten's work \cite{WittenMorse} is rooted---like ours---in the problem of gauge fixing of field theories. See \cite{Schiavina_Stucker2} for an overview and a link to the present work.}

\subsection{A remark concerning Fried's conjecture}\label{sec_Fried}
When $X$ is a smooth Anosov vector field on a compact manifold $M$ endowed with a Hermitian vector bundle $(E,\langle\cdot,\cdot\rangle_E,\nabla)\to M$ with flat connection (induced by a unitary representation $\rho$ of $\pi_1(M)$ on the fibers of $E$), it is conjectured \cite{fried,FriedLefschetz} that
\[
T_{\rho,g} = |\zeta_{\rho,X}(0)|^{(-1)^m},
\]
with $\dim(M)=2m+1$. Both sides of Fried's conjecture can be expressed in terms of superdeterminants of characteristic operators for various choices of general codifferentials for the twisted de Rham complex; indeed, we have seen that the analytic torsion is computed by
\[
T_{\rho,g}= \sdet(\Delta_g\vert_{L_g})^{\frac12}, \qquad L_g=\im(\delta_g)
\]
and the value at zero of Ruelle's zeta function reads
\[
\zeta_{\rho,X}(0)=\sdet(\L_X\vert_{L_X})^{(-1)^m}, \qquad L_X = \im(\iota_X).
\]
Note that a direct interpolation between the two quantities in terms of families of codifferentials, following our construction, should not be expected because of the mismatch in power of the superdeterminants involved. One can get a sense for the origin of this mismatch, by observing that one can reformulate both sides as
\begin{equation*}
\begin{split}
    T_{\rho,g} &= \sdet(\Delta_g\vert_{L_g})^{\frac12} =\frac{\sdet(\Delta_g\vert_{L_g})}{\sdet(\delta_g)} \doteq Z(\delta_g);\\
    \zeta_{\rho,X}(0)^{(-1)^m} &= \sdet(\L_X\vert_{L_X}) = \frac{\sdet(\L_X\vert_{L_X})}{\sdet(\iota_X)} \doteq Z(\delta_X)
\end{split}
\end{equation*}
where we defined $\sdet(\delta_g)\doteq\sdet(\delta_g\d\vert_{L_g})^{\frac12}=\sdet(\Delta_g\vert_{L_g})^{\frac12}$ and $\sdet(\iota_X)\doteq\sdet(\iota_X \circ \alpha\wedge)^{\frac12} = 1$. (The latter observation is justified by the fact that on the cosphere bundle over a Riemannian manifold one can define an ``adapted metric'' and an inner product on differential forms w.r.t.\ which $\iota_X$ and $\alpha\wedge$ are adjoints, see \cite[Section 2.3.1]{hadfield_kandel_schiavina}.)

Indeed, both the analytic torsion and the value at zero of the Ruelle zeta function have been interpreted as the partition function for a topological field theory, where the general codifferential plays the role of (a choice of) gauge-fixing operator, and the denominator ``$\sdet(\delta)$'' is interpreted as a Jacobian determinant factor \cite{hadfield_kandel_schiavina}. 

The following natural questions arise:
\begin{enumerate}
    \item Is there a family of general codifferentials $\delta_\tau$ interpolating between $\iota_X$ and $\delta_g$, with $X$ a (regular) Anosov vector field and $g$ a Riemannian metric on $M$, and a well-defined notion of $\sdet(\delta_\tau)$?
    \item Is the quantity 
    \[
    Z(\delta_\tau) = \frac{\sdet(D_\tau\vert_{L_\tau})}{\sdet(\delta_\tau)}
    \]
    locally constant along the interpolating family $\delta_\tau$?
\end{enumerate}
This suggests a path for a reformulation of Fried's conjecture, akin to the constructions presented in this paper. Our results do not directly apply to this question, but they imply local constancy of the functionals $Z(\delta_\tau)$ for the particular cases where $\delta_\tau$ is a family of Hodge codifferentials, or of contractions along regular Anosov vector fields, because in these cases the ``Jacobian factor'' $\sdet(\delta_\tau)$ can be handled more easily.

\medskip

\section{Geometric setup}
\label{section_geometric}
Let $M$ be an $n$-dimensional compact, connected, orientable manifold, and let $\pi\colon E\to M$ be a rank $r$ complex vector bundle over $M$. We further assume that $E$ is equipped with a smooth Hermitian inner product, which we denote by $\pair{\,\cdot\,}{\cdot\,}_E$, and a flat connection $\nabla$ compatible with the inner product. Denote by $\Omega^k(M,E)$ the space of smooth differential $k$-forms on $M$ with values in $E$. Since the connection is flat, the exterior covariant differential associated to $\nabla$,
\[ \d \colon \Omega^\bullet(M,E) \to \Omega^{\bullet+1}(M,E), \]
satisfies $\d\! \circ\, \d = 0$, and thus defines a cochain complex, called twisted de Rham complex, with cohomology groups $H^\bullet(M,E)$. We require this complex to be acyclic, i.e.\ $H^k(M,E) = 0$, for all $0\le k\le n$.

\begin{defn}[Twisted topological data]\label{def:TTD}
The data $(M,E,\nabla)$ is called \emph{twisted topological data}.
\end{defn}

\subsection{General codifferentials}\label{section_codifferentials}
On $\Omega^\bullet(M,E)$ there is a natural nondegenerate pairing 
\begin{equation}\label{e_pairing}
\int_M\langle \cdot,\cdot \rangle_E \colon \Omega^\bullet(M,E)\times \Omega^\bullet(M,E) \to \mathbb{C}.
\end{equation}
Here, for $\omega,\eta\in\Omega^\bullet(M,E)$ not necessarily homogeneous in degree, $\langle\omega\wedge\eta\rangle_E$ denotes taking the inner product in $E$ and the exterior product in $\wedge^\bullet T^*M$, and the integral is only of the top-form part of the resulting expression.

We introduce the notion of isotropic subspaces of $\Omega^\bullet(M,E)$, thought of as a graded vector space. Note that these will not be homogeneous in degree, in general.

\begin{defn}
    A subspace $I\subset \Omega^\bullet(M,E)$ is isotropic w.r.t.\ $\int_M\langle\cdot,\cdot\rangle_E$ iff
    \[
    \int_M\langle\omega,\eta\rangle_E \qquad \forall\omega,\eta\in I.
    \]
\end{defn}

\begin{defn}
\label{def_general_codifferential}
A general codifferential is a nilpotent differential operator
$$\delta: \Omega^\bullet(M,E) \to \Omega^{\bullet-1}(M,E),\qquad \delta^2 = 0,$$
which enjoys graded symmetry w.r.t.\ the pairing:\footnote{The sign depends on the degree of the homogeneous components.}
$$
\int_M\langle\omega\wedge \delta\eta\rangle_E = \pm \int_M\langle \delta\omega\wedge\eta\rangle_E \quad \mathrm{for}\quad \omega,\eta \in \Omega^\bullet(M,E),
$$
such that $L\doteq \im(\delta)$ admits a closed complement $C \subset \Omega^\bullet(M,E)$, i.e.\ 
\begin{equation}
\label{splitting}
    \Omega^\bullet(M,E) = L\oplus C.
\end{equation}
Given a general codifferential $\delta$ the \emph{characteristic operator (of the codifferential)} is the degree-preserving operator\footnote{This is the graded commutator of odd derivations of $\Omega^\bullet(M,E)$. Thus, $D=[\delta,\d]=\delta\d+\d \delta$.}
\[
D\doteq [\delta,\d] \colon \Omega^\bullet(M,E)\to \Omega^\bullet(M,E).
\]
Note that $D$ commutes with $\delta$ (and with $\d$), so that the subspcae $L=\im(\delta)$ is left invariant by $D$. We use the term \emph{restricted characteristic operator} for the restriction $D\vert_L$.
\end{defn}

We have the following:

\begin{lemma}
\label{inverse_Q_GF}
Let $\delta$ be a general codifferential. Then the subspace $L=\im(\delta)$ is isotropic. Moreover, let $C$ be a closed complement to $L$, i.e.  $\Omega^\bullet(M,E)=L\oplus C$, and consider the following statements: 
\begin{enumerate}
    \item $C$ is isotropic;
    \item the map
        \begin{equation}
        \label{gauge_fixing_isomorphism}
        \delta\vert_C: C \to L
        \end{equation}
        is bijective;
    \item the complex $(\Omega^\bullet(M,E),\delta)$ is acyclic.
\end{enumerate}
We have that 
\[
(1)\implies (2)\iff (3).
\]
\end{lemma}
\begin{proof}
We first check that $L\doteq \im(\delta)$ is isotropic in $\left(\Omega^\bullet(M,E),\int_M\langle\cdot,\cdot\rangle_E\right)$. Indeed,
\[
\int_M\langle \omega,\eta\rangle_E = \int_M\langle\delta \tau,\delta \beta\rangle = \pm \int_M \langle \tau,\delta^2\beta\rangle = 0, \qquad \forall \omega,\eta\in L.
\]

Let us prove $(1)\implies (2)$. Consider $\omega \in \ker(\delta) \cap C$, then by isotropicity of $C$ we have
\[
\int_M \langle\omega\wedge\eta\rangle_E = 0, \quad \forall \eta \in C
\]
as well as
\[
\int_M \langle\omega\wedge\eta\rangle_E = \int_M \langle\omega\wedge\delta\alpha\rangle_E = \pm\int_M \langle\delta \omega\wedge\alpha\rangle_E = 0, \quad \forall \eta \in L.
\]
Hence we conclude
$$\int_M \langle\omega\wedge\eta\rangle_E = 0, \quad \forall \eta \in \Omega^\bullet(M,E),$$
so $\omega = 0$ by non-degeneracy of the pairing, and we deduce that $\delta$ is injective on $C$, i.e.\ $$\ker(\delta) \cap C = \{0\}.$$ 
Nilpotency of $\delta$ further implies
\begin{equation}
\label{delta_surjective}
    L = \delta\bigl(\Omega^\bullet(M,E)\bigr) = \delta\bigl(\im(\delta)\oplus C\bigr) = \delta(C),
\end{equation}
so that $\delta\vert_C$ is surjective onto $L$.

To show $(2)\implies (3)$, note that the existence of a complementary subspace $C$ with $\delta|_C$ bijective implies in particular that $\ker(\delta)\cap C=\{0\}$. Since $L=\im(\delta)\subset\ker(\delta)$ and $\Omega^\bullet(M,E)=L\oplus C$, we must have $\ker(\delta) = \im(\delta)$, i.e.\ the complex defined by the general codifferential is acyclic. 

Let us now show that $(3)\implies (2)$. Given any complement $C$ of $L$, acyclicity of the complex $(\Omega^\bullet(M,E),\delta)$ implies $\ker(\delta)=\im(\delta)=L$, and $L\cap C=\{0\}$ shows that $\ker(\delta)\cap C = L\cap C = \{0\}$, so we again have injectivity of $\delta$ on the complement. The surjectivity of $\delta|_C$ follows from \eqref{delta_surjective}.

\end{proof}

\begin{rmk}[Codifferential and ``gauge fixing'']\label{rmk:codiff_GF}
We should mention that our notion of general codifferential is inspired by the mathematical physics literature, where one talks about ``gauge-fixing operator''. There, it is often commonplace to take $C= \im(\d)\oplus \ker(D)$ (see e.g.\ \cite{costello}), as a generalisation of (and inspired by) the Hodge decomposition. Indeed, one could consider a general codifferential $\delta$ that induces a splitting
\[
\Omega^\bullet(M,E) \stackrel{!}{=} \im(\delta) \oplus \underbrace{\im(\d) \oplus \ker(D)}_C,
\]
where $D=[\delta,\d]$ is the characteristic operator of the codifferential.

When, additionally, $\ker(D)=\{0\}$, the complement $C$ reduces to $\im(\d)$ and it is isotropic. However, this scenario turns out to be insufficient for the purposes of Section \ref{section_ruelle_independence}. This is due to the fact that the codifferential $\delta = \iota_X$, for some appropriately chosen vector field $X$, does not induce a splitting $\Omega^\bullet(M,E) = \im(\iota_X) \oplus \im(\d) \oplus \ker(\L_X)$ for \emph{smooth} differential forms. (See \cite{Schiavina_Stucker2} for a detailed discussion, and Remark \ref{rmk:gooddecompositionaxial}.) In such cases one looks for a more general complement $C$. 
\end{rmk}

\begin{rmk}
\label{remark_complexes}
In the examples we will encounter, the subspace $C$ is also given as the image of an operator $\delta^\bot:\Omega^\bullet(M,E) \to \Omega^{\bullet+1}(M,E)$, which squares to zero $(\delta^\bot)^2=0$ and enjoys graded symmetry w.r.t.\ the pairing: $\int_M\langle \delta^\bot\omega,\eta \rangle_E = \pm\int_M\langle \omega, \delta^\bot\eta \rangle_E$. This in particular ensures that $C$ is isotropic, and thus $\delta\vert_C$ is a bijection. Overall, then, we have to deal with three chain/cochain complexes; the twisted de Rham complex given by the operator $\d$
\[\begin{tikzcd}
\cdots \arrow[r] & \Omega^k(M,E) \arrow[r, "\d"] & \Omega^{k+1}(M,E) \arrow[r] & \cdots
\end{tikzcd}\]
and the two complementary complexes coming from the codifferential structure
\[
\begin{tikzcd}
\cdots \arrow[r, shift right] & \Omega^k(M,E) \arrow[l, shift right] \arrow[r, shift right, "\delta^\bot"'] & \Omega^{k+1}(M,E) \arrow[l, shift right, "\delta"'] \arrow[r, shift right] & \cdots \arrow[l, shift right]
\end{tikzcd}
\]
\end{rmk}

We will be interested in the characteristic operators $D$ obtained by different choices of general codifferential. In particular, we want to investigate how different choices of codifferential affect the value of the determinant of the ensuing (restricted) characteristic operator $D\vert_L$, a notion that requires the introduction of an appropriate regularization scheme. Since $D$ is a degree-preserving operator acting on a graded vector space, we need to extend the notion of regularized determinant to that of a regularized \emph{super}-determinant, and due to the nature of the characteristic operators we will encounter, zeta-regularization turns out not to be sufficiently general. Therefore, we will employ the notion of \emph{flat determinant} (see e.g.\ \cite{baladi,zworski}). For a review, relevant to this paper, see \cite[Appendix B]{Schiavina_Stucker2}. 

According to the properties of $D$, we distinguish the following general codifferentials:

\begin{defn}\label{def_codifferential_types}
A general codifferential (Definition \ref{def_general_codifferential}) is said to be 
\begin{itemize}
    \item \emph{acyclic} iff the complex $(\Omega^\bullet(M,E),\delta)$ is acyclic;
    \item \emph{transversal} iff the characteristic operator $D=[\delta,\d]$ is injective ($\ker(D)=\{0\}$) and there exists a complement $C$ of $L=\im(\delta)$ which is $D$-invariant.
    \item \emph{regular} iff it is acyclic, transversal and the flat determinant $\sdet(D\vert_L)$ is well-defined and nonvanishing.
\end{itemize}
\end{defn}

\begin{rmk}
    Acyclicity of the codifferential is sufficient to have the isomorphism of Equation \eqref{gauge_fixing_isomorphism}, which will be essential later on to deal with the restricted characteristic operator $D\vert_L$. When the complement $C$ is given as the image of a differential operator $\delta^\bot$, the second item in the transversality requirement is implied by the condition $[D,\delta^\bot]=0$.
\end{rmk}

\subsection{The field-theoretic perspective}
In this section we give a hint of an interpretation of the structure presented above, to appreciate its origin and provide potential generalisations.

Let us view $\Omega^\bullet(M,E)$ as a $\Z$-graded vector space where $k$-forms are assigned degree $k$.\footnote{This means that $(\Omega^\bullet(M,E))^{k}=\Omega^{k}(M,E)$.} Denote by $\Omega^\bullet(M,E)[j]$ the (negative) $j$-shift of this graded vector space, that is $(\Omega^\bullet(M,E)[j])^k=(\Omega^\bullet(M,E))^{k-j}$, i.e.\ the degree of a $k$-form in $\Omega^\bullet(M,E)[j]$ is $k-j$. We then consider the shifted densitised cotangent bundle as the vector space:
\begin{equation}\label{e:twocopies}
    \mathcal{F}\doteq T^\vee[-1](\Omega^\bullet(M,E)[1]) \doteq \Omega^\bullet(M,E)[1] \oplus \Omega^\bullet(M,E)[n-2]
\end{equation}
whose elements $(\A,\B) \in \F$ consist of a pair of inhomogeneous differential forms
$$\A = \A^{(0)} + \cdots + \A^{(n)}, \quad \B = \B^{(0)} + \cdots + \B^{(n)},$$
where $\A^{(k)}$ is an $E$-valued $k$-form of assigned degree $|\A^{(k)}| = k-1$, and $\B^{(k)}$ is an $E$-valued $k$-form of assigned degree $|\B^{(k)}| = k - n + 2$. 

On $\F$ there is an odd weak-symplectic pairing\footnote{The map $\Omega^\sharp\colon \F\to\F^*$ has odd degree and it is injective. Indeed, since $\B\wedge\A$ must be an $n$-form, we are pairing $\A^{(k)}$ with $\B^{(n-k)}$ and $|\A^{(k)}| + |\B^{(n-k)}|=+1$. Since $\Omega$ then maps into $\mathbb{R}$, it has (odd) degree $-1$.} $\Omega: \F\times\F \to \R$, given by
\[
    \Omega((0,\B),(\A,0)) = \int_M \langle\B\wedge\A\rangle_E,
\]
and extended by graded skew-symmetry and linearity to all of $\F\times\F$. The space of differential forms over $M$ compact can be given a family of seminorms that make $(\F,\Omega)$ into a weak-symplectic nuclear Fr\'echet vector space (see e.g.\ \cite{kriegl_michor}).

\begin{defn}
\label{def:lagrangian_subspace}
A subspace $\sfL\subset \F$ is said to be Lagrangian iff it is isotropic, i.e.\ $\Omega\vert_{\sfL}=0$, and there is a splitting $\mathcal{F} = \sfL\oplus\sfK$, where $\sfK$ is also isotropic.
\end{defn}

\begin{rmk}
Note that this implies $\sfL$ is a maximal isotropic subspace, and $\Omega$ yields a non-degenerate pairing between $\sfL$ and $\sfK$. This notion is often referred to as the \emph{split Lagrangian} property of a subspace \cite{cattaneo-contreras-split}. Other notions exist, and we refer to \cite{weinstein} for an introduction to Lagrangian subspaces in infinite dimensions.
\end{rmk}

Now, if we are given a general codifferential, in the sense of Definition \ref{def_general_codifferential}, we can construct the subspaces
$$\sfL = \im(\delta)\oplus\im(\delta) \subset \F, \quad {\mathsf{C}} = C \oplus C \subset \F.$$
These will turn out to both be Lagrangian in the sense of Definition \ref{def:lagrangian_subspace}, if we require $C$ isotropic, since we have 
$$\F = \sfL\oplus{\mathsf{C}}, \quad \mathrm{with}\quad \Omega|_{\sfL} = 0, \quad \Omega|_{{\mathsf{C}}} = 0.$$

In the mathematical physics literature, the subspace $\mathsf{L}$ is called a \emph{gauge-fixing} for the topological field theory defined by the ``space of fields'' $\mathcal{F}$ with cohomology theory induced by $\d$ on $\mathcal{F}$. Such a formulation of a field theory is due to Batalin and Vilkovisky, who also proved that---in finite dimensions---one has a local constancy result along families of Lagrangians $\sfL_\tau$. One can interpret the results in this paper as an infinite-dimensional generalisation of the BV theorem, in this specific instance. See \cite{hadfield_kandel_schiavina,Schiavina_Stucker2} and references therein.


\section{Flat superdeterminants of characteristic operators}
\label{section_independence}

In this section, we will study smooth families of general codifferentials ${\delta_\tau}$, and the flat superdeterminant of the associated (restricted) characteristic operators (Definition \ref{def_general_codifferential}):
\[
\sdet(D_\tau\vert_{L_\tau}), \qquad D_\tau = [\delta_\tau,\d], \qquad L_\tau = \im(\delta_\tau).
\]
In particular, we show how for families ${\delta_\tau}$ exhibiting the special structure of Section \ref{section_homotopy_structure}---and under suitable regularity assumptions---one can prove local constancy of the flat determinant in $\tau$. This result recovers invariance of the analytic torsion for a smooth family of metrics \cite{ray_singer}, and of the value at zero of the Ruelle zeta function for a smooth family of Anosov vector fields \cite{viet_dang_fried_conjecture}. Both statements are shown to fit within a more general framework, developed here.

\subsection{Structure of the general codifferential variation}
\label{section_homotopy_structure}
Let $\tau\in(-1,1) \mapsto {\delta_\tau}$ be a smooth family of general codifferentials, as in Definition \ref{def_general_codifferential}, which we view as a variation of the general codifferential ${\delta_0}$. Such a family can be thought of as a homotopy of isotropic subspaces ${L_\tau\doteq\im({\delta_\tau})\subset\Omega^\bullet(M,E)}$, see Lemma \ref{gauge_fixing_isomorphism}.

We restrict our attention to a particularly nice, yet fairly large, class of families of general co\-differentials. We will see that both the Hodge codifferentials for a family of metrics (Section \ref{sec:ATinvariance}) and the contraction w.r.t.\ families of Anosov vector fields (Section \ref{section_ruelle_independence}) exhibit this structure.

\begin{defn}
\label{def_homotopy_structure}
We say that a smooth family of general codifferentials $\tau \in (-1,1) \mapsto {\delta_\tau}$ is \emph{an inner variation} of ${{\delta_0}}$ if there exists a smooth family of degree-preserving local operators
\begin{equation*}
    \theta_\tau: \Omega^\bullet(M,E) \to \Omega^\bullet(M,E), \quad \forall \tau \in (-1,1)
\end{equation*}
such that
\begin{equation*}
    \frac{d}{d\tau} {\delta_\tau} = [\theta_\tau,{\delta_\tau}], \quad \forall \tau \in (-1,1).
\end{equation*}
\noindent A smooth family of general codifferentials $ \tau \in (-1,1) \mapsto {\delta_\tau}$ is said to be \emph{integrable} iff there is a smooth family of invertible and degree-preserving local maps $\beta_\tau\colon \Omega^\bullet(M,E)\to\Omega^\bullet(M,E)$ such that
\begin{equation}\label{e:regGFhomotopy}
     {\delta_\tau} = \beta_\tau \circ {\delta_0} \circ \beta_\tau^{-1}.
\end{equation}

\end{defn}

\begin{lemma}
    An integrable smooth family of general codifferentials is an inner variation.
\end{lemma}
\begin{proof}
    By definition, an integrable smooth family of general codifferentials admits a smooth path $\tau \to {\beta}_{\tau}$ in the space of degree-preserving local maps, acting by conjugation on $\Omega^\bullet(M,E)$. 
    
    Taking the $\tau$ derivative of Equation \eqref{e:regGFhomotopy}, we find
    \begin{equation}
    \label{homotopy_derivative1}
    \begin{split}
        \frac{d}{d\tau}{\delta_\tau} &= \Bigl(\frac{d}{d\tau}{\beta}_{\tau}\Bigr){{\delta_0}}{\beta}_{\tau}^{-1} + {\beta}_{\tau}{{\delta_0}}\Bigl(\frac{d}{d\tau}{\beta}_{\tau}^{-1}\Bigr) \\
        &= \Bigl(\frac{d}{d\tau}{\beta}_\tau\Bigr){\beta}_{\tau}^{-1}{\beta}_{\tau}{{\delta_0}}{\beta}_{\tau}^{-1} - {\beta}_{\tau}{{\delta_0}}{\beta}_{\tau}^{-1}\Bigl(\frac{d}{d\tau}{\beta}_\tau\Bigr){\beta}_{\tau}^{-1} \\
        &= \theta_{\tau}{\delta_\tau} - {\delta_\tau}\theta_{\tau}\\
        &= [\theta_\tau,\delta_\tau],
    \end{split}
    \end{equation}
    where we defined 
    \begin{equation}
        \theta_{\tau} = \Bigl(\frac{d}{d\tau}{\beta}_\tau\Bigr){\beta}_{\tau}^{-1}. 
    \end{equation}
    Thus, we see that for an integrable family of general codifferentials the derivative $\dot{\delta}_{\tau}$ can be obtained from ${\delta_\tau}$ as the commutator with the degree-preserving map $\theta_\tau$. 
    \end{proof}

\begin{rmk}
    The two examples that we consider in this paper both showcase smooth families of general codifferentials that are of the integrable type. In fact, in these examples one can deduce the operator $\theta_\tau$ from the derivative of a bundle automorphism $\beta_\tau$. However, our main result on local constancy under general codifferential variations works with the more general definition of an inner variation. It is tempting to argue that inner variations come from a Lie algebra action on the space of differential operators on $\Omega^\bullet(M,E)$, and that - as such - they may be integrated to a Lie group action, which would then yield the adjunction by a map $\beta_\tau$, interpreted as the flow of $[\theta_\tau,\cdot]$ seen as a vector field. However, even if this procedure can be made sense of, there is no reason that $\beta_\tau$ and its inverse should be local maps. In particular, note that the degree of the differential operator $\delta_\tau$ can change along an inner variation. We will defer this investigation, i.e. the question of whether an inner variation is integrable, to future work.
\end{rmk}

\subsection{Local constancy of the flat superdeterminant}
\label{subsection_invariance}

The purpose of this section is to study the $\tau$ dependence of the flat superdeterminant of the  characteristic operator of the general codifferential, namely:
\begin{equation}
\label{e_flat_super_det}
    \tau \to \sdet({D_\tau}|_{L_\tau}).
\end{equation}
We will show that for a smooth family of regular general codifferentials (see Definition \ref{def_codifferential_types}) which stems from an inner variation as in Definition \ref{def_homotopy_structure}, and under certain additional existence conditions of analytic nature, the flat superdeterminant in \eqref{e_flat_super_det} is constant in $\tau$.

The geometric reason for this invariance is the structure of $D_\tau = {\delta_\tau}\d+\d {\delta_\tau}$ as a characteristic operator as expounded in Section \ref{section_geometric}, together with the inner variation structure of the family of general codifferentials in Definition \ref{def_homotopy_structure}. Indeed, if $D_\tau$ were a smooth family of operators of this form acting on a finite dimensional graded vector space, invariance of the superdeterminant would follow from these algebraic properties. However, we are working on the infinite dimensional graded vector space $\Omega^\bullet(M,E)$, so additional assumptions on $D_\tau$ are necessary to ensure that certain expressions are well-defined.

For specific classes of characteristic operators, these additional assumptions can be eliminated or at least simplified. For instance, in Section \ref{sec:ATinvariance} we show that the conditions of Theorem \ref{thm_invariance} are satisfied when $D_\tau$ is a family of positive definite elliptic operators, and in Section \ref{section_ruelle_independence} we prove that the conditions hold for the case when $D_\tau = \L_{X(\tau)}$ is the Lie Derivative with respect to a family of Anosov flows. This will allow us to infer two well-known invariance results from the more general invariance result below; namely, the invariance of the analytic torsion with respect to variations of the metric, \cite{ray_singer}, and the invariance of the value at zero of the Ruelle zeta function with respect to variations of the Anosov flow, \cite{viet_dang_fried_conjecture}.

These two cases are concerned with characteristic operators of a very different analytic nature, the Laplacian and the Lie derivative, but the local constancy of the regularized superdeterminant can be seen as stemming from the same basic structure. We thus formulated Theorem \ref{thm_invariance} for a general family of characteristic operators, requiring $D_\tau$ to satisfy properties sufficient for the well-behavedness of the relevant objects. We do not claim that this set of assumptions is optimal, i.e. that they are the necessary conditions for the invariance of \eqref{e_flat_super_det}. Before stating the theorem, we provide a brief motivation of the assumptions.

For the invariance statement to make sense, we must first assume that the curve $\tau \to \sdet({D_\tau}|_{L_\tau})$ exists, i.e. that the flat determinant is well-defined for each $\tau$, see \cite[Appendix B]{Schiavina_Stucker2} for the notion of flat determinant used here. In terms of the general codifferential, we say that ${\delta_\tau}$ is regular for all $\tau$ (Definition \ref{def_codifferential_types}).

Recall that this entails that $D_\tau$ generates a strongly continuous semigroup $\exp(-tD_\tau)$ on $\Omega^\bullet(M,E)$ whose flat trace restricted to the general codifferential subspace $L_\tau=\im(\delta_\tau)$ is well-defined as a distribution in $t\in\mathbb{R}_+$. In particular, for each $\tau$ the wavefront set of the Schwartz kernel of $\exp(-tD_\tau)$, viewed as a distribution\footnote{Here and in the following, we will often suppress the vector bundle in which the distribution takes its values from our notation, when this bundle is clear from the context.} $K_\tau \in \D(\R_+\times M\times M)$, is disjoint from
$$N^*\iota = \bigl\{(t,0,x,\xi,x,-\xi)\in T^*(\R_+\times M\times M)\ |\ t\in\R_+, (x,\xi)\in T^*M\bigr\},$$
the conormal bundle to the diagonal map $\iota\colon \R_+\times M\to \R_+\times M\times M$, $\iota(t,x)=(t,x,x)$.
The flat determinant of $D_\tau\vert_{L_\tau}$ can then be expressed in terms of the following auxiliary function depending on two complex parameters $\lambda,s\in\mathbb{C}$:
\begin{equation}
\label{det_function}
    F(\tau,\lambda,s) \doteq \frac{1}{\Gamma(s)}\lim_{N\to\infty}\pair{\str\left(e^{-t{D_\tau}}\big|_{L_\tau}\right)}{t^{s-1}e^{-\lambda t}\,\chi_N(t)},
\end{equation}
where $\{\chi_N\}_{N\in\N}$ is a sequence of cutoff functions
\begin{equation}
\label{cutoff_functions}
    \chi_N\in\C_c(\R_+) \quad\text{with}\quad 0\leq \chi_N \leq 1 \quad\text{and}\quad \chi_N \equiv 1 \,\text{ on }\, \bigl[\tfrac{1}{N},N\bigr].
\end{equation}
The regularity condition on ${\delta_\tau}$ is then translated into the requirement that the limit in \eqref{det_function} converges locally uniformly for $\Re(\lambda),\Re(s)$ large enough and that $F(\tau,\lambda,s)$ has an analytic continuation to $\lambda=0$, $s=0$ for each $\tau\in(-1,1)$. We then have (c.f.\ \cite[Appendix B]{Schiavina_Stucker2})
\[
\sdet({D_\tau}|_{L_\tau}) \doteq \exp\Bigr(-\frac{d}{ds} F(\tau,\lambda,s)\Big|_{s=0,\,\lambda=0}\,\Bigr).
\]
Note that, by definition of transversal general codifferential (Definition \ref{def_codifferential_types}), for each $\tau$ there is a splitting of the form $\Omega^\bullet(M,E) = {L_\tau}\oplus {C_\tau}$, for some $\tau$-dependent complement ${C_\tau}$ which is left invariant by $D_\tau$. We use the projection operator $\Pi_{{L_\tau}}$ induced by this splitting to define the flat determinant of the restricted operator. (More details regarding this restriction procedure are reported in \cite[Appendix B]{Schiavina_Stucker2}. See in particular Definition \cite[Definition B.12]{Schiavina_Stucker2}.)

Then, the logarithm of the flat superdeterminant of the characteristic operator reads
\begin{equation}
\label{det_from_F}
    \log\sdet({D_\tau}|_{L_\tau}) = -\frac{d}{ds}F(\tau,\lambda,s)\Big|_{\lambda=0,\,s=0},
\end{equation}
whence the proof of its invariance relies on a careful study of the $\tau$-derivative of the function $F(\tau,\lambda, s)$. In order to ensure that this $\tau$-derivative is well-defined, at least for $\Re(\lambda),\Re(s)$ large enough, we must make certain differentiability assumptions for the semigroup $\exp(-tD_\tau)$.

In Assumption \ref{assumption_diff_semigroup}, we require that the semigroup of operators itself is continuously differentiable. Note that for general $D_\tau$ this does not immediately follow from the smoothness of $\tau\to D_\tau$. Differentiability of the semigroup will allow us to derive a formula for the derivative $\frac{d}{d\tau}\exp(-tD_\tau)$, see Lemma \ref{duhamel_lemma}.

We further need this differentiability to carry over to the flat trace of $\exp(-tD_\tau)$. Recall that $\tr:\D_\Gamma(\R_+\times M^2)\to\D(\R_+)$ defines a sequentially continuous linear functional with respect to the H\"ormander topology, see \cite[Definition 8.2.2]{hormander1} for this notion, where $\D_\Gamma(\R_+\times M^2)$ is the space of distributions with wavefront set contained in a closed conic set $\Gamma$ disjoint from the conormal to the diagonal $N^*\iota$. Thus, a natural condition which ensures the differentiability of the flat trace of the semigroup is that $\tau \to \exp(-tD_\tau)$ should be differentiable with respect to the H\"ormander topology, i.e. the difference quotients converge in the H\"ormander seminorms, see Assumption \ref{assumption_diff_hormander}. 

For technical reasons that arise in the proof of Lemma \ref{lemma_tau_to_t}, we must further place a stronger restriction on the wavefront sets of $\exp(-tD_\tau)$. Although disjointness from $N^*\iota$ is enough to guarantee the existence of the flat trace, our invariance proof requires $\WF(\exp(-tD_\tau))$ to satisfy the stricter Assumption \ref{assumption_smaller_wavefront}. Note that this guarantees precisely that the integrand in formula \eqref{duhamel}, below, has wavefront set disjoint from the conormal bundle to the map 
\begin{equation}\label{e_diagonal}
\tilde{\iota}:(t,u,x)\in\R_+^2\times M \to (t,u,x,x)\in\R_+^2\times M^2.
\end{equation}

For $(\lambda,s)$ contained in the domain where the limit in \eqref{det_function} is well-defined, we will express the $\tau$-derivative of $F$ in terms of another auxiliary function $G$, which we define as follows:
\begin{equation}
\label{G_function}
    G(\tau,\lambda,s) \doteq \lim_{N\to\infty} \pair{\str\big(\theta_{\tau}e^{-t{D_\tau}}\big)}{t^{s-1}e^{-\lambda t}\chi_N(t)},
\end{equation}
where $\theta_\tau$ is the operator defining the inner variation ${\delta_\tau}$ (Definition \ref{def_homotopy_structure}). In order for the above limit to exist and define an analytic function of $(\lambda,s)$, we will require locally uniform convergence in $(\tau,\lambda,s)$ for $\Re(\lambda),\Re(s)$ large enough, see Assumption \ref{assumption_G_convergencce}. This also guarantees that $\tau$-differentiability carries over to the limit.

Indeed, using the inner variation structure of ${\delta_\tau}$ and our additional assumptions, we show in Lemma \ref{lem:F'intermsofG} that
\begin{equation}
\label{e:F'intermsofG}
    \frac{d}{d\tau}F(\tau,\lambda,s) = \frac{\lambda}{\Gamma(s)}G(\tau,\lambda,s+1) - \frac{s}{\Gamma(s)}G(\tau,\lambda,s),
\end{equation}
when $(\lambda,s)$ is contained in the domain where the limits \eqref{det_function} and \eqref{G_function} exist. This formula already suggests that the $\tau$-derivative of \eqref{det_from_F} should vanish. Of course, the point $(\lambda,s)=(0,0)$ typically fails to be contained in the domain where $G$ is defined by \eqref{G_function}.

Our final Assumption \ref{assumption_analytic_continuation}, and typically the most subtle to prove in specific cases, is then an analytic continuation condition on $G$, so that, after some manipulation stemming from Equation \eqref{e:F'intermsofG}, we can make sense of the quotient
\[
\frac{\sdet({D_\tau}|_{L_\tau})}{\sdet({D_0}|_{L_0})} = \exp\Big(-\frac{d}{ds}\big(F(\tau,\lambda,s) - F(0,\lambda,s)\big)\Big|_{\lambda=0,\,s=0}\Big) = 1.
\]

This discussion motivates the following main result.

\begin{maintheorem}
\label{thm_invariance}
Let $\tau \in (-1,1) \to {\delta_\tau}$ be a smooth family of regular general codifferentials (Definitions \ref{def_general_codifferential} and \ref{def_codifferential_types}) such that ${\delta_\tau}$ is an inner variation of ${\delta_0}$, as in Definition \ref{def_homotopy_structure}. Denote by $D_\tau = [\delta_\tau,\d]$ the associated family of characteristic operators and by $L_\tau = \im(\delta_\tau)$. Assume that
\begin{enumerate}[label=(A.\roman*)]
    \item \label{assumption_diff_semigroup} the family of strongly continuous semigroups
    $e^{-tD_\tau}: \Omega^\bullet(M,E) \to \Omega^\bullet(M,E)$
    is continuously differentiable with respect to $\tau$;\footnote{In the sense that for any $\alpha \in \Omega^k(M,E)$ and any $t\in\R_+$ the map $\tau \to e^{-tD_\tau}\alpha$ is differentiable with respect to the Fr\'echet topology on $\Omega^k(M,E)$ and the derivative $\frac{d}{d\tau}e^{-tD_\tau}\alpha$ is continuous in $(\tau,t)$.}
    \item \label{assumption_diff_hormander} for each $\tau_0\in(-1,1)$ there exists a closed conic set $\Gamma \subset T^*(\R_+\times M\times M)$ disjoint from\footnote{Recall that $N^*\iota=\{(t,0,x,\xi,x,-\xi)\in T^*(\R_+\times M\times M)\ |\ t\in\R_+, (x,\xi)\in T^*M\}$.} $N^*\iota$, such that $\tau \to K_\tau$ is differentiable in a neighborhood of $\tau_0$ with respect to the H\"ormander topology on $\D_\Gamma(\R_+\times M \times M)$, where $K_\tau$ is the Schwartz kernel of $\exp(-tD_\tau)$;
    \item \label{assumption_smaller_wavefront} for each $\tau$, the wavefront set of $K_\tau$ satisfies
    $$(t,0,x,\xi,y,-\eta)\in\WF(K_\tau) \implies (s,0,y,\eta,x,-\xi)\notin\WF(K_\tau),\quad \forall s\in\R_+;$$
    \item \label{assumption_G_convergencce} there is a domain
    $$Z = \{\Re(\lambda)>C_1, \Re(s)>C_2\} \subset \mathbb{C}^2,$$
    such that for a sequence of cutoff functions as in \eqref{cutoff_functions} and any smooth bounded function $f\in C^\infty_b(\R_+)$ the sequence
    $$\pair{\str\big(\theta_{\tau}e^{-t{D_\tau}}\big)}{t^{s-1}e^{-\lambda t}\chi_N(t)f(t)}$$
    converges locally uniformly in $(\tau,\lambda,s) \in (-1,1)\times Z$ as $N\to\infty$; in particular, the limit defining the function $G(\tau,\lambda,s)$ of Equation \ref{G_function} exists for each $\tau$ and each $(\lambda,s) \in Z$;
    \item \label{assumption_analytic_continuation} for all $\tau\in(-1,1)$ the function $G(\tau,\lambda,s)$ admits an analytic continuation to some connected domain $\tilde{Z}\subset\mathbb{C}^2$ with $(0,0),(0,1) \in \tilde{Z}$ and $\tilde{Z}\cap Z \neq \emptyset$; moreover, $G(\tau,\lambda,s)$ remains bounded on compact subsets of $(-1,1)\times \tilde{Z}$.
\end{enumerate}
Then, for all $\tau \in (-1,1)$, we have
\[
\sdet({D_\tau}|_{L_\tau}) = \sdet({D_0}|_{L_0}).
\]
\end{maintheorem}

\begin{rmk}[Generalization]
\label{rmk_generalization}
    The specific nature of $(\Omega^\bullet(M,E),\d)$ as the twisted de Rham complex is not important for the general theory. In fact, we can generalize Theorem \ref{thm_invariance} to the setting of a graded vector bundle $\V$ over $M$ equipped with a degree $1$ differential operator $d$, defining an acyclic complex
    \[\begin{tikzcd}
        \cdots \arrow[r] & C^\infty(M,\V^k) \arrow[r, "d"] & C^\infty(M,\V^{k+1}) \arrow[r] & \cdots
    \end{tikzcd}\]
    i.e. such that $d\circ d = 0$ and the cohomology groups of this complex are trivial. We can then define the notion of a general codifferential $\delta$ on $\V\to M$, that is a degree $-1$ differential operator satisfying $\delta\circ\delta=0$ and thus defining a complex
    \[\begin{tikzcd}
        \cdots & \arrow[l] C^\infty(M,\V^k) & \arrow[l, swap, "\delta"] C^\infty(M,\V^{k+1}) & \arrow[l] \cdots
    \end{tikzcd}\]
    The associated characteristic operator is the graded commutator $D=[\delta,d]=\delta d+d\delta$. We can now consider the flat superdeterminant $\sdet(D_\tau|_{\im(\delta_\tau)})$ along a smooth family $\tau\to\delta_\tau$ of such general codifferentials which forms an inner variation, in the sense that $\frac{d}{d\tau}\delta_\tau = [\theta_\tau,\delta_\tau]$ for some degree preserving differential operators $\theta_\tau$ on $\V\to M$. As in Definition \ref{def_codifferential_types}, we further require that the complexes defined by $\delta_\tau$ are acyclic and there are $D_\tau$-invariant splittings $C^\infty(M,\V) = \im(\delta_\tau)\oplus C_\tau$. Then, assuming that the flat superdeterminant is well-defined and nonvanishing for all $\tau$ and an appropriate version of Assumptions \ref{assumption_diff_semigroup}-\ref{assumption_analytic_continuation} hold, essentially the same proof as for Theorem \ref{thm_invariance} shows that
    $$\sdet(D_\tau|_{\im(\delta_\tau)}) = \sdet(D_0|_{\im(\delta_0)}), \quad \forall\,\tau.$$
    
    We chose to formulate Theorem \ref{thm_invariance} above in terms of the twisted de Rham differential, since the two examples studied in Sections \ref{sec:ATinvariance} and \ref{section_ruelle_independence} are obtained from $\d$ by an appropriate choice of general codifferential.
\end{rmk}

\begin{rmk}
    If one prefers to avoid working with semigroups on the Fr\'echet space of smooth sections, it is enough to assume that $\exp(-tD_\tau)$ defines a strongly continuous semigroup both on $L^2$ and on the Sobolev space $H^s$ for large enough $s$. One then requires the appropriate modification of Assumption \ref{assumption_diff_semigroup} (the regularity $s$ must be high enough for the right hand side in Equation \eqref{duhamel}, below, to make sense).
\end{rmk}
\begin{rmk}
    Requiring convergence in the presence of a bounded smooth function in Assumption \ref{assumption_G_convergencce} can be viewed as the analog of taking the absolute value, which is not available for distributions. Indeed, if $\str\bigl(\theta_{\tau}e^{-t{D_\tau}}\bigr)$ defines a continuous function of $t$, as opposed to just a distribution on $\R_+$, as is for instance the case when $D_\tau$ is elliptic, then Assumption \ref{assumption_G_convergencce} is equivalent to requiring locally uniform convergence of the integral 
    $\int_0^\infty\bigl|\str\bigl(\theta_{\tau}e^{-t{D_\tau}}\bigr)t^{s-1}e^{-\lambda t}\bigr|\,dt$
    for $(\tau,\lambda,s)\in (-1,1)\times Z$.
\end{rmk}
\begin{rmk}
\label{rmk_eliminating_lambda_s}
    If it happens that $\lambda=0$ is already in the domain of convergence for the limit defining $F(\tau,\lambda,s)$ in \eqref{det_function}, then no analytic continuation in $\lambda$ is necessary and one can evaluate all expressions directly at $\lambda=0$. Equation \eqref{e:F'intermsofG} then becomes $$\frac{d}{d\tau}F(\tau,0,s) = -\frac{s}{\Gamma(s)}G(\tau,0,s)$$ for $s$ in the domain of convergence. Thus, Assumption \ref{assumption_analytic_continuation} can be replaced with the analytic continuation of $G(\tau,0,s)$ to $s=0$. Similarly, if $s=0$ is in the domain of convergence of \eqref{det_function}. One can then directly evaluate the $s$-derivative appearing in the definition of the flat determinant, \eqref{det_from_F}, and \eqref{e:F'intermsofG} can be reformulated as\footnote{recall that the gamma function has a simple pole with residue $1$ at $s=0$} $$\frac{d}{d\tau}\bigl(\frac{d}{ds}F(\tau,\lambda,s=0)\bigr) = \lambda G(\tau,\lambda,1)$$ for $\lambda$ in the domain of convergence. So Assumption \ref{assumption_analytic_continuation} can be replaced in this case with the analytic continuation of $G(\tau,\lambda,1)$ to $\lambda=0$.
\end{rmk}

In Section \ref{sec:ATinvariance} and Section \ref{section_ruelle_independence} we show that the assumptions of Theorem \ref{thm_invariance} are valid in two very different but important cases: one where the flat super-determinant of the characteristic operator computes the analytic torsion for a manifold with a local system, and when it instead returns the value at zero of the Ruelle zeta function on a manifold that admits regular Anosov flows. Hence, we get a unified proof of their local constancy.

\subsubsection{Proof of Theorem \ref{thm_invariance}}

By assumption, ${\delta_\tau}$ are \emph{regular} general codifferentials (Definition \ref{def_codifferential_types}). In particular, the flat superdeterminant of the characteristic operator ${D_\tau} = [\delta_\tau,\d]$ restricted to the general codifferential subspace $L_\tau$ is well-defined for each $\tau$. The proof of the invariance of this superdeterminant will proceed by computing the $\tau$-derivative of the auxiliary function $F(\tau,\lambda,s)$ given in \eqref{det_function}. As a first step in this direction, we derive a formula for the $\tau$-derivative of the operator semigroup $\exp(-t{D_\tau})$. This is given by what is sometimes called Duhamel's formula, see \cite{berline_getzler} for the case where $D_\tau$ is a family of Laplacians. Here, we derive a corresponding formula for a more general family of characteristic operators under the hypothesis that $D_\tau$ generates a continuously differentiable family (w.r.t.\ $\tau$) of strongly continuous semigorups on the nuclear Fr\'echet space $\Omega^\bullet(M,E)$.

\begin{lemma}
\label{duhamel_lemma}
Let $\tau \in (-1,1) \to {D_\tau}$ be a smooth family of differential operators on $\C(M,V)$, for some vector bundle $V$ over $M$, generating a family of strongly continuous semigroups
$$e^{-t{D_\tau}}: \C(M,V) \to \C(M,V),$$
such that $\tau \to e^{-t{D_\tau}}f$ is differentiable for each $f \in \C(M,V), t\in \R_+$ and $\frac{d}{d\tau}e^{-t{D_\tau}}f \in \C(M,V)$ depends continuously on $(\tau,t)$.
Then the $\tau$-derivative is given by
\begin{equation}
\label{duhamel}
    \frac{d}{d\tau}e^{-t{D_\tau}} = -\int_0^t e^{-(t-u){D_\tau}}\Bigl(\frac{d}{d\tau}{D_\tau}\Bigr)e^{-u{D_\tau}}\,du.
\end{equation}
\end{lemma}
\begin{proof}
For each $f\in\C(M,V)$, by strong continuity of the semigroup, $e^{-tD_\tau}f$ depends continuously on $t\in\R_+$ in the Fr\'echet space topology of $\C(M,V)$. Thus, by the Fr\'echet space version of the uniform boundedness principle, for any $T>0$ the set of operators $\{e^{-tD_\tau} \,|\, t\in [0,T]\}$ is equicontinuous on $\C(M,V)$. Note also that the differential operator $\dot D_\tau = \frac{d}{d\tau}D_\tau$ is continuous on $\C(M,V)$. Using these two observations, one finds
that for any $f\in\C(M,V)$ and $t\in\R_+$:
$$u\in[0,t] \to e^{-(t-u){D_\tau}}\Bigl(\frac{d}{d\tau}{D_\tau}\Bigr)e^{-u{D_\tau}}f \in\C(M,V)$$
is continuous. Thus, the integral
$$\int_0^t e^{-(t-u){D_\tau}}\Bigl(\frac{d}{d\tau}{D_\tau}\Bigr)e^{-u{D_\tau}}f\,du$$
is well-defined as an element of $\C(M,V)$ and the right-hand side of \eqref{duhamel} defines a continuous operator on $\C(M,V)$, see for instance \cite{hamilton_nash-moser} for an overview of calculus in Fr\'echet spaces (see also \cite{kriegl_michor}). In particular, below we will make use of the fundamental theorems of calculus for integrals with values in Fr\'echet spaces, see \cite[Thm 2.2.2 and Thm 2.2.3]{hamilton_nash-moser}.

Since $D_\tau$ is the generator of the semigroup, there exists a dense subspace $S_\tau \subseteq \C(M,V)$, such that $t \to e^{-tD_\tau}f$ is differentiable for all $f\in S_\tau$ and 
\begin{equation}
\label{generator_semigroup}
\frac{d}{dt}e^{-tD_\tau}f = -D_\tau e^{-tD_\tau}f = -e^{-tD_\tau}D_\tau f, \quad \forall f \in S_\tau,
\end{equation}
see \cite{miyadera_semigroups} for the theory of semigroups on Fr\'echet spaces. Since the generator $D_\tau$ is itself continuous on $\C(M,V)$, equation \eqref{generator_semigroup} in fact holds on all of $\C(M,V)$. Indeed, for any small $h \in \R$ we have
$$\frac{1}{h}\Bigl(e^{-(t+h)D_\tau}f - e^{-tD_\tau}f\Bigr) = -\frac{1}{h}D_\tau\int_{t}^{t+h} e^{-uD_\tau}f\,du$$
on the dense subspace $S_\tau \subseteq \C(M,V)$. By equicontinuity, both sides of the above equation depend continuously on $f$, so the equation extends to $\C(M,V)$. Taking the limit $h\to 0$, we find $\frac{d}{dt} e^{-tD_\tau}f = -D_\tau e^{-tD_\tau}f$ for all $f \in \C(M,V)$.

Let now $f \in \C(M,V)$ be fixed. By assumption,
$$\frac{d}{d\tau}\frac{d}{dt}e^{-tD_\tau}f = -\frac{d}{d\tau}D_\tau e^{-tD_\tau}f = -\dot D_\tau e^{-tD_\tau}f - D_\tau\frac{d}{d\tau}e^{-tD_\tau}f$$
depends continuously on $(t,\tau)$ as an element of $\C(M,V)$. Thus, we can exchange the order of derivatives, and find
$$\Bigl(\frac{d}{dt} + D_\tau\Bigr)\frac{d}{d\tau}e^{-tD_\tau}f = -\dot D_\tau e^{-tD_\tau}f.$$
On the other hand,
$$\frac{d}{dt}e^{-(t-u){D_\tau}}\dot D_\tau e^{-u{D_\tau}}f = -D_\tau e^{-(t-u){D_\tau}}\dot D_\tau e^{-u{D_\tau}}f$$
depends continuously on $(t,u)$, so we find
$$\Bigl(\frac{d}{dt} + D_\tau\Bigr)\int_0^t e^{-(t-u){D_\tau}}\dot D_\tau e^{-u{D_\tau}}\,du = \dot D_\tau e^{-tD_\tau}f.$$
Thus, if we define
$$v(t) = \frac{d}{d\tau}e^{-t{D_\tau}}f + \int_0^t e^{-(t-u){D_\tau}}\Bigl(\frac{d}{d\tau}{D_\tau}\Bigr)e^{-u{D_\tau}}f\,du$$
then $t \in \R_+ \to v(t) \in \C(M,E)$ is continuously differentiable and satisfies
\begin{equation}
\label{IVP_semigroup}
\begin{split}
    \Bigl(\frac{d}{dt} + D_\tau\Bigr)v(t) &= 0, \quad \forall t\in\R_+,\\
    \lim_{t\to 0}v(t) &= 0.
\end{split}
\end{equation}
We claim that this implies $v(t)=0$ for all $t$. Indeed, such $v$ satisfies 
$$v(t)=\int_0^t\frac{d}{dr}v(r)\,dr = -D_\tau\int_0^t v(r)\,dr, \quad \forall t\in\R_+.$$
Now, for any $T>0$ the function $t \to e^{-(T-t)D_\tau}\int_0^t v(r)\,dr$ is continuously differentiable for $t \in (0,T)$ and satisfies
$$\frac{d}{dt}\Bigl(e^{-(T-t)D_\tau}\int_0^t v(r)\,dr\Bigr) = e^{-(T-t)D_\tau}v(t) - e^{-(T-t)D_\tau}D_\tau\int_0^t v(r)\,dr = 0.$$
Integrating this equality, we find $\int_0^T v(r)\,dr = 0$ for all $T>0$, i.e. $v=0$. Since $f \in \C(M,V)$ was arbitrary, this implies the lemma.
\end{proof}

Note that the characteristic operator $D= \delta\d + \d \delta$ determined by a general codifferential $\delta$ commutes with both $\delta$ and $\d$ since $\delta^2 = \d^2 = 0$. Whenever $D$ generates a strongly continuous semigroup on the Fr\'echet space $\Omega^\bullet(M,E)$, this commutativity carries over to the semigroup, that is
\begin{equation}
    [e^{-tD},\delta] = 0, \quad [e^{-tD},\d] = 0,
\end{equation}
as shown in the following lemma.

\begin{lemma}
\label{lemma_commutativity}
Let $D, P$ be commuting differential operators on some vector bundle $V$ over $M$ and assume $D$ is the generator of a strongly continuous semigroup
$$e^{-tD}: \C(M,V) \to \C(M,V).$$
Then $P$ commutes with $e^{-tD}$ for all $t\in\R_+$.
\end{lemma}
\begin{proof}
    As shown in the proof of the previous lemma, $t \to e^{-tD}f$ is differentiable for all $f\in\C(M,V)$ and $\frac{d}{dt}e^{-tD}f = -De^{-tD}f$. Since $P$ defines a continuous operator on $\C(M,V)$, the map $t\to Pe^{-tD}f$ is also differentiable and 
    $$\frac{d}{dt}Pe^{-tD}f = -PDe^{-tD}f = -DPe^{-tD}f$$
    by the commutativity of $P$ and $D$. Thus, for fixed $f\in\C(M,V)$,
    $$v(t) = Pe^{-tD}f - e^{-tD}Pf$$
    satisfies
    \begin{equation*}
    \begin{split}
        \Bigl(\frac{d}{dt} + D_\tau\Bigr)v(t) &= 0, \quad \forall t\in\R_+,\\
        \lim_{t\to 0}v(t) &= 0.
    \end{split}
    \end{equation*}
    As in the proof of the previous lemma, this implies $v(t)=0$ for all $t\in\R_+$.
\end{proof}

In order to compute the $\tau$-derivative of the expression in \eqref{det_function}, it is useful to rewrite the flat super trace $\str(\exp(-tD_\tau)|_{L_\tau})$ in a form that does not involve the explicit restriction to the general codifferential subspace. This reformulation follows from the algebraic properties of a regular general codifferential $\delta$ (Definition \ref{def_codifferential_types}) and the resulting characteristic operator $D$. By transversality, there is a splitting $\Omega^\bullet(M,E)=L\oplus C$ with $L=\im(\delta)$ which is left invariant by $D$.
We need to consider homogeneous differential forms in $\Omega^k(M,E)$ separately.
In the following, given an operator $A$ on $\Omega^\bullet(M,E)$, we denote by $A^{(k)}$ the restriction of $A$ on $E$-valued $k$-forms. In addition, we write 
\[
L^{(k)} \doteq L\cap\Omega^k(M,E) = \im(\delta^{(k+1)}), \qquad C^{(k)} \doteq C\cap\Omega^k(M,E).
\]
Thanks to Lemma \ref{lemma_commutativity}, $\delta$ commutes with the semigroup $\exp(-tD)$. Taking into account the form degrees, this commutativity becomes
$$e^{-tD^{(k)}}\delta^{(k+1)} = \delta^{(k+1)}e^{-tD^{(k+1)}}.$$
Thus, using the isomorphism ${\delta}: {C} \xrightarrow{\sim} {L}$ ensured by the acyclicity of the general codifferential, see Lemma \ref{inverse_Q_GF}, we have the following commutative diagram in terms of the subspaces $L$ and $C$:
\[\begin{tikzcd}[column sep = huge, row sep = large]
C^{(k+1)} \arrow[r, "\exp(-tD^{(k+1)})"] \arrow[d, "\rotatebox{90}{\(\sim\)}", "\delta^{(k+1)}"']
& C^{(k+1)} \arrow[d, "\delta^{(k+1)}", "\rotatebox{90}{\(\sim\)}"'] \\
L^{(k)} \arrow[r, "\exp(-tD^{(k)})"]
& L^{(k)}
\end{tikzcd}\]
This commutative diagram allows us to eliminate the explicit restriction to $L$ in the flat super trace appearing in the definition of $\sdet({D}|_{L})$.

\begin{lemma}
\label{restricted_trace}
Let $\delta$ be an acyclic, transversal general codifferential (Definition \ref{def_codifferential_types}) such that the restricted flat trace $\tr(e^{-tD^{(k)}}\vert_{L})$ is well-defined for each $k$. Then the flat super trace reads\footnote{Note that the flat trace in the last expression below is taken of an operator \emph{not} restricted to $L$.}
$$\str\bigl(e^{-tD}\big|_{L}\bigr) \doteq \sum_{k=0}^n(-1)^k\tr\bigl(e^{-tD^{(k)}}\big|_{L^{(k)}}\bigr) = \sum_{k=0}^n(-1)^{k+1}k\,\tr\bigl(e^{-tD^{(k)}}\bigr).$$
\end{lemma}
\begin{proof}
By well-definedness of the restricted flat trace, see \cite[Definition B.12]{Schiavina_Stucker2}, the projection operator $\Pi_{L^{(k)}}$ induced by the splitting $\Omega^k(M,E)=L^{(k)}\oplus C^{(k)}$ is continuous and does not increase the wavefront set, i.e. $\WF(\Pi_{L^{(k)}})\subset N^*\Delta_{M\times M}$ is contained in the conormal to the diagonal in $M\times M$. The restricted flat trace is given by $\tr(e^{-tD^{(k)}}|_{L^{(k)}}) = \tr(e^{-tD^{(k)}}\Pi_{L^{(k)}})$.

Let $(\delta^{(k+1)}\vert_C)^{-1}$ denote the inverse of the isomorphism $\delta^{(k+1)}:C^{(k+1)} \xrightarrow{\sim} L^{(k)}$, which exists thanks to acyclicity, see Lemma \ref{inverse_Q_GF}. By the above commutative diagram, we then have
\begin{equation*}
\begin{split}
    e^{-tD^{(k)}} \,\Pi_{L^{(k)}} &= e^{-tD^{(k)}} \circ \delta^{(k+1)} \circ \delta^{(k+1)}|_C^{-1} \,\Pi_{L^{(k)}} \\
    &=  \delta^{(k+1)} \circ e^{-tD^{(k+1)}} \circ \delta^{(k+1)}|_C^{-1} \,\Pi_{L^{(k)}}.
\end{split}
\end{equation*}
Taking the flat trace and using cyclicity of the flat trace (see \cite[Section 4.5]{viet_dang_dynamical_torsion}), we find
\begin{equation}
\label{restricted_trace_equality}
    \tr\big(e^{-tD^{(k)}}\,\Pi_{L^{(k)}}\big) = \tr\big(e^{-tD^{(k+1)}} \delta^{(k+1)}|_C^{-1}\,\Pi_{L^{(k)}} \delta^{(k+1)} \big) = \tr\big(e^{-tD^{(k+1)}} \,\Pi_{C^{(k+1)}} \big),
\end{equation}
where we used that
$$\delta^{(k+1)}|_C^{-1}\,\Pi_{L^{(k)}} \delta^{(k+1)} = \mathrm{id} - \Pi_{L^{(k+1)}} = \Pi_{C^{(k+1)}},$$ is just the projection operator with image $C^{(k+1)}$ and kernel $L^{(k+1)}=\im(\delta^{(k+2)})$.
Note that the properties of $\Pi_{L^{(k)}}$ imply that $\Pi_{C^{(k)}} = \mathrm{id} - \Pi_{L^{(k)}}$ is a continuous projection, whose wavefront set satisfies
$$\WF(\Pi_{C^{(k)}}) \subset \WF(\mathrm{id}) \cup \WF(\Pi_{L^{(k)}}) \subset N^*\Delta_{M\times M}.$$
Thus, the flat trace on the right hand side of \eqref{restricted_trace_equality} is well-defined, and the use of cyclicity is justified.
By the definition of flat trace restricted to an invariant subspace we can write \eqref{restricted_trace_equality} simply as
$$\tr(e^{-tD^{(k)}}\big|_{L^{(k)}}) = \tr(e^{-tD^{(k+1)}}\big|_{C^{(k+1)}}).$$
Using this equality in the definition of the flat super trace, we find
\begin{equation*}
\begin{split}
    \str(e^{-tD}\big|_{L}) &= \sum_{k=0}^n(-1)^k\Bigl((k+1)\,\tr\bigl(e^{-tD^{(k)}}\big|_{L^{(k)}}\bigr) - k\,\tr\bigl(e^{-tD^{(k)}}\big|_{L^{(k)}}\bigr)\Bigr) \\
    &= \sum_{k=0}^n(-1)^k\Bigl((k+1)\,\tr\bigl(e^{-tD^{(k+1)}}\big|_{C^{(k+1)}}\bigr) - k\,\tr\bigl(e^{-tD^{(k)}}\big|_{L^{(k)}}\bigr)\Bigr) \\
    &= \sum_{k=0}^n(-1)^{k+1}\Bigl(k\,\tr\bigl(e^{-tD^{(k)}}\big|_{C^{(k)}}\bigr) + k\,\tr\bigl(e^{-tD^{(k)}}\big|_{L^{(k)}}\bigr)\Bigr) \\ 
    &= \sum_{k=0}^n(-1)^{k+1}k\,\tr\bigl(e^{-tD^{(k)}}\bigr),
\end{split}
\end{equation*}
where, in the second-to-last line, we simply reshuffled the terms in the sum.
\end{proof}

We have not yet made use of the special form of our family of general codifferentials. Using the same notation as above to denote the restriction of an operator to $E$-valued $k$-forms, the inner variation structure of Definition \ref{def_homotopy_structure} can be written
\begin{equation}
\label{homotopy_derivative_degrees}
    \frac{d}{d\tau}\delta_\tau^{(k)} = \theta_{\tau}^{(k-1)}\delta_\tau^{(k)} - \delta_\tau^{(k)}\theta_{\tau}^{(k)}.
\end{equation}
This algebraic structure will allow us to make a key simplification of a certain expression involving the $\tau$-derivative of the characteristic operator $D^{(k)}_{\tau}$.

\begin{lemma}
\label{trace_simplification_lemma}
For a smooth family of regular general codifferentials $\delta_\tau$ given by an inner variation as in Definition \ref{def_homotopy_structure}, the following holds:
\begin{equation}
\label{trace_simplification}
    \sum_{k=0}^n(-1)^k k\,\tr\Bigl(\Bigl(\frac{d}{d\tau}D^{(k)}_{\tau}\Bigr)e^{-tD^{(k)}_{\tau}}\Bigr) = \sum_{k=0}^n(-1)^k\, \tr\Bigl(\theta^{(k)}_\tau\frac{d}{dt}e^{-tD^{(k)}_{\tau}}\Bigr)
\end{equation}
\end{lemma}
\begin{proof}
Taking the $\tau$-derivative of $D^{(k)}_{\tau} = \delta_\tau^{(k+1)}\d^{(k)} + \d^{(k-1)}\delta_\tau^{(k)}$ and using equation \eqref{homotopy_derivative_degrees} results in
\begin{equation}
\label{trace_simplification1}
\begin{split}
    \frac{d}{d\tau}D^{(k)}_{\tau} &= \frac{d}{d\tau}\delta_\tau^{(k+1)}\d^{(k)} + \d^{(k-1)}\frac{d}{d\tau}\delta_\tau^{(k)} \\
    &= \theta^{(k)}_\tau \delta_\tau^{(k+1)}\d^{(k)} - \delta_\tau^{(k+1)}\theta^{(k+1)}_\tau \d^{(k)} + \d^{(k-1)}\theta^{(k-1)}_\tau \delta_\tau^{(k)} - \d^{(k-1)}\delta_\tau^{(k)}\theta^{(k)}_\tau.
    \end{split}
\end{equation}
Note that $\frac{d}{d\tau}D_\tau^{(k)}$ is a differential operator, so composing with $\exp(-tD_\tau^{(k)})$ does not increase the wavefront set and the flat trace in \eqref{trace_simplification} is well-defined.
Inserting equation \eqref{trace_simplification1} into the flat trace on the left hand side of \eqref{trace_simplification}, gives
\begin{equation}
\begin{split}
\label{trace_simplification2}
    &\sum_{k=0}^n(-1)^k k\,\tr\Bigl(\Bigl(\frac{d}{d\tau}D^{(k)}_{\tau}\Bigr)e^{-tD^{(k)}_{\tau}}\Bigr) \\
    &= \sum_{k=0}^n(-1)^k k\, \tr\Bigl(\theta^{(k)}_\tau \delta_\tau^{(k+1)}\d^{(k)}e^{-tD^{(k)}_{\tau}}\Bigr) - \sum_{k=0}^n(-1)^k k\, \tr\Bigl(\delta_\tau^{(k+1)}\theta^{(k+1)}_\tau \d^{(k)}e^{-tD^{(k)}_{\tau}}\Bigr) \\
    &+ \sum_{k=0}^n(-1)^k k\, \tr\Bigl(\d^{(k-1)}\theta^{(k-1)}_\tau \delta_\tau^{(k)}e^{-tD^{(k)}_{\tau}}\Bigr) - \sum_{k=0}^n(-1)^k k\, \tr\Bigl(\d^{(k-1)}\delta_\tau^{(k)}\theta^{(k)}_\tau e^{-tD^{(k)}_{\tau}}\Bigr).
\end{split}
\end{equation}
By Lemma \ref{lemma_commutativity}, $\exp(-tD_\tau)$ commutes with both ${\delta_\tau}$ and $\d$. Writing out the form degree, this becomes
$$e^{-tD_\tau^{(k)}}\delta_\tau^{(k+1)} = \delta_\tau^{(k+1)}e^{-tD_\tau^{(k+1)}}, \qquad e^{-tD_\tau^{(k)}}\d^{(k-1)} = \d^{(k-1)}e^{-tD_\tau^{(k-1)}}.$$
Note the shift in degree for the operator $\exp(-tD_\tau)$. Using this commutativity and the cyclicity of the flat trace \cite[Section 4.5]{viet_dang_dynamical_torsion}, the last three terms in \eqref{trace_simplification2} become
\begin{equation*}
\begin{split}
    \tr\Bigl(\delta_\tau^{(k+1)}\theta^{(k+1)}_\tau \d^{(k)}e^{-tD^{(k)}_{\tau}}\Bigr) &= \tr\Bigl(\theta^{(k+1)}_\tau \d^{(k)}e^{-tD^{(k)}_{\tau}}\delta_\tau^{(k+1)}\Bigr) = \tr\Bigl(\theta^{(k+1)}_\tau \d^{(k)}\delta_\tau^{(k+1)}e^{-tD^{(k+1)}_{\tau}}\Bigr) \\
    \tr\Bigl(\d^{(k-1)}\theta^{(k-1)}_\tau \delta_\tau^{(k)}e^{-tD^{(k)}_{\tau}}\Bigr) &= \tr\Bigl(\theta^{(k-1)}_\tau \delta_\tau^{(k)}e^{-tD^{(k)}_{\tau}}\d^{(k-1)}\Bigr) = \tr\Bigl(\theta^{(k-1)}_\tau \delta_\tau^{(k)}\d^{(k-1)}e^{-tD^{(k-1)}_{\tau}}\Bigr) \\
    \tr\Bigl(\d^{(k-1)}\delta_\tau^{(k)}\theta^{(k)}_\tau e^{-tD^{(k)}_{\tau}}\Bigr) &= \tr\Bigl(\theta^{(k)}_\tau e^{-tD^{(k)}_{\tau}}\d^{(k-1)}\delta_\tau^{(k)}\Bigr) = \tr\Bigl(\theta^{(k)}_\tau \d^{(k-1)}\delta_\tau^{(k)}e^{-tD^{(k)}_{\tau}}\Bigr).
\end{split}
\end{equation*}
Inserting this in \eqref{trace_simplification2}, we find
\begin{equation*}
\begin{split}
    &\sum_{k=0}^n(-1)^k k\,\tr\Bigl(\Bigl(\frac{d}{d\tau}D^{(k)}_{\tau}\Bigr)e^{-tD^{(k)}_{\tau}}\Bigr) \\
    &= \sum_{k=0}^n(-1)^k k\,\tr\Bigl(\theta^{(k)}_\tau \delta_\tau^{(k+1)}\d^{(k)}e^{-tD^{(k)}_{\tau}}\Bigr) - \sum_{k=0}^n(-1)^k k\,\tr\Bigl(\theta^{(k+1)}_\tau \d^{(k)}\delta_\tau^{(k+1)}e^{-tD^{(k+1)}_{\tau}}\Bigr) \\
    &+ \sum_{k=0}^n(-1)^k k\,\tr\Bigl(\theta^{(k-1)}_\tau \delta_\tau^{(k)}\d^{(k-1)}e^{-tD^{(k-1)}_{\tau}}\Bigr) - \sum_{k=0}^n(-1)^k k\,\tr\Bigl(\theta^{(k)}_\tau \d^{(k-1)}\delta_\tau^{(k)}e^{-tD^{(k)}_{\tau}}\Bigr) \\
    &= \sum_{k=0}^n(-1)^{k+1}\bigl((k+1)-k\bigr)\tr\Bigl(\theta^{(k)}_\tau \delta_\tau^{(k+1)}\d^{(k)}e^{-tD^{(k)}_{\tau}}\Bigr) \\
    &+ \sum_{k=0}^n(-1)^{k+1}\bigl(k-(k-1)\bigr)\tr\Bigl(\theta^{(k)}_\tau \d^{(k-1)}\delta_\tau^{(k)}e^{-tD^{(k)}_{\tau}}\Bigr) \\
    &= \sum_{k=0}^n(-1)^{k+1}\tr\Bigl(\theta^{(k)}_\tau D^{(k)}e^{-tD^{(k)}_{\tau}}\Bigr) = \sum_{k=0}^n(-1)^{k}\tr\Bigl(\theta^{(k)}_\tau \frac{d}{dt}e^{-tD^{(k)}_{\tau}}\Bigr),
\end{split}
\end{equation*}
where, in the second equality, we relabeled the terms in the second and third sums and then combined the terms in the first and third sums and in the second and fourth sums respectively.
\end{proof}

We now study the $\tau$-derivative of the distribution $\str\bigl(e^{-t{D_\tau}}\big|_{L_\tau}\bigr) \in \D(\R_+)$. We would like to pass the $\tau$-derivative inside the flat trace so we can use Duhamel's formula (Lemma \ref{duhamel_lemma}). This is possible thanks to Assumption \ref{assumption_diff_hormander}. We can then use the previous results, in particular Lemma \ref{trace_simplification_lemma}, to obtain a simple expression for the $\tau$-derivative of the flat super trace when the family of general codifferentials has the structure specified in Definition \ref{def_homotopy_structure}. We remark that Assumption \ref{assumption_smaller_wavefront} is used to commute the flat trace with the integral appearing in \eqref{duhamel}.

\begin{lemma}
\label{lemma_tau_to_t}
For a smooth inner variation of regular general codifferentials ${\delta_\tau}$ (Definition \ref{def_homotopy_structure}) satisfying Assumptions \ref{assumption_diff_semigroup}, \ref{assumption_diff_hormander} and \ref{assumption_smaller_wavefront}, we have
\begin{equation}
\label{derivative_of_trace}
    \frac{d}{d\tau}\pair{\str\Bigl(e^{-t{D_\tau}}\big|_{L_\tau}\Bigr)}{\chi} = - \pair{\str\Bigl(\theta_\tau e^{-t{D_\tau}}\Bigr)}{\frac{d}{dt}\bigl(t\chi\bigr)},
\end{equation}
for any $\chi \in \C_c(\R_+)$, where $\theta_\tau$ denotes the operator from Definition \ref{def_homotopy_structure}.
\end{lemma}

\begin{proof}
We will first use Lemma \ref{restricted_trace} to rewrite the left hand side of \eqref{derivative_of_trace} without the explicit restriction to $L_\tau$. Then, by Assumption \ref{assumption_diff_hormander} the map $\tau \to e^{-tD_\tau}$ is differentiable with respect to the H\"ormander topology on $\D_\Gamma(\R_+\times M^2)$, where we identify $e^{-tD_\tau}$ with its Schwartz kernel $K_\tau$ viewed as a distribution on $\R_+\times M^2$. That is, the difference quotients $\frac{1}{h}(\exp(-tD_{\tau+h} - \exp(-tD_\tau))$ converge in $\D_\Gamma(\R_+\times M^2)$ as $h\to 0$. Here, $\Gamma$ is a closed conic set disjoint from the conormal $N^*\iota$ to the map $\iota: (t,x) \in \R_+\times M \to (t,x,x) \in \R_+\times M^2$. Since the flat trace defines a sequentially continuous linear functional $\tr:\D_\Gamma(\R_+\times M^2) \to \D(\R_+)$ (see \cite[Proposition B.3]{Schiavina_Stucker2}), Assumption \ref{assumption_diff_hormander} allows us to take the $\tau$-derivative inside the flat trace. Indeed, using Lemma \ref{restricted_trace}, we have
\begin{equation}
\label{trace_manipulation}
\begin{split}
    \frac{d}{d\tau} \pair{\str\left(e^{-t{D_\tau}}\big|_{L_\tau}\right)}{\chi} &= \sum_{k=0}^n (-1)^{k+1}k\, \frac{d}{d\tau}\pair{\tr(e^{-tD^{(k)}_{\tau}})}{\chi} \\
    &= \sum_{k=0}^n (-1)^{k+1}k\, \lim_{h\to 0} \frac{1}{h} \pair{\tr\Bigl(e^{-tD^{(k)}(\tau+h)} - e^{-tD^{(k)}_{\tau}}\Bigr)}{\chi} \\
    &= \sum_{k=0}^n (-1)^{k+1}k\, \pair{\tr\Bigl(\lim_{h\to 0}\frac{1}{h} \big(e^{-tD^{(k)}(\tau+h)} - e^{-tD^{(k)}_{\tau}}\big)\Bigr)}{\chi} \\
    &= \sum_{k=0}^n (-1)^{k+1}k\, \pair{\tr\Bigl( \frac{d}{d\tau}e^{-tD^{(k)}_{\tau}}\Bigr)}{\chi}.
\end{split}
\end{equation}

Using the formula for the derivative of a family of semigroups from Lemma \ref{duhamel_lemma}, we have
\begin{equation}
\label{duhamel_degrees}
    \frac{d}{d\tau}e^{-t{D^{(k)}_\tau}} = -\int_0^t e^{-(t-u){D^{(k)}_\tau}}\left(\frac{d}{d\tau}{D^{(k)}_\tau}\right)e^{-u{D^{(k)}_\tau}}\,du
\end{equation}
We will now make use of Assumption \ref{assumption_smaller_wavefront} to exchange the flat trace and the integral over $u$. To this end, we denote by $\Phi_\tau$ the Schwartz kernel of the operator
\begin{equation}
\label{Phi_def}
    e^{-(t-u){D_\tau}}\left(\frac{d}{d\tau}{D_\tau}\right)e^{-u{D_\tau}}.
\end{equation}
We can view $\Phi_\tau$ as a distribution in\footnote{Here and in what follows, we suppress the bundle in which $\Phi_\tau$ takes its values from our notation since it is not pertinent for wavefront considerations. Actually, $\Phi_\tau^{(k)} \in \D\bigl(\Omega\times M^2;\pi_{M^2}^*((\wedge^kT^*M\otimes E)\boxtimes(\wedge^kT^*M\otimes E)^*\otimes\wedge^nT^*M)\bigr)$, where $\pi_{M^2}$ is the projection from $\Omega\times M^2$ onto $M^2$.} $\D(\Omega\times M^2)$, where $\Omega = \{(t,u) \in \R \,|\, 0<u<t\}$. Note that the operator in \eqref{Phi_def} is strongly continuous on $\Omega^\bullet(M,E)$ with respect to $(t,u)$ in the closure $\overline{\Omega}$. Thus, the pushforward $\pi_*\Phi_\tau \in \D(\R_+\times M^2)$ by the projection map
$$\pi: (t,u,x,y) \in \Omega\times M^2 \to (t,x,y) \in \R_+\times M^2$$
is well-defined and formula \eqref{duhamel_degrees} can be written in terms of Schwartz kernels as
\begin{equation}
\label{Phi_K_relation}
    \frac{d}{d\tau}K_\tau^{(k)} = - \pi_* \Phi_\tau^{(k)}.
\end{equation}
Thus, by definition of the flat trace, we have
\begin{equation}
\label{pullback_pushforward}
    \pair{\tr\Bigl( \frac{d}{d\tau}e^{-tD^{(k)}_{\tau}}\Bigr)}{\chi} = \pair{\mathrm{Tr}\Bigl(\iota^*\frac{d}{d\tau}K_\tau^{(k)}\Bigr)}{\chi\otimes 1} = -\pair{\mathrm{Tr}\Bigl(\iota^*\pi_*\Phi_\tau^{(k)}\Bigr)}{\chi\otimes 1},
\end{equation}
where $\mathrm{Tr}$ denotes the fiberwise trace in the bundle $(\wedge^kT^*M\otimes E)\otimes(\wedge^kT^*M\otimes E)^*$ and $\iota^*$ is the pullback by the diagonal map
$$\iota: (t,x) \in \R_+\times M \to (t,x,x) \in \R_+\times M^2.$$

We now show that we can in some sense switch the order of the pullback and the pushforward in the above expression. This will be achieved by controlling the wavefront set of $\Phi_\tau$. To this end, we express $\Phi_\tau$ in terms of $K_\tau$ using fundamental operations on distributions, that is the pullback, pushforward and tensor product. Defining the following maps:
\begin{align*}
    &\pi_2: (t,u,x,y) \in \Omega\times M^2 \to (u,x,y) \in \R_+\times M^2, \\
    &f: (t,u,x,y) \in \Omega\times M^2 \to (t-u,x,y) \in \R_+\times M^2, \\
    &g: (t,u,x,z,y) \in \Omega\times M^3 \to (t,u,x,z,t,u,z,y) \in (\Omega\times M^2)^2, \\
    &h: (t,u,x,z,y) \in \Omega\times M^3 \to (t,u,x,y) \in \Omega\times M^2,
\end{align*}
we see from \eqref{Phi_def} that $\Phi_\tau$ can be written as
\begin{equation}
\label{Phi_through_K}
    \Phi_\tau = h_*g^*(f^*K_\tau \otimes \dot{D}_\tau \pi_2^*K_\tau).
\end{equation}
Using the wavefront set mapping properties of these fundamental operations, see for instance \cite{brouder}, we find that 
$$\WF(\Phi_\tau) \subset \Lambda_\tau,$$
where $\Lambda_\tau \subset T^*(\Omega\times M\times M)$ is the conic set
\begin{equation}
\label{wavefront_duhamel}
\begin{split}
    \Lambda_\tau = &\bigl\{(t,\theta,u,\nu,x,\xi,y,\eta) \,|\, \exists\, (z,\omega) \in T^*M \,\,\mathrm{s.t.}\,\, (t-u,\theta,x,\xi,z,-\omega)\in\WF(K_\tau),\\ 
    &\hspace{5.55cm} \mathrm{and}\,\, (u,\theta+\nu,z,\omega,y,\eta)\in\WF(K_\tau)\bigr\} \\
    \cup\,&\bigl\{(t,\theta,u,-\theta,x,\xi,y,0) \,|\, \exists z \in M \,\,\mathrm{s.t.}\,\, (t-u,\theta,x,\xi,z,0)\in\WF(K_\tau)\bigr\} \\
    \cup\,&\bigl\{(t,0,u,\theta,x,0,y,\eta) \,|\, \exists z \in M \,\,\mathrm{s.t.}\,\, (u,\theta,z,0,y,\eta)\in\WF(K_\tau)\bigr\}.
\end{split}
\end{equation}

Assumption \ref{assumption_smaller_wavefront} ensures that for each $\tau$, we have $\Lambda_\tau$ disjoint from
$$N^*\tilde{\iota} = \bigl\{(t,0,u,0,x,\xi,x,-\xi) \,|\, (t,u)\in\Omega, (x,\xi)\in T^*M\bigr\},$$
the conormal bundle to the diagonal map
$$\tilde{\iota}: (t,u,x) \in \Omega\times M \to (t,u,x,x) \in \Omega\times M^2.$$
Thus, the pullback $\tilde{\iota}^*\Phi_\tau$ is well-defined for each $\tau$. We further define the projection map
$$\tilde{\pi}: (t,u,x) \in \Omega\times M \to (t,x) \in \R_+ \times M.$$
Now, on the space of compactly supported smooth functions, $f\in \C_c(\Omega\times M^2)$, we have
$$\iota^*\pi_*f = \int_0^t f(t,u,x,x)\,du = \tilde{\pi}_*\tilde{\iota}^*f.$$
Since $\C_c(\Omega\times M^2)$ is dense in $\D_{\Lambda_\tau}(\Omega\times M^2)$ and both operations $\iota^*\circ\pi_*$ and $\tilde{\pi}_*\circ\tilde{\iota}^*$ are sequentially continuous with respect to the H\"ormander topology, the commutativity $\iota^*\circ\pi_* = \tilde{\pi}_*\circ\tilde{\iota}^*$ extends to $\D_{\Lambda_\tau}(\Omega\times M^2)$ for each $\tau$. In particular, we have
$$\iota^*\pi_*\Phi_\tau^{(k)} = \tilde{\pi}_*\tilde{\iota}^*\Phi_\tau^{(k)}.$$
Applying this in \eqref{pullback_pushforward}, recalling the definition of $\Phi_\tau^{(k)}$ and using the cyclicity of the flat trace (see \cite[Section 4.5]{viet_dang_dynamical_torsion}), we find
\begin{equation*}
\begin{split}
    \pair{\tr\Bigl( \frac{d}{d\tau}e^{-tD^{(k)}_{\tau}}\Bigr)}{\chi} &= -\pair{\mathrm{Tr}\bigl(\tilde{\pi}_*\tilde{\iota}^*\Phi_\tau^{(k)}\Bigr)}{\chi\otimes 1} = -\pair{\int_0^t\tr\Bigl(e^{-(t-u)D^{(k)}_{\tau}}\big(\frac{d}{d\tau}D^{(k)}_{\tau}\big)e^{-uD^{(k)}_{\tau}}\Bigr)\,du}{\chi} \\
    &= -\pair{\int_0^t\tr\Bigl(\Bigl(\frac{d}{d\tau}D^{(k)}_{\tau}\Bigr)e^{-tD^{(k)}_{\tau}}\Bigr)\,du}{\chi} = -\pair{t\,\tr\Bigl(\Bigl(\frac{d}{d\tau}D^{(k)}_{\tau}\Bigr)e^{-tD^{(k)}_{\tau}}\Bigr)}{\chi}.
\end{split}
\end{equation*}

Inserting the above expression into \eqref{trace_manipulation} and making use of Lemma \ref{trace_simplification_lemma}, we arrive at \eqref{derivative_of_trace}:
\begin{equation}
\begin{split}
    \frac{d}{d\tau} \pair{\str\left(e^{-t{D_\tau}}\big|_{L_\tau}\right)}{\chi} &= \sum_{k=0}^n (-1)^kk\,\pair{\tr\Bigl(\Bigl(\frac{d}{d\tau}D^{(k)}_{\tau}\Bigr)e^{-tD^{(k)}_{\tau}}\Bigr)}{t\chi} \\
    &= \sum_{k=0}^n (-1)^k\,\pair{\tr\Bigl(\theta_{\tau}\frac{d}{dt}e^{-tD^{(k)}_{\tau}}\Bigr)}{t\chi} \\
    &= \sum_{k=0}^n (-1)^{k+1}\,\pair{\tr\Bigl(\theta_{\tau}e^{-tD^{(k)}_{\tau}}\Bigr)}{\frac{d}{dt}(t\chi)} \\
    &= -\pair{\str\bigl(\theta_{\tau}e^{-t{D_\tau}}\bigr)}{\frac{d}{dt}(t\chi)}.
\end{split}
\end{equation}
Here, we moved the factor of $t$ multiplying the distribution on $\R_+$ over to the compactly supported smooth function $\chi$ and applied Lemma \ref{trace_simplification_lemma}. In the third equality we moved the $t$ derivative out of the flat trace and performed a partial integration, in the distributional sense. Exchanging the order of $t$ derivative and flat trace is justified by the sequential continuity of $\frac{d}{dt}$ with respect to the H\"ormander topology on $\D_\Gamma(\R_+\times M\times M)$. More precisely, $\frac{d}{dt}$ and $\str$ commute for smooth Schwartz kernels, where the flat trace is just integration over the diagonal in $M\times M$. Approximating a distributional Schwartz kernel by smooth functions and using sequentially continuity of $\str$ and $\frac{d}{dt}$ shows that the commutativity remains true in this situation.
\end{proof}

\begin{rmk}
The proof of the preceding lemma shows that we could replace assumptions \ref{assumption_diff_hormander} and \ref{assumption_smaller_wavefront} by the following single assumption:
\begin{itemize}
    \item for each $\tau_0\in(-1,1)$ there exists a closed conic set $\Gamma \subset T^*(\R_+\times M\times M)$ satisfying
    \begin{equation}
    \label{smaller_WF_property}
        (t,0,x,\xi,y,-\eta)\in\Gamma \implies (s,0,y,\eta,x,-\xi)\notin\Gamma, \quad \forall s\in\R_+,
    \end{equation}
    such that $\WF(K_\tau)\subset\Gamma$ and $\tau \to K_\tau$ is continuous in a neighborhood of $\tau_0$ with respect to the H\"ormander topology on $\D_\Gamma(\R_+\times M \times M)$.
\end{itemize}
Note that \eqref{smaller_WF_property} in particular implies $\Gamma\cap N^*\iota=\emptyset$, i.e. this is a stricter condition on $\WF(K_\tau)$ than merely requiring $\WF(K_\tau)\cap N^*\iota=\emptyset$. In other words, we can replace the condition of differentiability in the larger space of distributions from Assumption \ref{assumption_diff_hormander} by merely requiring continuity, but in the smaller space $\D_\Gamma(\R_+\times M\times M)$ with $\Gamma$ as above.
Indeed, by \eqref{Phi_K_relation}, we can write the difference quotient as
\begin{equation}
\label{diff_quotient_as_integral}
    \frac{1}{h}(K_{\tau+h} - K_\tau) = - \frac{1}{h}\int_{\tau}^{\tau+h} \pi_*\Phi_\sigma\,d\sigma.
\end{equation}
The proof of lemma \ref{lemma_tau_to_t}, in particular \eqref{wavefront_duhamel}, shows that $\WF(\pi_*\Phi_\sigma) \subset \pi_*\Lambda$, where
\begin{align*}
    \pi_*\Lambda = &\{(t,\theta,x,\xi,y,\eta) \,|\, \exists\, u<t, (z,\omega) \in T^*M \,\,\mathrm{s.t.}\,\, (t-u,\theta,x,\xi,z,-\omega)\in\Gamma,\, (u,\theta,z,\omega,y,\eta)\in\Gamma\} \\
    \cup&\{(t,\theta,x,\xi,y,0) \,|\, \exists u<t, z \in M \,\,\mathrm{s.t.}\,\, (t-u,\theta,x,\xi,z,0)\in\Gamma\} \\
    \cup&\{(t,\theta,x,0,y,\eta) \,|\, \exists u<t, z \in M \,\,\mathrm{s.t.}\,\, (u,\theta,z,0,y,\eta)\in\Gamma\}.
\end{align*}
Property \eqref{smaller_WF_property} of $\Gamma$ ensures that $\pi_*\Phi_\sigma$ is disjoint from $N^*\iota$ for $\sigma$ in a neighborhood of $\tau_0$. Moreover, since $\Phi_\sigma$ is built from $K_\sigma$ using pullback, pushforward and tensor product, see \eqref{Phi_through_K}, the continuity of these fundamental operations with respect to the H\"ormander topology, see \cite{brouder}, together with the assumed continuity of $\sigma\to K_\sigma \in \D_\Gamma(\R_+\times M\times M)$, implies that $\sigma \to \pi_*\Phi_\sigma$ is continuous with respect to the H\"ormander topology on $\D_{\pi_*\Lambda}(\R_+\times M\times M)$. In particular, $\pi_*\Phi_\sigma$ is bounded in the Hörmander seminorms for the conic set $\pi_*\Lambda$. It follows from \eqref{diff_quotient_as_integral} that for each Hörmander seminorm $\|\cdot\|_{N,\phi,V}$, we have
$$\|\tfrac{1}{h}(K_{\tau+h} - K_\tau)\|_{N,\phi,V} \le \sup_{\sigma\in[\tau,\tau+h]}\|\pi_*\Phi_\sigma\|_{N,\phi,V} < \infty,$$
Thus, as $h\to 0$ the difference quotient converges in the H\"ormander topology on $\D_{\pi_*\Lambda}(\R_+\times M\times M)$, see \cite[Definition 8.2.2]{hormander1}, establishing Assumption \ref{assumption_diff_hormander} for the conic set $\pi_*\Lambda$.
\end{rmk}

We now turn to the study of the function $F(\tau,\lambda,s)$ used in the definition of the flat superdeterminant, see \eqref{det_from_F}. We show that under Assumption \ref{assumption_G_convergencce} and for large enough $\Re(\lambda)$, $\Re(s)$ its derivative can be expressed in terms of the function $G(\tau,\lambda, s)$ defined in \eqref{G_function}.

\begin{lemma}\label{lem:F'intermsofG}
For a smooth inner variation of regular general codifferentials ${\delta_\tau}$ satisfying Assumptions \ref{assumption_diff_semigroup}, \ref{assumption_diff_hormander}, \ref{assumption_smaller_wavefront} and \ref{assumption_G_convergencce}, we have
\begin{equation}
\label{F_derivative}
    \frac{d}{d\tau}F(\tau,\lambda,s) = \lambda\frac{1}{\Gamma(s)}G(\tau,\lambda,s+1) - \frac{s}{\Gamma(s)}G(\tau,\lambda,s),
\end{equation}
for all $\tau\in(-1,1)$ and $(\lambda,s) \in Z$, where $Z \subset \mathbb{C}^2$ is the domain from Assumption \ref{assumption_G_convergencce}. Moreover, the right hand side is analytic in $(\lambda,s) \in Z$ and continuous in $\tau$.
\end{lemma}
\begin{proof}
Recall that $F$ is given as the limit
$$F(\tau,\lambda,s) = \frac{1}{\Gamma(s)}\lim_{N\to\infty}\pair{\str\left(e^{-t{D_\tau}}\big|_{L_\tau}\right)}{t^{s-1}e^{-t\lambda}\chi_N(t)},$$
where the cutoff functions $\chi_N \in \C_c(\R_+)$ satisfy $0\leq\chi_N\leq 1$ and $\chi_N(t)=1$ for $t\in[\frac{1}{N},N]$. In addition, we can choose the $\chi_N$ such that
$$|\dot{\chi}_N(t)| \le C, \quad \mathrm{for}\,\, t>N, \qquad |\dot{\chi}_N(t)| \le CN, \quad \mathrm{for}\,\, t<\frac{1}{N},$$
for some constant $C$. By Lemma \ref{lemma_tau_to_t} we have
\begin{equation}
\label{trace_pre_limit}
\begin{split}
    \frac{d}{d\tau}\pair{\str\left(e^{-t{D_\tau}}\big|_{L_\tau}\right)}{t^{s-1}e^{-t\lambda}\chi_N(t)} = &- \pair{\str\left(\theta_\tau e^{-t{D_\tau}}\right)}{\frac{d}{dt}\Bigl(t^s e^{-t\lambda}\chi_N(t)\Bigr)} \\
    = &- s \pair{\str\left(\theta_\tau e^{-t{D_\tau}}\right)}{t^{s-1} e^{-t\lambda}\chi_N(t)} \\
    &+ \lambda \pair{\str\left(\theta_\tau e^{-t{D_\tau}}\right)}{t^s e^{-t\lambda}\chi_N(t)} \\
    &- \pair{\str\left(\theta_\tau e^{-t{D_\tau}}\right)}{t^s e^{-t\lambda}\dot{\chi}_N(t)}
\end{split}
\end{equation}

Thanks to Assumption \ref{assumption_G_convergencce}, the last term in equation \eqref{trace_pre_limit} vanishes in the limit $N\to\infty$ when $(\lambda,s)\in Z$. Indeed, the derivative of the cutoff function $\chi_N$ is identically zero in $t\in[\frac{1}{N},N]$. By our choice of cutoff functions above, we further have
$$\bigl|t^se^{-\lambda t}\dot{\chi}_N(t)\bigr| \le CN\bigl|t^{s} e^{-\lambda t}\bigl| \le C \bigl|t^{s-1} e^{-\lambda t}\bigl| \le \frac{C}{N^\epsilon}\bigl|t^{s-1-\epsilon}e^{-\lambda t}\bigl| \quad\text{for}\quad t\in\bigl(0,\tfrac{1}{N}\bigr),$$
and
$$\bigl|t^se^{-\lambda t}\dot{\chi}_N(t)\bigl| \le C e^{-\epsilon N} \bigl|t^s e^{-(\lambda-\epsilon)t}\bigl| \quad\text{for}\quad t\in(N,\infty).$$
Here, $\epsilon > 0$ is chosen small enough so that $(\lambda-\epsilon,s-\epsilon)$ is still contained in the open set $Z$. We can further choose a subsequence $\{\chi_{M_N}\}$ of $\{\chi_N\}$, such that $\supp(\dot\chi_N) \subset \{t\in\R_+ \,|\, \chi_{M_N}(t)=1\}$ for all $N$.
Thus, $t^se^{-\lambda t}\dot{\chi}_N(t)$ can be written as the sum of two terms:
$$t^se^{-\lambda t}\dot{\chi}_N(t) = N^{-\epsilon}\,t^{s'-1}e^{-\lambda t}\chi_{M_N}(t)f_1(t) + e^{-\delta N}\,t^{s}e^{-\lambda' t}\chi_{M_N}(t)f_2(t),$$
where $f_1$ and $f_2$ are bounded smooth functions and $(s,\lambda'),(s',\lambda) \in Z$. By Assumption \ref{assumption_G_convergencce}, when we pair this with the distribution $\str\big(\theta_{\tau}e^{-t{D_\tau}}\big)$ and take the limit $N\to\infty$, the two terms will vanish due to the prefactors involving $N$.

Consider now the first two terms of equation \eqref{trace_pre_limit}. These are analytic functions in the parameters $\lambda$ and $s$. Furthermore, Assumption \ref{assumption_diff_hormander} ensures that $\tau \to K_\tau$ is continuous with respect to the H\"ormander topology on $\D_\Gamma(\R_+\times M\times M)$ with $\Gamma\cap N^*\iota=\emptyset$, where $K_\tau$ is the Schwartz kernel of $\exp(-t{D_\tau})$. This continuity is preserved by the application of the differential operator $\theta_\tau$, depending smoothly on $\tau$. Thus, the terms of equation \eqref{trace_pre_limit} are in fact continuous functions of $\tau$. By Assumption \ref{assumption_G_convergencce} the $N\to\infty$ limit is locally uniform in $(\tau,\lambda,s)$ for $(\lambda,s) \in Z$. This implies that the continuity and analyticity of the terms in \eqref{trace_pre_limit} is preserved in the limit. Moreover, since the derivatives 
$$\frac{d}{d\tau}\pair{\str\left(e^{-t{D_\tau}}\big|_{L_\tau}\right)}{t^{s-1}e^{-t\lambda}\chi_N(t)}$$
converge locally uniformly in $\tau$ as $N\to\infty$, we can take the $\tau$-derivative inside the limit, when computing the derivative of $F(\tau,\lambda,s)$. Therefore, when $(\lambda,s)\in Z$, we find
\begin{equation*}
\begin{split}
    \frac{d}{d\tau}F(\tau,\lambda,s) &= \frac{1}{\Gamma(s)}\lim_{N\to\infty}\frac{d}{d\tau}\pair{\str\left(e^{-t{D_\tau}}\big|_{\maybe{\mathsf{L}}_\tau}\right)}{t^{s-1}e^{-t\lambda}\chi_N(t)} \\ 
    &= \lambda\,\frac{1}{\Gamma(s)}\pair{\str\big(\theta_{\tau}e^{-t{D_\tau}}\big)}{t^se^{-\lambda t}} - \frac{s}{\Gamma(s)}\pair{\str\big(\theta_{\tau}e^{-t{D_\tau}}\big)}{t^{s-1}e^{-\lambda t}},
    \end{split}
\end{equation*}
which is a continuous function of $(\tau,\lambda,s)$ and analytic in $(\lambda,s)\in Z$.
\end{proof}

We can now apply the final Assumption \ref{assumption_analytic_continuation} concerning the analytic continuation of $G(\tau,\lambda,s)$ to conclude the theorem.

\begin{proof}[Proof of Theorem \ref{thm_invariance}]
Integrating equation \eqref{F_derivative} over $\tau$ and applying the fundamental theorem of calculus to the continuous function $\frac{d}{d\tau}F(\tau,\lambda,s)$ gives
\begin{equation}
\label{F_difference}
    F(\tau,\lambda,s) - F(0,\lambda,s) = \lambda\,\frac{1}{\Gamma(s)}\int_0^\tau G(\tau',\lambda,s+1)\,d\tau' - \frac{s}{\Gamma(s)}\int_0^\tau G(\tau',\lambda,s)\,d\tau'.
\end{equation}
This expresses the difference of the function $F$ for two different general codifferentials in terms of the analytic function $G(\tau',\cdot,\cdot): Z \to \mathbb{C}$. The flat determinant is obtained by evaluating the $s$-derivative of $F$ at the point $\lambda=0,\, s=0$. So the prefactors of $\lambda$ and $s$ in \eqref{F_difference} already suggest that the difference of the flat determinants should vanish. However, \eqref{F_difference} only holds on the domain $Z\subset\mathbb{C}^2$ and to proceed one must be able to analytically continue the right hand side of \eqref{F_difference} to $\lambda=0,\, s=0$.

By Assumption \ref{assumption_analytic_continuation}, both the integrands in \eqref{F_difference} have analytic continuations to $\lambda=0,\,s=0$. Moreover, by the local boundedness of this analytic continuation, the integrals converge absolutely and $\int_0^\tau|G(\tau',\lambda,s)|\,d\tau'$ is locally bounded for $(\lambda,s) \in \tilde{Z}$. This implies that the integrals converge to analytic functions of $(\lambda,s)$ in the domain $\tilde{Z}$.
Thus, we can evaluate the $s$-derivative of equation \eqref{F_difference} at ${(\lambda,s)=(0,0)}$. Recalling that the inverse of the gamma function is an entire function with a first order zero at $s=0$, the second term in \eqref{F_difference} is of order $s^2$ near $s=0$. Since $\int_0^\tau G(\tau',\lambda,s)\,d\tau'$ is analytic at $(\lambda,s) = (0,0)$, this implies that
$$\frac{d}{ds}\Big(\frac{s}{\Gamma(s)}\int_0^\tau G(\tau',\lambda,s)\,d\tau'\Big)$$
vanishes at $s=0$. Similarly, 
$$\frac{d}{ds}\Big(\frac{1}{\Gamma(s)}\int_0^\tau G(\tau',\lambda,s+1)\,d\tau'\Big)$$
remains finite at $(s,\lambda) = (0,0)$, so evaluating the $s$-derivative of the first term in \eqref{F_difference} at $\lambda=0$ yields zero. Thus,
$$\frac{d}{ds}\big(F(\tau,\lambda,s) - F(0,\lambda,s)\big)\Big|_{\lambda=0,\,s=0} = 0.$$
Finally, applying the definition of the flat determinant, see \eqref{det_from_F}, we find
\begin{equation}
    \frac{\sdet({D_\tau}|_{L_\tau})}{\sdet({D_0}|_{L_0})} = \exp\Big(-\frac{d}{ds}\big(F(\tau,\lambda,s) - F(0,\lambda,s)\big)\Big|_{\lambda=0,\,s=0}\Big) = 1
\end{equation}
Since $\sdet({D_0}|_{L_0})$ is non-zero by assumption, we find that the flat determinant restricted to the general codifferential subspace remains constant along the variation of general codifferentials.
\end{proof}

\section{Local constancy of the analytic torsion}\label{sec:ATinvariance}
In this section we recover a classic result by Ray and Singer concerning the invariance of the analytic torsion on the choice of a metric over the base manifold \cite{ray_singer}. In fact, we will more generally apply our local constancy result to families of general codifferentials $\delta_\tau$ with positive definite elliptic characteristic operators $D_\tau=[\delta_\tau,\d]$. We will see that on odd dimensional manifolds any inner variation $\tau\to\delta_\tau$ of general codifferentials whose characteristic operator is elliptic and (strictly) positive definite satisfies all the requirements of Theorem \ref{thm_invariance}.

\subsection{The Hodge general codifferential}
We consider here the Hodge codifferential. Let $g$ be a Riemannian metric on $M$ and consider the map:  
\begin{equation}
\label{codifferential}
\codif_g: \Omega^k(M,E) \to \Omega^{k-1}(M,E), \quad  \codif_g = (-1)^{n(k-1)+1}*_g\d*_g = (-1)^k *_g\d*_g^{-1},
\end{equation}
where $*_g$ is the Hodge star operator with respect to the metric $g$. Note that $\codif_g$ is the adjoint of $\d$ with respect to the inner product on $\Omega^\bullet(M,E)$ induced by $g$ and the Hermitian inner product on the fibers of $E$. Recall that this Hermitian structure is assumed to be compatible with the connection on $E$.

\begin{lemma}
\label{codiff_as_GF_op}
    Given the twisted topological data of Definition \ref{def:TTD}, if $g$ is a Riemannian metric on $M$, then the Hodge codifferential with respect to $g$ defines an acyclic, transversal general codifferential $\delta=\codif_g$, in the sense of Definitions \ref{def_general_codifferential} and \ref{def_codifferential_types}. The splitting is given by the Hodge decomposition for an acyclic de Rham complex
    \begin{equation}
    \label{hodge_splitting}
        \Omega^\bullet(M,E) = \im(\codif_g) \oplus \im(\d) \doteq L\oplus C,
    \end{equation}
    and the characteristic operator is just the Hodge Laplacian twisted by the connection $$D = \Delta_g.$$
\end{lemma}
\begin{proof}
    Note that $\codif_g: \Omega^\bullet(M,E) \to \Omega^{\bullet-1}(M,E)$ is a first order differential operator. Due to the flatness of the connection, we have
    $$\codif_g\circ\codif_g = - *_g\d\circ\d*_g^{-1} = 0.$$
    Thanks to the compatibility of $\nabla$ and $\pair{\cdot}{\cdot}_E$, we further find
    $$\int_M\langle\omega\wedge\codif_g\eta\rangle_E = (-1)^k\int_M\langle\codif_g\omega\wedge\eta\rangle_E, \quad\text{for}\quad \eta\in\Omega^k(M,E),\, \omega\in\Omega^{n-k+1}(M,E).$$
    Hodge decomposition yields the splitting
    $$\Omega^\bullet(M,E) = \im(\codif_g) \oplus \im(\d) \oplus \ker(\Delta_g),$$
    where the Laplacian $\Delta_g = [\codif_g,\d]$ coincides with the characteristic operator. By assumption, the twisted de Rham complex is acyclic, so $\ker(\Delta_g) = \{0\}$ and we have \eqref{hodge_splitting}. Again by compatibility of the connection and the Hermitian structure on $E$, we indeed have the two isotropic complements
    $$
    \int_M \langle\omega\wedge\eta\rangle_E = 0, \quad \forall \omega,\eta \in \im(\d), \qquad \int_M \langle\omega\wedge\eta\rangle_E = 0, \quad \forall \omega,\eta \in \im(\delta_g),
    $$
    Finally, since $\d$ commutes with $\Delta_g$, we find that the characteristic operator leaves the subspace $\im(\d)$ invariant. Note that the Hodge splitting \eqref{hodge_splitting} is moreover orthogonal with respect to the inner product induced by the metric $g$.
\end{proof}

\begin{rmk}
The Hodge codifferential is a special instance, in that two of the three complexes of Remark \ref{remark_complexes} coincide. The complementary complex to the general codifferential $\codif_g$, which we called $\delta^\bot$ and defines the splitting in \eqref{hodge_splitting} via $C=\im(\delta^\bot)$, is just the original complex given by the twisted de Rham differential $\d$. The example we study in Section \ref{section_ruelle_independence} will showcase a more general situation.
\end{rmk}

\begin{lemma}
\label{codiff_homotopy}
    Let $\tau\in(-1,1)\to g(\tau)$ be a smooth family of Riemannian metrics on $M$. Then the family of codifferentials $\codif_{g(\tau)}$ forms an inner variation of $\delta_0\doteq\delta_{g(0)}$ as in Definition \ref{def_homotopy_structure}. More precisely,
    $$\frac{d}{d\tau}\codif_{g(\tau)} = [\theta_\tau,\codif_{g(\tau)}], \quad\text{with}\quad \theta_\tau = \Bigl(\frac{d}{d\tau}*_{g(\tau)}\Bigr)*_{g(\tau)}^{-1}.$$
    In fact, the family is integrable with
    $$\delta_{g(\tau)} = *_{g(\tau)}*_{g(0)}^{-1}\delta_{g(0)}\bigl(*_{g(\tau)}*_{g(0)}^{-1}\bigr)^{-1}.$$
\end{lemma}
\begin{proof}
    Restricted to forms of degree $k$, we have $\codif_{g(\tau)}=(-1)^k*_{g(\tau)}\d*_{g(\tau)}^{-1}$ for all $\tau$, where $*_{g(\tau)}$ is the Hodge star operator with respect to the metric $g(\tau)$. Thus,
    $$\codif_{g(\tau)} = *_{g(\tau)}*_{g(0)}^{-1}\codif_{g(0)}*_{g(0)}*_{g(\tau)}^{-1} = *_{g(\tau)}*_{g(0)}^{-1}\codif_{g(0)}\bigl(*_{g(\tau)}*_{g(0)}^{-1}\bigr)^{-1}.$$
    This can be written $\codif_{g(\tau)}=\beta_\tau\circ\codif_{g(0)}\circ\beta_\tau^{-1}$ with respect to the smooth family of degree-preserving bundle automorphisms
    $${\beta}_{\tau} = (*_{g(\tau)}*_{g(0)}^{-1}) \otimes 1: \wedge^\bullet T^*M \otimes E \to \wedge^\bullet T^*M \otimes E.$$
    This shows that $\tau\to\codif_{g(\tau)}$ is an integrable family of general codifferentials in the sense of Definition \ref{def_homotopy_structure}, and hence also an inner variation. A direct calculation shows that $\frac{d}{d\tau}\codif_{g(\tau)} = [\theta_\tau,\codif_{g(\tau)}]$, where
    $${\theta}_{\tau} = \bigl((\tfrac{d}{d\tau}*_{g(\tau)})*_{g(\tau)}^{-1}\bigr) \otimes 1: \wedge^\bullet T^*M \otimes E \to \wedge^\bullet T^*M \otimes E$$
    is a smooth family of degree-preserving bundle endomorphisms.
\end{proof}

The Hodge codifferential $\delta=\codif_g$ gives rise to the elliptic differential operator $D=\Delta_g$ as characteristic operator. Note moreover that $\Delta_g$ is essentially self-adjoint on the space of $L^2$ forms with respect to the inner product induced by the metric $g$. In fact, thanks to the assumption of vanishing cohomology for the twisted de Rham complex, $\Delta_g$ has strictly positive spectrum. It is well known in this case that the zeta-function regularized determinant of $\Delta_g$ is well-defined, and it coincides with the flat determinant used here (see e.g.\ \cite{hadfield_kandel_schiavina,Schiavina_Stucker2}). The square-root of the regularized determinant of the twisted Laplacian restricted to coexact forms is known in the literature as the analytic torsion or Ray-Singer torsion
$$T_{\rho,g}(M,E) = \sdet\bigl(\Delta_g|_{\im(\codif_g)}\bigr)^{\frac{1}{2}}.$$
In the next subsection, we will show that the constancy of the analytic torsion along a smooth variation of metrics follows from our general result, Theorem \ref{thm_invariance}. Here, we briefly recall the well-definedness of the flat determinant in the more general case of a positive definite elliptic operator.

\begin{lemma}
\label{lemma_elliptic_regular}
    Let $\delta$ be a general codifferential such that $D=[\delta,\d]$ is elliptic and (strictly) positive definite with respect to some inner product $\pair{\cdot}{\cdot}$ on $\Omega^\bullet(M,E)$, induced by a density on $M$ and an inner product on the fibers of $T^*M$. Then $\delta$ is a regular general codifferential in the sense of \ref{def_codifferential_types}. In particular, the flat super-determinant
    $$\sdet\bigl(D|_{\im(\delta)}\bigr)$$
    is well-defined and nonvanishing.
\end{lemma}
\begin{proof}
    We first show that $\delta$ must be acyclic and transversal. Note that the completion of $\Omega^k(M,E)$ with respect to $\pair{\cdot}{\cdot}$ gives the $L^2$ space of $E$-valued $k$-forms and the elliptic operator $D$ defines a positive semi-definite, and thus self-adjoint, unbounded operator on this $L^2$ space. By the spectral theory of self-adjoint elliptic operators, $D$ has discrete spectrum consisting of eigenvalues with smooth eigenfunctions. The positivity of $D$ on $\Omega^\bullet(M,E)$ implies that $0$ must lie in the resolvent set and $D$ is in fact strictly positive definite on $L^2$. Again by ellipticity, the $L^2$ inverse $D^{-1}$ is a pseudodifferential operator, in particular mapping $\Omega^\bullet(M,E)$ to itself. Furthermore, $D^{-1}$ commutes with $\delta$ and $\d$. Thus, we can define $L=\im(\delta)$, $C=\im(\d)$ and the continuous projection operators
    $$\Pi_L = D^{-1}\delta\d: \Omega^\bullet(M,E) \to L, \quad \Pi_C = D^{-1}\d\delta: \Omega^\bullet(M,E) \to C.$$
    Since $\Pi_L+\Pi_C=1$ and $\Pi_L\Pi_C=\Pi_C\Pi_L=0$, we have a generalized Hodge decomposition:
    $$\Omega^\bullet(M,E) = L\oplus C = \im(\delta)\oplus\im(\d),$$
    where $D$ leaves $\im(\d)$ invariant and $\ker(\delta)\cap\im(\d)=\{0\}$ due to the invertibility of $D$.
    
    Regarding the flat determinant, it is well-known that a positive definite elliptic differential operator has a smooth heat kernel\footnote{See Appendix \ref{appendix_heat_kernel} for an overview of the heat kernel construction} $K(t,x,y) \in \C(\R_+\times M\times M)$. That is, $D$ generates a semigroup $e^{-tD}$ with smooth Schwartz kernel $K$. Thus, the flat trace of $e^{-tD}$ is well-defined and given by integrating $K$ over the diagonal in $M\times M$. Moreover, the pseudodifferential projection operator $\Pi_{L}$ has wavefront set contained in the conormal to the diagonal in $M\times M$. Thus, the restricted flat supertrace $\str(e^{-tD}|_{L})=\str(e^{-tD}\Pi_{L})$ is well-defined, see \cite[Definition B.12]{Schiavina_Stucker2}. Using Lemma \ref{restricted_trace}, we can write the restricted trace as
    $$\str\bigl(e^{-tD}|_{L}\bigr) = \sum_{k=0}^n(-1)^{k+1}k\tr\bigl(e^{-tD^{(k)}}\bigr),$$
    where $D^{(k)}$ denotes the action of $D$ on $\Omega^k(M,E)$.

    Now, in each form degree, we have
    $$\str\bigl(e^{-tD^{(k)}}\bigr) = \int_M \mathrm{Tr}\bigl(K^{(k)}(t,x,x)\bigr) = \mathrm{tr}_{L^2}\bigl(e^{-tD^{(k)}}\bigr),$$
    where $\mathrm{Tr}$ denotes the trace in the fibers of $\wedge^kT^*M\otimes E$ and $\mathrm{tr}_{L^2}$ is the usual $L^2$ trace. The latter can be written in terms of the eigenvalues of $D^{(k)}$:
    $$\mathrm{tr}_{L^2}\bigl(e^{-tD^{(k)}}\bigr) = \sum_{j\in\N}e^{-t\lambda_j}.$$
    To compute the flat determinant in degree $k$, we must consider the following integral, see \eqref{det_function}, which we split into two pieces:
    \begin{equation}
    \begin{split}
    \label{split_integral}
        F^{(k)}(\lambda,s) &= \frac{1}{\Gamma(s)}\int_0^\infty \tr\bigl(e^{-tD^{(k)}}\bigr) t^{s-1}e^{-\lambda t}\,dt \\
        &= \frac{1}{\Gamma(s)}\int_1^\infty \mathrm{tr}_{L^2}\bigl(e^{-tD^{(k)}}\bigr) t^{s-1}e^{-\lambda t}\,dt + \frac{1}{\Gamma(s)}\int_0^1\int_M \mathrm{Tr}\bigl(K^{(k)}(t,x,x)\bigr) t^{s-1}e^{-\lambda t}\,dt.
    \end{split}
    \end{equation}
    The first term can be estimated locally uniformly in $s\in\mathbb{C}$ by
    $$\int_1^\infty \Bigr|\mathrm{tr}_{L^2}\bigl(e^{-tD^{(k)}}\bigr) t^{s-1}e^{-\lambda t}\Bigr|\,dt \leq \sum_{j\in\N}\int_1^\infty e^{-t(\lambda_j+\Re(\lambda))}t^{\Re(s)-1}\,dt \leq C\sum_{j\in\N}e^{-(\lambda_j+\Re(\lambda))},$$
    where we take $\Re(\lambda)>-c$ for $c=\min_j\lambda_{j}>0$. This converges locally uniformly in $\lambda$ thanks to Weyl's law for the positive definite elliptic operator $D$, see also \cite[Lemma 1.6.3]{gilkey} for a weaker but sufficient eigenvalue bound. Thus, the first term in \eqref{split_integral} defines an analytic function in $\{s\in\mathbb{C}\}\times\{\Re(\lambda)>-c\}$.

    For the second term in \eqref{split_integral}, we use the well-known heat kernel asymptotics, see for instance \cite[Lemma 1.7.4]{gilkey}. We have $|\mathrm{Tr}(K^{(k)}(t,x,x))| \leq t^{-\frac{n}{m}}$ for $t\in(0,1)$, where $m$ is the order of $D$ and $n$ the dimension of $M$. Thus, the integral converges to an analytic function for all $\lambda\in\mathbb{C}$ and $\Re(s)>\frac{n}{m}$.
    Since $\lambda=0$ is in the domain of convergence of \eqref{split_integral}, we can directly set $\lambda=0$ and must only prove an analytic continuation of $F^{(k)}(0,s)$ to $s=0$. By the heat kernel asymptotics, we have
    $$\int_M\mathrm{Tr}(K^{(k)}(t,x,x)) = \sum_{j=0}^{n+m}t^{\frac{j-n}{m}}B_j + \mathcal{O}(t)$$
    for some coefficients $B_j$, where $B_j=0$ for $j$ odd. Thus, we find in the region of convergence
    $$F^{(k)}(0,s) = F_{a}(s) + \sum_{j=0}^{n+m}\frac{1}{\Gamma(s)}\cdot\frac{B_j}{s-\frac{j-n}{m}},$$
    where $F_{a}$ is analytic in a neighborhood of $s=0$. This expression has an evident analytic continuation to $s=0$. Note that the potential first order pole at $s=0$ for $j=n$ is canceled by the first order pole of $\Gamma(s)$ at $s=0$. Thus, $F^{(k)}(0,s)$ is analytic in a neighborhood of $s=0$ and the flat superdeterminant is well-defined by the non-zero expression
    \begin{equation}\label{elliptic_nonvanishing}
        \sdet\bigl(e^{-tD}|_{\im(\delta)}\bigr) = \exp\Bigl(\sum_{k=0}^n(-1)^{k}\frac{d}{ds}\Big|_{s=0}F^{(k)}(0,s)\Bigr).
    \end{equation}
\end{proof}

\subsection{Local constancy in the elliptic case}
\label{section_invariance_elliptic}
In this subsection, we apply the framework developed in Section \ref{section_independence} to the particularly nice case when the characteristic operator is elliptic and positive definite. As a corollary, we obtain the local constancy of the analytic torsion under smooth variations of the metric. The proof of the following statement will be presented after we state this corollary.

\begin{prop}
\label{invariance_elliptic}
    Let $M$ be an odd dimensional compact, orientable manifold. Let $\tau\in(-1,1) \to {\delta_\tau}$ be a smooth family of general codifferentials that form an inner variation of $\delta_0$ in the sense of Definition \ref{def_homotopy_structure}. Assume further that $D_\tau=[\delta_\tau,\d]$ is elliptic and (strictly) positive definite with respect to a smooth family of inner products $\pair{\cdot}{\cdot}_\tau$ on $\Omega^\bullet(M,E)$, where $\pair{\cdot}{\cdot}_\tau$ is induced by a smooth family of inner products on the fibers of $\wedge^kT^*M\otimes E$ and a smooth family of densities on $M$. Then for all $\tau\in(-1,1)$, we have
    $$\sdet\bigl(D_\tau|_{L_\tau}\bigr) = \sdet\bigl(D_0|_{L_0}\bigr).$$
\end{prop}

\begin{cor}[{\cite[Theorem 2.1]{ray_singer}}]\label{thm_torsion_constant}
Let $\tau \to g(\tau)$ be a smooth family of Riemannian metrics on $M$. Then the analytic torsion
\[
\tau \mapsto T_{\rho,g(\tau)}(M,E) = \sdet(\Delta_{g(\tau)}|_{L_\tau})^{\frac{1}{2}},
\]
is constant in $\tau$.
\end{cor}
\begin{proof}
    When $\dim(M)$ is even, the analytic torsion is always trivially $1$, due to an added symmetry coming from the Hodge star operator, as shown in \cite[Theorem 2.3]{ray_singer}. Thus, odd-dimensional $M$ is the case of interest. We saw in Lemma \ref{codiff_as_GF_op} that $D_\tau=\Delta_{g(\tau)}$ can be interpreted as the characteristic operator with respect to the smooth family of general codifferentials ${\delta_\tau}=\codif_{g(\tau)}$. In Lemma \ref{codiff_homotopy}, it is shown that this family forms an inner variation. The twisted Laplacian is clearly an elliptic operator which is positive semidefinite with respect to the smooth family of inner products on $\Omega^\bullet(M,E)$ induced by the metric $g(\tau)$ together with the inner product on $E$. By assumption on the vanishing of the twisted de Rham cohomology it is in fact positive definite. Thus, the Corollary follows immediately from Proposition \ref{invariance_elliptic}.
\end{proof}

\begin{rmk}
    As noted in Remark \ref{rmk_generalization}, there is nothing special about the twisted de Rham differential and Proposition \ref{invariance_elliptic} could be formulated in a more general setting. Namely, consider a graded vector bundle $\V$ over an odd dimensional compact, orientable manifold $M$ equipped with a degree $1$ differential operator $d$ and a smooth family of degree $-1$ differential operators $\delta_\tau$, which satisfy $d\circ d=0$ and $\delta_\tau\circ\delta_\tau=0$ for all $\tau$, and thus form a pair of complexes
    \[
    \begin{tikzcd}
        \cdots \arrow[r, shift left] & C^\infty(M,\V^k) \arrow[l, shift left] \arrow[r, shift left, swap, "d"'] & C^\infty(M,\V^{k+1}) \arrow[l, shift left, swap, "\delta_\tau"'] \arrow[r, shift left] & \cdots \arrow[l, shift left]
    \end{tikzcd}
    \]
    If $\tau\to\delta_\tau$ is an inner variation, in the sense that $\frac{d}{d\tau}\delta_\tau = [\theta_\tau,\delta_\tau]$ for a family of degree $0$ differential operators $\theta_\tau$, and if the graded commutators $D_\tau=[\delta_\tau,d]$ are elliptic and strictly positive definite for each $\tau$, then the proof of Proposition \ref{invariance_elliptic} goes through almost verbatim and we have local constancy of the restricted flat determinant:
    $$\sdet(D_\tau|_{\im(\delta_\tau)}) = \sdet(D_0|_{\im(\delta_0)}), \quad \forall\,\tau.$$
\end{rmk}

We will prove Proposition \ref{invariance_elliptic} by showing that the requirements of Theorem \ref{thm_invariance} are satisfied. The proof will involve the smooth heat kernel for $D_\tau$ and its asymptotics as $t\to 0$. In particular, we use that the family of heat kernels for $D_\tau$ depends smoothly on the parameter $\tau$. There is a vast literature concerning heat kernels of elliptic operators. In the interest of remaining self-contained and because some care is needed to treat all constructions uniformly in $\tau$, we provide an overview of the heat kernel construction in the appendix, following in particular \cite{gilkey,berline_getzler}. Thus, a proof of the following statement can be found in Appendix \ref{appendix_heat_kernel}.

\begin{lemma}
\label{lemma_smooth_heat_kernel}
    Let $\tau\to D_\tau$ be a smooth family of positive definite elliptic operators as in Proposition \ref{invariance_elliptic}. Then the family of heat kernels for $D_\tau$ depends smoothly on $\tau$. That is, denoting by $\V=\wedge^\bullet T^*M\otimes E$ the vector bundle, the Schwartz kernel $K(\tau,t,x,y)$ of $e^{-tD_\tau}$ satisfies
    \begin{equation}
    \label{smoothness_heat_kernel}
        K \in \C\bigl((-1,1)\times \R_+\times M\times M;\V\boxtimes(\V^*\otimes\wedge^nT^*M)\bigr).
    \end{equation}
\end{lemma}

Using Lemma \ref{lemma_smooth_heat_kernel}, we can now prove Proposition \ref{invariance_elliptic}.

\begin{proof}[Proof of Proposition \ref{invariance_elliptic}]\label{proof_invariance_elliptic}
    By assumption $\tau\in(-1,1)\to\delta_\tau$ is an inner variation, and Lemma \ref{lemma_elliptic_regular} shows that $\delta_\tau$ is regular for each $\tau$. It remains to show that Assumptions \ref{assumption_diff_semigroup}-\ref{assumption_analytic_continuation} are satisfied. 

    Assumptions \ref{assumption_diff_semigroup} and \ref{assumption_smaller_wavefront} follow immediately from the smoothness of the Schwartz kernel $K$ of $e^{-tD_\tau}$, see \eqref{smoothness_heat_kernel}. For Assumption \ref{assumption_diff_hormander} we take $\Gamma=\emptyset$. The smoothness in \eqref{smoothness_heat_kernel} implies that $\tau \to K(\tau,\cdot,\cdot,\cdot) \in \C(\R_+\times M\times M)$ is smooth with respect to the usual Fr\'echet topology on the space of smooth functions. Since $\D_\emptyset(\R_+\times M\times M)=\C(\R_+\times M\times M)$ are topologically isomorphic, see \cite[Lemma 7.2]{brouder}, Assumption \ref{assumption_diff_hormander} follows.

    We turn to Assumption \ref{assumption_G_convergencce}. Since $\str(\theta_\tau e^{-tD_\tau})$ depends smoothly on $t$, the pairing in \ref{assumption_G_convergencce} becomes an integral and \ref{assumption_G_convergencce} is equivalent to the locally uniform convergence of
    $$\int_0^\infty \bigl|\str(\theta_\tau e^{-tD_\tau})t^{s-1}e^{-\lambda t}\bigr|\,dt$$
    for $\Re(s),\Re(\lambda)$ large enough.
    Moreover, since $\theta_\tau e^{-tD_\tau}$ is a smoothing operator, we have
    $$\str\bigl(\theta_\tau e^{-tD_\tau}\bigr) = \int_M \mathrm{sTr}\bigl(\theta_\tau K(\tau,t,x,y)|_{x=y}\bigr) = \mathrm{str}_{L^2}\bigl(\theta_\tau e^{-tD_\tau}\bigr),$$
    where $\mathrm{sTr}$ denotes the fiberwise supertrace in the bundle $\wedge^\bullet T^*M\otimes E$ and $\mathrm{str}_{L^2}$ is the usual $L^2$ super trace of a trace-class operator. Using the above identifications we split the integral over $t$ into three parts:
    \begin{equation}
    \label{three_terms}
    \begin{split}
        \int_0^\infty \str\bigl(\theta_\tau e^{-tD_\tau}\bigr) t^{s-1}e^{-\lambda t}\,dt &= \int_1^\infty \mathrm{str}_{L^2}\bigl(\theta_\tau e^{-tD_\tau}\bigr) t^{s-1}e^{-\lambda t}\,dt \\
        &+ \int_0^1\int_M \mathrm{sTr}\bigl(\theta_\tau K(\tau,t,x,y)|_{x=y}-\theta_\tau K_N(\tau,t,x,y)|_{x=y}\bigr) t^{s-1}e^{-\lambda t}\,dt \\
        &+ \int_0^1\int_M \mathrm{sTr}\bigl(\theta_\tau K_N(\tau,t,x,y)|_{x=y}\bigr) t^{s-1}e^{-\lambda t}\,dt,
    \end{split}
    \end{equation}
    where $K_N$ is the approximate heat kernel constructed in Lemma \ref{lemma_approx_heat_kernel} of the appendix.

    To estimate the first term, note that the $L^2$ trace can be written in terms of the eigenvalues of the positive self-adjoint operator $D_\tau$. Indeed, denoting by $\lambda_{\tau,1}\leq \lambda_{\tau,2}\leq\cdots$ the positive eigenvalues counted with multiplicty and by $\phi_{\tau,j}$ a corresponding orthonormal set of eigenfunctions, we find
    $$\bigl|\mathrm{str}_{L^2}\bigl(\theta_\tau e^{-tD_\tau}\bigr)\bigr| \leq \sum_{j=1}^\infty e^{-t\lambda_{\tau,j}}\bigl|\pair{\phi_{\tau,j}}{\theta_\tau\phi_{\tau,j}}_{L^2}\bigr| \leq C\sum_{j=1}^\infty(1+\lambda_{\tau,j})e^{-t\lambda_{\tau,j}}$$
    locally uniformly in $\tau$.
    Here, we used elliptic regularity for the order $m$ elliptic operator $D_\tau$ to estimate
    $$\bigl|\pair{\phi_{\tau,j}}{\theta_\tau\phi_{\tau,j}}_{L^2} \leq \|\theta_\tau\phi_{\tau,j}\|_{L^2} \leq C\|\phi_{\tau,j}\|_{H^l} \leq C\bigl(\|D_\tau\phi_{\tau,j}\|_{L^2} + \|\phi_{\tau,j}\|_{L^2}\bigr),$$
    where $l\leq m$ is the order of the differential operator $\theta_\tau$. Note that the norms above are $\tau$-dependent but due to the smoothness with respect to $\tau$ the estimate is uniform for $\tau$ in compacta. The first integral in \eqref{three_terms} can now be bounded by
    \begin{equation*}
        \int_1^\infty \bigl|\mathrm{str}_{L^2}\bigl(\theta_\tau e^{-tD_\tau}\bigr) t^{s-1}e^{-\lambda t}\bigr|\,dt \leq C\sum_{j=1}^\infty (1+\lambda_{\tau,j}) \int_1^\infty e^{-t(\lambda_{\tau,j}+\Re(\lambda))} t^{\Re(s)-1}\,dt.
    \end{equation*}
    The integral over $t$ satisfies
    $$\int_1^\infty e^{-t(\lambda_{\tau,j}+\Re(\lambda))} t^{\Re(s)-1}\,dt \leq e^{-(\lambda_{\tau,j}+\Re(\lambda))} \int_0^\infty e^{-t(\lambda_{\tau,1}+\Re(\lambda))} (1+t)^{\Re(s)-1}\,dt \leq Ce^{-(\lambda_{\tau,j}+\Re(\lambda))} $$
    uniformly for $\tau$ in a compact subset $I\subset(-1,1)$, $\Re(\lambda) > -\frac{1}{2}\min_{\tau\in I}(\lambda_{\tau,1})$ and $s$ in any compact subset of $\mathbb{C}$. Thus, we find
    \begin{equation}
    \label{elliptic_trace_sum}
        \int_1^\infty \bigl|\mathrm{str}_{L^2}\bigl(\theta_\tau e^{-tD_\tau}\bigr) t^{s-1}e^{-\lambda t}\bigr|\,dt \leq C\sum_{j=1}^\infty (1+\lambda_{\tau,j})e^{-(\lambda_{\tau,j}+\Re(\lambda))}.
    \end{equation}
    We now use that the eigenvalues of $D_\tau$ satisfy the following growth estimate uniformly for $\tau\in I$:
    \begin{equation}
    \label{eigenvalue_growth}
        \lambda_{\tau,j} \geq Cj^\delta, \quad\text{for some }\, C,\delta > 0.
    \end{equation}
    This can be inferred from Weyl's law for a self-adjoint elliptic operator, see for instance \cite{hormander_weyl_law}. In fact, the weaker asymptotics of \eqref{eigenvalue_growth} follow by the more straightforward arguments in the proof of \cite[Lemma 1.6.3]{gilkey}, where uniformity for $\tau$ contained in a compact subset of $(-1,1)$ follows from the smoothness of $\tau\to D_\tau$ and the inner product $\tau\to\pair{\cdot}{\cdot}_\tau$. Thus, the sum over $j$ in \eqref{elliptic_trace_sum} converges locally uniformly in $\tau$. We have shown that for an $\epsilon>0$, which is smaller than the minimal eigenvalue of $D_\tau$ for $\tau\in I$, the integral converges locally uniformly for all $(s,\lambda) \in \mathbb{C}\times\{\Re(\lambda)>-\epsilon\}$. Thus, the first term in \eqref{three_terms} defines an analytic function in this domain.
    
    For the second term, we note that $K_N$ provides a good approximation to $K$ at small $t$. Indeed, \eqref{volterra_series} together with the estimates in \eqref{convolution_estimate} imply that
    $$\|K(\tau,t,x,y) - K_N(\tau,t,x,y)\|_{C^l(M\times M)} \leq C t^{\frac{N-n-l}{m}}, \quad \forall\, t\in (0,1)$$
    locally uniformly in $\tau$. Choosing $N\geq m(n+l+1)$, where $l$ denotes the order of $\theta_\tau$, we can bound
    $$\int_M \bigl|\mathrm{sTr}\bigl(\theta_\tau K(\tau,t,x,y)|_{x=y}\bigr| \leq Ct, \quad \forall\, t\in (0,1)$$
    uniformly for $\tau$ in compacta. Thus, the second integral in \eqref{three_terms} converges locally uniformly for all $(s,\lambda) \in \{\Re(s)>-1\}\times\mathbb{C}$ to an analytic function in this domain.

    For the third term we use the well known heat kernel asymptotics, see for instance \cite[Lemma 1.7.7]{gilkey}. We find
    \begin{equation}
    \label{heat_kernel_expansion}
        \int_M \mathrm{sTr}\bigl(\theta_\tau K_N(\tau,t,x,y)|_{x=y}\bigr) = \sum_{k=0}^{N} t^{\frac{k-l-n}{m}}B_k(\tau), \quad \forall\, t\in(0,1),
    \end{equation}
    where the $B_k$ depend smoothly on $\tau$ and $B_k=0$ if $k+l$ is odd. The last integral in \eqref{three_terms} can now be seen to converge locally uniformly for all $(s,\lambda) \in \{\Re(s)>l+n\}\times\mathbb{C}$. Note that the $B_k$ can be computed from the parametrix construction of Lemma \ref{lemma_approx_heat_kernel} in the appendix. Indeed, working in local coordinates and applying the order $l$ differential operator $\theta_\tau$ to \eqref{approx_heat_kernel}, we see that
    $$\theta_\tau K_N(\tau,t,x,y) = \sum_{0\leq k < N}\frac{1}{(2\pi)^n}\frac{1}{2\pi i}\int_{\R^n}\int_\gamma e^{i(x-y)\xi-t\lambda}b_{k}(\tau,x,\xi,\lambda)\,d\lambda d\xi,$$
    where $b_k \in S^{l-m-k}_{\Lambda,m}(U;\mathbb{C}^{r\times r})$ is a homogeneous symbol of order $l-m-k$ depending smoothly on $\tau$. The $b_k$ are obtained by applying $\theta_\tau$ to the symbol $q^N$ and collecting terms of the same order. Restricting to the diagonal and using homogeneity of the symbols, as in the proof of Lemma \ref{lemma_kernel_local}, we find for $t\in(0,1)$:
    \begin{equation}
    \label{heat_kernel_invariants}
        \theta_\tau K_N(\tau,t,x,y)|_{x=y} = \sum_{k=0}^{N}\int_{\R^n}\int_\gamma e^{-t\lambda}b_{k}(\tau,x,\xi,\lambda)\,d\lambda d\xi = \sum_{k=0}^{N}t^{\frac{k-l-n}{m}}\int_{\R^n}\int_\gamma e^{-\lambda}b_{k}(\tau,x,\xi,\lambda)\,d\lambda d\xi.
    \end{equation}
    The small $t$ expansion in \eqref{heat_kernel_expansion} follows by taking the supertrace in the fibers of the vector bundle and integrating over $M$. The $B_k(\tau)$ can be computed in local coordinates from \eqref{heat_kernel_invariants}. 

    We finally turn to the analytic continuation of the expression in \eqref{three_terms}, which is Assumption \ref{assumption_analytic_continuation}. The integrals all converge in a neighborhood of $\lambda=0$. We can thus evaluate \eqref{three_terms} at $\lambda=0$ and only need to prove analytic continuation in $s$ to $s=0$. Note that evaluating \eqref{three_terms} at $\lambda=0$ and multiplying with $-\frac{s}{\Gamma(s)}$ gives precisely the $\tau$-derivative of the function $F(\tau,0,s)$ used to define the restricted flat determinant, see Remark \ref{rmk_eliminating_lambda_s}. Furthermore, the first two terms in \eqref{three_terms} are already analytic in $s$ for all $\Re(s)>-1$. For the third term in \eqref{three_terms}, we set $\lambda=0$ and use the heat kernel asymptotics of \eqref{heat_kernel_expansion} to obtain for $\Re(s)>n+l$:
    $$\int_0^1\int_M \mathrm{sTr}\bigl(\theta_\tau K_N(\tau,t,x,y)|_{x=y}\bigr) t^{s-1}\,dt = \sum_{k=0}^{N} \Bigl(s-\frac{k-l-n}{m}\Bigr)^{-1}B_k(\tau).$$
    This clearly extends meromorphically to $\Re(s)>-1$ with poles at $\frac{k-l-n}{m}$ and corresponding residues $B_k(\tau)$. It remains to note that $B_k(\tau)=0$ when $k+l$ is odd, see \cite[Lemma 1.7.7]{gilkey}. Since $M$ has odd dimension $n$ by assumption, this shows that $s=0$ is not a pole of the analytic continuation. Note that the boundedness required in Assumption \ref{assumption_analytic_continuation} follows from the smoothness of the $B_m(\tau)$.
\end{proof}

\section{Local constancy of the value at zero of the Ruelle zeta function}
\label{section_ruelle_independence}

In this section we prove the local constancy of the value at zero of the Ruelle zeta function for a family of regular Anosov--Reeb vector fields. This result was first proved in \cite{viet_dang_fried_conjecture}, but we view it as a consequence of our general framework, and the otherwise known microanalytic behavior of the zeta function. This requires interpreting the (value at zero of the) Ruelle zeta function as a regularized determinant. See \cite{zworski,hadfield_kandel_schiavina,Schiavina_Stucker2}.

\subsection{The contact general codifferential}
\label{subsection_contact_gauge}
Consider an odd-dimensional smooth manifold $M$, equipped with a contact form $\alpha\in\Omega^1(M)$, i.e.\ such that $\alpha\wedge (d\alpha)^{\frac{\dim(M)-1}{2}}$ is a volume form. Let $X$ be the Reeb vector field associated to $\alpha$, and recall that $X$ is the unique vector field satisfying
\begin{equation}
\label{Reeb}
    \iota_X \alpha = 1, \qquad \iota_X d\alpha = 0,
\end{equation}
where $\iota_X$ denotes contraction with $X$. 

\begin{lemma}\label{lem_contactinvariantsplit}
    Let $X$ be the Reeb vector field of a contact form $\alpha$. Then contraction with $X$ defines an acyclic general codifferential $\delta=\iota_X$ in the sense of Definition \ref{def_general_codifferential}, with splitting given by
    \begin{equation}
    \label{splitting_contact_gauge}
        \Omega^\bullet(M,E) = \im(\iota_X)\oplus\im(\alpha\wedge) \doteq L\oplus C,
    \end{equation}
    and the characteristic operator is the Lie derivative with respect to $X$ (twisted by the connection):
    $$D = \L_X.$$
    Moreover, the characteristic operator leaves the splitting in \eqref{splitting_contact_gauge} invariant.
\end{lemma}
\begin{proof}
    We can view $\iota_X: \Omega^\bullet(M,E) \to \Omega^{\bullet-1}(M,E)$ as a zeroth order differential operator, which indeed satisfies $\iota_X\circ\iota_X = 0$. Moreover, we have
    $$\int_M\langle\iota_X\omega\wedge\eta\rangle_E = (-1)^{k+1}\int_M\langle\omega\wedge\iota_X\eta\rangle_E, \quad\mathrm{for}\quad \omega\in\Omega^k(M,E), \,\eta\in\Omega^{n-k+1}(M,E).$$
    Thanks to the non-vanishing of the vector field $X$ on $M$, the complex defined by $\iota_X$ is acyclic. Denote by $T^*_0M$ the conormal bundle to $X$, i.e.\ $(T^*_0M)_x$ are the cotangent vectors annihilating $X(x)$. Then we have a splitting $T^*M = \R\alpha \oplus T^*_0M$.
    Transferring this decomposition to the space of $E$-valued $k$-forms gives the claimed splitting
    \begin{equation}
    \label{contact_splitting}
        \Omega^k(M,E) = \Omega_0^k(M,E) \oplus \alpha\wedge\Omega_0^{k-1}(M,E),
    \end{equation}
    where 
    $$\Omega_0^k(M,E) = \ker(\iota_X)\cap\Omega^k(M,E) = \im(\iota_X)\cap\Omega^k(M,E).$$
    The characteristic operator $D=[\iota_X,\d]=\L_X$ gives the (twisted) Lie derivative by the Cartan formula.
    Finally, by definition of the Reeb vector field, see \eqref{Reeb}, we have
    \begin{equation*}
        \L_X\alpha = (d\iota_X + \iota_Xd){\alpha} = d(1) + \iota_Xd{\alpha} = 0.
    \end{equation*}
    Thus, for any $\omega \in \Omega^\bullet(M,E)$:
    \begin{equation*}
        \L_X({\alpha}\wedge\omega) = (\L_X{\alpha})\wedge\omega + {\alpha}\wedge(\L_X\omega) = {\alpha}\wedge(\L_X\omega).
    \end{equation*}
    So $\L_X$ leaves $\im({\alpha}\wedge)$ invariant.
\end{proof}

\begin{cor}\label{cor_codiffReeb}
    Let $X$ be as in Lemma \ref{lem_contactinvariantsplit}. If furthermore $\ker(\L_X)=\{0\}$, then $\delta=\iota_X$ is an acyclic and transversal general codifferential in the sense of Definition \ref{def_codifferential_types}.
\end{cor}

Smoothly varying the contact form and its Reeb vector field gives rise to an inner variation of general codifferentials, as we show in the following Lemma.

\begin{lemma}
\label{contact_variation}
    Let $\tau\in(-1,1)\to \alpha(\tau)$ be a smooth family of contact forms on $M$ and $X(\tau)$ the associated Reeb vector fields. Then the family of general codifferentials $\iota_{X(\tau)}$ is integrable in the sense of Definition \ref{def_homotopy_structure}. In particular, $\tau \to \iota_{X(\tau)}$ forms an inner variation of general codifferentials.
\end{lemma}
\begin{proof}
    Since the vector fields ${X_\tau}$ are non-vanishing and depend smoothly on $\tau$, we can find a smooth family $S(\tau): TM\to TM$ of invertible endomorphism of the tangent bundle, such that
    \begin{equation*}
        S(\tau)(x){X_0}(x) = {X_\tau}(x), \qquad \forall\,x\in M,
    \end{equation*}
    see \cite[Section 4]{viet_dang_fried_conjecture}. Letting $S(\tau)^\top: T^*M \to T^*M$ denote the transpose of $S(\tau)$, we find for any $k$-form $\omega \in \Omega^k(M)$, and any vector fields $Y^1,\dots,Y^{k-1}$:
    \begin{equation*}
    \begin{split}
        \bigl(\wedge^{k-1}S(\tau)^\top\bigr)^{-1}&\circ\iota_{{X_0}}\circ\wedge^k S(\tau)^\top(\omega)(Y^1,\dots,Y^{k-1}) \\
        &= \wedge^k S(\tau)^\top(\omega)({X_0},S(\tau)^{-1}Y^1,\dots,S(\tau)^{-1}Y^{k-1}) \\
        &= \omega(S(\tau){X_0},Y^1,\dots,Y^{k-1}) = \omega({X_\tau},Y^1,\dots,Y^{k-1}) = \iota_{X(\tau)}(\omega)(Y^1,\dots,Y^{k-1}).
    \end{split}
    \end{equation*}
    Therefore,
    $$\iota_{{X_\tau}} = \bigl(\wedge^{k-1}S(\tau)^\top\bigr)^{-1}\circ\iota_{{X_0}}\circ\wedge^k S(\tau)^\top,\qquad\mathrm{on}\quad \Omega^k(M).$$
    Defining the degree-preserving bundle automorphism 
    $${\beta}_{\tau} = \wedge^\bullet \bigl(S(\tau)^\top\bigr)^{-1}\otimes 1 : \wedge^\bullet T^*M\otimes E \to \wedge^\bullet T^*M\otimes E,$$
    we see that the general codifferentials satisfy $\iota_{{X_\tau}} = \beta_\tau\circ\iota_{{X_0}}\circ\beta_{\tau}^{-1}$ on $\Omega^\bullet(M,E)$, i.e. they form an integrable family. In particular, the family $\iota_{X(\tau)}$ is an inner variation, satisfying
    \begin{equation*}
        \frac{d}{d\tau}\iota_{X(\tau)} = [\theta_\tau,\iota_{X(\tau)}], \quad\text{with}\quad \theta_\tau = \bigl(\tfrac{d}{d\tau}\beta_\tau\bigr)\beta_\tau^{-1}.
    \end{equation*}
\end{proof}

\subsection{Anosov flows and the Ruelle zeta function}

We have seen that the Reeb vector field of a contact form defines a general codifferential $\iota_X$ with characteristic operator $\L_X$. In general $\iota_X$ will not be regular, i.e. $\L_X$ will not have a well-defined flat superdeterminant. We thus restrict to a special class of vector fields where the flat superdeterminant can be obtained, namely when $X$ generates an Anosov flow on $M$.

\begin{defn}
\label{def_Anosov}
The flow $\phi_t\colon M\to M$ of a vector field $X\in\C(M,TM)$ is called Anosov if there exists a continuous splitting of the tangent bundle:
\begin{equation*}
    T_xM = E_0(x) \oplus E_s(x) \oplus E_u(x), \quad\text{with }\, E_0(x) = \R X(x),
\end{equation*}
which is invariant under $d\phi_t$, i.e.\ $d\phi_t(E_\bullet(x)) = E_\bullet(\phi_t(x))$ for $\bullet = 0, s, u$, and for some constants $C, \mu >0$, we have for all $t \ge 0$:
\begin{equation*}
    \forall v \in E_s(x):\, \|d\phi_t(x)v\| \le Ce^{-\mu t}\|v\|, \quad \forall v \in E_u(x):\, \|d\phi_{-t}(x)v\| \le Ce^{-\mu t}\|v\|.
\end{equation*}
\end{defn}
Here $\|\cdot\|$ is the norm induced by some Riemannian metric on $M$. The Anosov property, although not the specific values of the constants $C, \mu$, is independent of this choice of metric. 
$E_s$ is called the stable bundle and $E_u$ the unstable bundle. We will assume in addition that the the stable and unstable bundles are orientable.
When $X$ is Anosov, the decomposition in \eqref{contact_splitting} can also be viewed in terms of the stable and unstable bundles of the Anosov flow, in that the splitting in Definition \ref{def_Anosov} induces a dual splitting of the cotangent space: $T^*M = E_0^* \oplus E_s^* \oplus E_u^*$ and we have $E_0^* = \R\alpha$, $E_s^* \oplus E_u^* = T^*_0M$.

\begin{rmk}\label{rmk:gooddecompositionaxial}
We observe that the operator $\L_X=[\iota_X,\d]$ is not very well behaved on the space of \emph{smooth} differential forms. In particular, $\ker(\L_X)$ should be viewed in terms of the Pollicott-Ruelle resonances of \eqref{resolvent_continuation} below and hence as a subspace of the anisotropic Sobolev spaces $H_{sG}$, see \cite{zworski}. So one cannot expect to get the Hodge-like decomposition\footnote{This also holds for Morse--Smale flows.}
\[
\Omega^\bullet(M,E) \not\simeq \im(\iota_X) \oplus \im(\d)\oplus \ker(\L_X),
\]
though see \cite{dang_riviere} for a cohomological perspective on Pollicott-Ruelle resonant states.
Even when $\ker(\L_X)$ is trivial, $\L_X$ is generally not surjective when acting on $\Omega^\bullet(M,E)$ and the splitting above does not hold in the strict sense, although a smooth differential form can then be written as a sum of distributional forms in $\im(\iota_X)$ and $\im(\d)$, see \cite{Schiavina_Stucker2}.
Compare this observation with Remark \ref{rmk:codiff_GF}, where an alternative definition of codifferential is considered. What we use in the contact Anosov case is the decomposition provided by the complementary acyclic complex $\alpha\wedge \circlearrowright \Omega^\bullet(M,E)$, so that
\[
\Omega^\bullet(M,E) = \im(\iota_X) \oplus \im(\alpha\wedge),
\]
regardless of the properties of $\ker(\L_X)$.
\end{rmk}

\begin{rmk}
The scenario we will consider in what follows is given by $X$ both Reeb and Anosov. Our main example is provided by the geodesic vector field on the unit cotangent bundle: When $(N,g)$ is a compact Riemannian manifold with strictly negative sectional curvature, then the geodesic flow on the cosphere bundle $S^*N$ is Anosov and its generator is the Reeb vector field for the contact form given by pulling back the tautological one-form on $T^*N$. 
\end{rmk}

The Anosov property has many implications on the structure of the orbits of $\phi_{t}$. We will mostly be interested in the set of closed, or periodic, orbits $\mathcal{P}$. We denote the period of a closed orbit $\gamma \in \P$ by $T_\gamma$. Definition \ref{def_Anosov} together with the compactness of $M$ gives
\begin{equation}
\label{min_period}
    \exists\, T_0>0, \quad\text{such that}\quad T_\gamma > T_0, \quad \forall\,\gamma\in\P.
\end{equation}
Moreover, see \cite[Lemma 2.2]{zworski}, one can show that $\P$ is countable and obtain a bound on the growth of the periods of closed orbits:
there exists some $C,L>0$, such that the number of closed orbits with period less than $T$ satisfies
\begin{equation}
\label{period_growth}
    \bigl|\bigl\{\gamma \in \P \,\mid\, T_\gamma \le T \bigr\}\bigr| \leq Ce^{(2n-1)LT}, \quad \forall\, T>0.
\end{equation}

When $X$ is Anosov, the twisted Lie derivative $\L_X=[\iota_X,\d]$ has discrete spectrum on certain anisotropic Sobolev spaces $H_{sG}$. Roughly speaking, these are spaces of distributions with Sobolev regularity $s$ in the stable directions and $-s$ in the unstable directions, see \cite[Section 3.1]{zworski} for a precise definition. By \cite[Proposition 3.3]{zworski}, for any $C>0$ one can choose $s$ large enough so that the resolvent
\begin{equation}
\label{resolvent_continuation}
    R_X(\lambda) = (\L_X+\lambda)^{-1}: H_{sG}\to H_{sG},
\end{equation}
defines a meromorphic family of operators in $\{\Re(\lambda) > -C\}$, see also \cite{butterley_liverani,faure_sjostrand}. The poles of this meromorphic extension are called \emph{Pollicott-Ruelle resonances}.

\begin{defn}
\label{def_regular_anosov}
    We will say that an Anosov vector field $X$ is \emph{regular}, if $\lambda=0$ is not a Pollicott-Ruelle resonance of the twisted Lie derivative $\L_X$.
\end{defn}

\noindent Note that for regular $X$ the resolvent $R_X(\lambda)$ is analytic in a neighborhood of $\lambda=0$.

Another important object associated to an Anosov flow is its Ruelle zeta function.
Note that the flat connection $\nabla$ on $E$ induces a unitary representation of the fundamental group,
$$\rho \colon \pi_1(M) \to U(\mathbb{C}^r),$$
where $r$ is the rank of $E$ and $\rho([\gamma])$ is defined to be the parallel transport map with respect to $\nabla$ along a representative $\gamma$ of $[\gamma] \in \pi_1(M)$, which is independent of the basepoint up to conjugacy. The Ruelle zeta function is defined as a product over the set of primitive closed orbits $\mathcal{P}^\#$ of the Anosov flow generated by $X$, akin to how the Riemann Zeta function is given as a product over the prime numbers. It is a function $\zeta_{X,\rho}(\lambda)$ of a complex parameter $\lambda$ and can be written for $\Re(\lambda) \gg 1$ as the following convergent product.\footnote{Convergence for large $\Re(\lambda)$ can be justified using the bound in \eqref{period_growth}.}
\begin{defn}
The \emph{Ruelle zeta function} is defined in $\Re(\lambda)\gg 1$ by
\begin{equation}
\label{ruelle_zeta_function}
    \zeta_{X,\rho}(\lambda) = \prod_{\gamma\in\mathcal{P}^\#} \mathrm{det}\bigl(1-\rho([\gamma])e^{-\lambda T_\gamma}\bigr).
\end{equation}
\end{defn}
\noindent It is a nontrivial result of \cite{GiuliettiLiveraniPollicott} that the Ruelle zeta function has a meromorphic extension to $\mathbb{C}$, see also \cite{zworski} for the perspective followed here.

\medskip

When $X$ is both Anosov and contact, we can think of $\delta=\iota_X$ as a general codifferential, so that $L_X \doteq \im(\delta)$ is an isotropic subspace, and the Lie derivative is the associated characteristic operator $D=[\delta,\d]=\L_X$.
At this stage, it is not at all clear that either the flat determinant of $\L_X$, or the value $\zeta_{X}(0)$, should actually be well-defined. In what follows we will present a useful, alternative, description of the Ruelle zeta function in terms of flat traces, following \cite{zworski}.

Note that $\L_X$ is the generator for the pull-back by the flow of $X$, i.e. $e^{-t\L_X}=\phi_{-t}^*$. The Schwartz kernel $K(t,x,y)$ of $\phi_{-t}^*$ is supported on the smooth submanifold $\{y=\phi_{-t}(x)\}\subset \R_+\times M\times M$; in fact, in local coordinates its components are smooth multiples of the delta function on this surface. Thus, its wavefront set lies in the conormal bundle to this submanifold, see \cite[Theorem 8.2.4]{hormander1},
\begin{equation}
\label{wavefront_anosov}
    \WF(K) \subseteq \bigl\{(t,-\xi\cdot X(x), x, \xi, \phi_{-t}(x), -d\phi_{t}(\phi_{-t}(x))^\top\cdot \xi) \,\,|\,\, t\in\R_+, (x,\xi)\in T^*M\bigr\}.
\end{equation}
The Anosov property ensures that this set is disjoint from the conormal to the diagonal map $\iota:(t,x)\to(t,x,x)$, that is
\begin{equation}
\label{wavefront_condition}
    \WF(K) \cap \bigl\{(t,0,x,\xi,x,-\xi) \,\,|\,\, t\in\R_+, (x,\xi)\in T^*M\bigr\} = \emptyset.
\end{equation}
Indeed, if the left hand side were non-empty, then there would be a non-zero $t$ and $(x,\xi)\in T^*M$ satisfying
\begin{equation*}
    \xi\cdot X(x) = 0, \quad x=\phi_t(x), \quad \xi = d\phi_t(x)^\top\cdot \xi.
\end{equation*}
The first equation implies that $\xi \in E_s^*(x)\oplus E_u^*(x)$ and by iterating the third equation we see that $\xi = d\phi_{Nt}(x)^\top\cdot \xi$ for all $N\in\Z$. Since $d\phi_t(x)$ preserves the splitting, writing $\xi = \xi_s + \xi_u$ with $\xi_s\in E_s^*(x)$ and $\xi_u\in E_u^*(x)$, we find $\xi_s = d\phi_{Nt}(x)^\top\cdot\xi_s$ for all $N$. Taking $N \to\infty$ and using the Anosov property shows that $\xi_s=0$. Similarly, we obtain $\xi_u=0$ by taking $N\to -\infty$ (note that the dual stable and unstable bundles exhibit expanding/contracting behavior similar to Definition \ref{def_Anosov}). However, by definition, the wavefront set is disjoint from the zero section, so we obtain a contradiction.

By \eqref{wavefront_condition} the flat trace of $e^{-t\L_X}$ is well-defined as a distribution on $\R_+$. Moreover, using the projection operator $\iota_X\circ\alpha\wedge$ to project onto the image of $\iota_X$, we can define the flat super trace restricted to $\im(\iota_X)=\Omega_0^\bullet(M,E)$. This flat trace in fact has an explicit expression as a sum of delta functions known as the Guillemin trace formula, see \cite[Theorem 8]{guillemin_trace_formula} or \cite[Appendix B]{zworski} for a proof.
\begin{lemma}[Guillemin Trace Formula]
\label{guillemin_thm}
The flat trace of $\exp(-t\L_X|_{\Omega_0^k(M,E)})$ is well-defined as a distribution on $\R_+$ when $X$ is Anosov, and is given by
\begin{equation}
\label{guillemin}
    \tr\big(e^{-t\L_X}\big|_{\Omega_0^k(M,E)}\big) = \sum_{\gamma\in\mathcal{P}} \frac{T^\#_\gamma\, \mathrm{tr}\big(\wedge^kP_\gamma\big)\,\mathrm{tr}\big(\rho([\gamma])\big)}{|\mathrm{det}(I-P_\gamma)|}\, \delta(t-T_\gamma).
\end{equation}
\end{lemma}

The sum above is over all closed orbits $\gamma(t) = \phi_t(x_0)$ of the Anosov flow, i.e.\ $\phi_{T_\gamma}(x_0) = x_0$ with $T_\gamma$ the period of $\gamma$. Thus, $\mathcal{P}$ does not only include the primitive orbits in $\mathcal{P}^\#$, but also all iterations of these. $P_\gamma = d\phi_{-T_\gamma}(x_0)|_{E_s\oplus E_u}$ is the linearised Poincaré map of the orbit and $T^\#_\gamma$ denotes the primitive period of $\gamma$, that is the smallest $t$ such that ${\phi_t(x_0)=x_0}$. Note that the expressions involving $P_\gamma$ in \eqref{guillemin} are well-defined, i.e. independent of the choice of basepoint $x_0$, since $d\phi_{-T_\gamma}(\phi_s(x_0))$ is conjugate to $d\phi_{-T_\gamma}(x_0)$ by $d\phi_s(x_0)$.
Furthermore, thanks to the Anosov property, $I-P_\gamma$ is invertible, so the determinant in the denominator is non-zero. Indeed, a non-zero $v\in E_s\oplus E_u$ with ${v=d\phi_{-T_\gamma}(x_0)v}$ also satisfies $v=d\phi_{NT_\gamma}(x_0)v$ for all $N\in\mathbb{Z}$, in contradiction to Definition \ref{def_Anosov}.

Lemma \ref{guillemin_thm} together with the bounds on the number of closed orbits, \eqref{min_period} and \eqref{period_growth}, imply that the limit
\begin{equation}
\label{aux_func_ruelle}
\begin{split}
    F^{(k)}(\lambda,s) &= \frac{1}{\Gamma(s)}\lim_{N\to\infty}\pair{\tr\left(e^{-t{\L_X}}\big|_{\Omega_0^k(M,E)}\right)}{t^{s-1}e^{-\lambda t}\,\chi_N(t)} \\
    &= \frac{1}{\Gamma(s)}\sum_{\gamma\in\mathcal{P}} \frac{T^\#_\gamma\, \mathrm{tr}\big(\wedge^kP_\gamma\big)\,\mathrm{tr}\big(\rho([\gamma])\big)T_\gamma^{s-1}e^{-\lambda T_\gamma}}{|\mathrm{det}(I-P_\gamma)|}
    \end{split}
\end{equation}
converges to an analytic function for all $s\in\mathbb{C}$ and $\Re(\lambda)>C$ for some constant $C$ large enough. Indeed, $\rho$ is a unitary representation, so $\mathrm{tr}(\rho([\gamma]))$ is bounded by a constant independent of $\gamma$. Furthermore, $|\mathrm{det}(1-P_\gamma)|$ is bounded away from zero for all $\gamma\in\P$. This can be seen by writing any $v\in E_s(x_0)\oplus E_u(x_0)$ as $v=v_s+v_u$ with $v_s\in E_s(x_0)$ and $v_u\in E_u(x_0)$ and noting that for all $T_\gamma$ large enough, we have
\begin{equation*}
\begin{split}
    \|(1-P_\gamma)v\| = \|v-d\phi_{-T_\gamma}(x_0)v\| &\geq c\bigl(\|v_u-d\phi_{-T_\gamma}(x_0)v_u\| + \|v_s-d\phi_{-T_\gamma}(x_0)v_s\|\bigr) \\
    &\geq c\bigl(\|v_u\| - \|d\phi_{-T_\gamma}(x_0)v_u\| + \|d\phi_{-T_\gamma}(x_0)v_s\| - \|v_s\|\bigr) \\
    &\geq c\bigl(\tfrac{1}{2}\|v_u\| + \|v_s\|\bigr) \geq \tfrac{c}{2}\|v\|,
\end{split}
\end{equation*}
where we used the Anosov property to estimate $\|d\phi_{-T_\gamma}(x_0)v_u\| \leq \frac{1}{2}\|v_u\|$ and $\|v_s\| \leq \frac{1}{2}\|d\phi_{-T_\gamma}(x_0)v_s\|$ for $T_\gamma$ large enough. This shows that the norm of $\|(1-P_\gamma)^{-1}\|$ is bounded uniformly in $\gamma$ and thus the determinant in \eqref{aux_func_ruelle} is bounded uniformly from below. Finally, by compactness we have $\sup\{\|d\phi_{-t}(x)\| \,|\, x\in M,t\in[0,1]\}<\infty$, so using the semigroup property of $d\phi_{-t}$, we obtain an exponential bound $\|d\phi_{-t}(x)\|\leq C_0e^{\beta t}$, leading to $|\mathrm{tr}(\wedge^k P_\gamma)| \leq Ce^{\beta T_\gamma}$ for some $\beta\in\R$. Thus, using the growth bound on the number of resonances in \eqref{period_growth} and choosing $\Re(\lambda)$ large enough, the sum in \eqref{aux_func_ruelle} converges locally uniformly in $s$ and $\lambda$. Note that choosing $\Re(s)> 0$ is not necessary for regularizing the limit in \eqref{aux_func_ruelle}, since $\tr(e^{-t\L_X}|_{\Omega_0^k(M,E)})$ vanishes for all $t\leq T_0$, i.e. smaller than the minimal period, see \eqref{min_period}, so there is no issue of convergence as $t\to 0$.

Since $F^{(k)}(\lambda,s)$ is analytic in a neighborhood of $s=0$, we can compute the derivative in \eqref{det_from_F} and take the alternating sum over the form degree to obtain the logarithm of the flat superdeterminant of $\L_X+\lambda$ restricted to $\Omega^\bullet_0(M,E)$:
\begin{align}
    \log\sdet\bigl((\L_X+\lambda)|_{\im(\iota_X)}\bigr) &= -\sum_{k=0}^n(-1)^k\frac{d}{ds}\Bigr|_{s=0}F^{(k)}(\lambda,s) = -\int_{T_0}^\infty t^{-1}\str\left(e^{-t(\L_X+\lambda)}\big|_{\im(\iota_X)}\right)\,dt \notag\\\label{logdet_guillemin}
    &= -\sum_{\gamma\in\mathcal{P}} \frac{T^\#_\gamma}{T_\gamma} \sum_{k=0}^n(-1)^k \frac{\mathrm{tr}\big(\wedge^kP_\gamma\big)\,\mathrm{tr}\big(\rho([\gamma])\big)}{|\mathrm{det}(I-P_\gamma)|}\,e^{-\lambda T_\gamma},
\end{align}
where we abused notation and wrote the distributional pairing in \eqref{aux_func_ruelle} as an integral.

When the stable bundle $E_s$ is orientable, we have (see \cite{zworski})
\[
|\mathrm{det}(I-P_\gamma)| = (-1)^{\dim(E_s)}\mathrm{det}(I-P_\gamma).
\]
In the contact case, we futher have $\dim(E_s)=\dim(E_u)=m$, where we write $n=2m+1$ for the dimension of $M$, see for instance \cite{faure_tsujii}. Thus, using the formula $\mathrm{det}(I-A) = \sum_k (-1)^k\mathrm{tr}(\wedge^kA)$, we can simplify the sum over $k$ in \eqref{logdet_guillemin} and relate this expression to the Ruelle zeta function in \eqref{ruelle_zeta_function}:
\begin{equation*}
\begin{split}
    \log\sdet\big((\L_X+&\lambda)|_{\im(\iota_X)}\big) = -(-1)^m\sum_{\gamma\in\mathcal{P}} \frac{T^\#_\gamma}{T_\gamma} \mathrm{tr}(\rho([\gamma]))\,e^{-\lambda T_\gamma} = -(-1)^m\sum_{\gamma\in\mathcal{P}^\#} \sum_{j=1}^\infty \frac{1}{j} \mathrm{tr}(\rho([\gamma])^j)\,e^{-j\lambda T_\gamma^\#} \\
    &= (-1)^m\sum_{\gamma\in\mathcal{P}^\#} \mathrm{tr}\bigl(-\sum_{j=1}^\infty \frac{1}{j} \rho([\gamma])^j \,e^{-j\lambda T_\gamma^\#}\bigr) = (-1)^m\sum_{\gamma\in\mathcal{P}^\#} \mathrm{tr}\bigl(\log(1-\rho([\gamma])\,e^{-j\lambda T_\gamma^\#})\bigr) \\ 
    &= (-1)^m \sum_{\gamma\in\mathcal{P}^\#} \log\mathrm{det}\bigl(1-\rho([\gamma])\,e^{-j\lambda T_\gamma^\#}\bigr) = (-1)^m \log\bigl(\zeta_{X,\rho}(\lambda)\bigr).
    \end{split}
\end{equation*}
Here, we split the sum over the closed orbits as a sum over the primitive closed orbits and all iterations of these, and then recognised the Taylor series for the logarithm. Thus, we see that for $\Re(\lambda)$ large enough
\begin{equation}
\label{flat_det_ruelle}
    \sdet\big((\L_X+\lambda)|_{\im(\iota_X)}\big) = \zeta_{X,\rho}(\lambda)^{(-1)^m}.
\end{equation}

The analytic continuation of the flat determinant in \eqref{flat_det_ruelle} to a meromorphic function on $\mathbb{C}$ was shown by Zworski and Dyatlov in \cite{zworski}. We provide a brief summary of this result. If we take the $\lambda$-derivative of \eqref{logdet_guillemin} in the domain of convergence and use the arguments of \cite[Section 4]{zworski} to exchange the order of the integral and the flat trace, we find at each form degree $k$:
\begin{equation}
\label{trace_resolvent}
\begin{split}
    \frac{d}{d\lambda}\log\det\big((\L_X+\lambda)|_{\Omega_0^k(M,E)}\big) &= \int_{T_0}^\infty \tr\left(e^{-t(\L_X+\lambda)}\big|_{\Omega_0^k(M,E)}\right)\,dt = \tr\left(\int_{T_0}^\infty e^{-t(\L_X+\lambda)}\,dt\big|_{\Omega_0^k(M,E)}\right) \\
    &= \tr\left(e^{-T_0(\L_X+\lambda)}\int_{0}^\infty e^{-t(\L_X+\lambda)}\,dt\big|_{\Omega_0^k(M,E)}\right) \\
    &= e^{-T_0\lambda}\, \tr\left(\phi_{-T_0}^*R_X(\lambda)\big|_{\Omega_0^k(M,E)}\right).
\end{split}
\end{equation}
Here, we used the formula $R_X(\lambda)= (\L_X+\lambda)^{-1} = \int_0^\infty e^{-t(\L_X+\lambda)}\,dt$ for the resolvent, which holds for $\Re(\lambda)$ large enough. By \cite[Proposition 3.3]{zworski} the flat trace of the pullback of the resolvent is well-defined for small enough $T_0$. 

Using the meromorphicity of the resolvent noted in \eqref{resolvent_continuation}, the flat trace above extends to a meromorphic function on $\mathbb{C}$ with poles corresponding to the Pollicott-Ruelle resonances. In fact, the expression in \eqref{trace_resolvent} has only simple poles with positive integer residues corresponding to the dimension of the space of resonant states, see \cite[Lemma 4.2]{zworski}. From this property it follows that $\det\big((\L_X+\lambda)|_{\Omega_0^k(M,E)}\big)$, the exponential of the antiderivative of \eqref{trace_resolvent}, defines an analytic function on $\mathbb{C}$ with zeros located at the Pollicott-Ruelle resonances, whose order of vanishing is the dimension of the resonant states. Note that the zeros in odd form degree become poles of the flat superdeterminant. We find in particular that if $X$ has no Pollicott-Ruelle resonance at $0$, then $\sdet\big((\L_X+\lambda)|_{\im(\iota_X)}\big)$ has a well-defined non-zero value at $\lambda=0$.

Summarizing the discussion in this section, we have:
\begin{prop}
\label{prop_ruelle_analytic}
    Let $X$ be a regular Anosov vector field, see Definition \ref{def_regular_anosov}, which is also the Reeb vector field of a contact form. Then $\delta=\iota_X$ is a regular general codifferential in the sense of Definition \ref{def_codifferential_types}. Moreover, the flat superdeterminant
    $$\sdet(\L_X|_{\im(\iota_X)}) = \zeta_{X,\rho}(0)^{(-1)^m}$$
    computes the (nonzero) value of the Ruelle zeta function $\zeta_{X,\rho}(\lambda)$ at $\lambda=0$.
\end{prop}

\subsection{Proof of local constancy of the value at zero of the Ruelle zeta function}

In this subsection, we show that the assumptions of Theorem \ref{thm_invariance} are satisfied for the family of general codifferentials $\delta_\tau=\iota_{X_\tau}$, where $X_\tau$ is a smooth family of regular contact Anosov vector fields in the sense of Definition \ref{def_regular_anosov}. This allows us to infer the local constancy of the regularized superdeterminant $\sdet(\L_{X_\tau}\vert_{\im(X_\tau)})$. Since the value at zero of the Ruelle zeta function is related to this superdeterminant by Proposition \ref{prop_ruelle_analytic}, we recover the results of \cite{viet_dang_fried_conjecture}, as a corollary of Theorem \ref{thm_invariance}, for a family of contact Anosov vector fields. Namely:

\begin{prop}[{\cite[Theorem 2]{viet_dang_fried_conjecture}}]
\label{invariance_ruelle}
    Let $\tau \in (-1,1) \to X_\tau$ be a smooth\footnote{smooth with respect to the Fréchet topology on $\C(M,TM)$} family of regular contact Anosov vector fields (Definition \ref{def_regular_anosov}). Then the Ruelle zeta functions associated to the $X_\tau$ satisfy
    $$\zeta_{X_\tau,\rho}(0) = \zeta_{X_0,\rho}(0), \quad \forall\,\tau\in(-1,1).$$
\end{prop}

\begin{rmk}
    Note that the result of \cite{viet_dang_fried_conjecture} holds more generally for Anosov vector fields that are not required to be contact. The contact structure ensures the existence of a smooth one-form $\alpha$ satisfying $\ker(\alpha) = E_s\oplus E_u$ and $\alpha(X)=1$. For a general Anosov vector field, one can of course define such a one-form, but it need only be H\"older continuous, see for instance \cite[Section 2.1.3]{faure_sjostrand}. Thus, the general case does not fit into our framework of general codifferentials, since the splitting in \eqref{contact_splitting} need not be smooth. With some extra work, it may be possible to extend our general local constancy result to the case of non-smooth isotropic splittings of the space of differential forms.
\end{rmk}

\begin{rmk}
Note that our approach to this local constancy result differs from the approach in \cite{viet_dang_fried_conjecture} in that we phrase everything in terms of flat traces and flat determinants. On the other hand, Dang et.\ al.\ start by taking the $\tau$-derivative of the expression for the Ruelle zeta function in terms of closed orbits, as in \eqref{ruelle_zeta_function}, and only reformulate the result in terms of flat traces to show the analytic continuation of this derivative. Our consistent use of the flat determinant allows us to place this result inside the larger framework of Section \ref{subsection_invariance}.
\end{rmk}

In the previous subsections, we saw that $\tau\to\iota_{X_\tau}$ satisfies the structural assumptions of Theorem \ref{thm_invariance}, i.e. defines an inner variation of general codifferentials, which are regular for every $\tau$ (Proposition \ref{prop_ruelle_analytic}). It remains to address the analytic requirements. Note that the smoothness of $\tau\to X_\tau$ w.r.t.\ the topology on $\C(M,TM)$ implies the smoothness of $(\tau,t,x)\to\phi^\tau_{-t}(x)$ for the flow generated by $X_\tau$. Thus, it can be seen that the semigroup generated by $\L_{X_\tau}$, namely the pullback map $e^{-t\L_{X_\tau}} = (\phi^\tau_{-t})^*$, is smooth with respect to the Fréchet topology on $\Omega^\bullet(M,E)$.

Denote the Schwartz kernel of $(\phi_{-t}^\tau)^*$ acting on $E$-valued $k$-forms by $K^{(k)}_\tau$. We begin by proving the differentiability of $\tau\to K^{(k)}_\tau \in \D_\Gamma(\R_+\times M\times M)$ with respect to the H\"ormander topology for some conic set $\Gamma$, showing the validity of Assumption \ref{assumption_diff_hormander}.

\begin{lemma}
\label{WF_uniform}
Let $\tau \in (-1,1) \to X_\tau$ be a smooth family of Anosov vector field. Then, for any compact interval $I\subset (-1,1)$, there exists a closed conic set $\Gamma \subset T^*(\R_+\times M\times M)$ with $\Gamma\cap N^*\iota = \emptyset$ and $\WF(K_\tau^{(k)})\subset\Gamma$ for all $\tau \in I$, such that 
$$\tau \to K^{(k)}_\tau \in \D_\Gamma(\R_+\times M\times M; (\wedge^kT^*M\otimes E)\boxtimes(\wedge^kT^*M\otimes E)^*\otimes \wedge^nT^*M)$$
is differentiable with respect to the H\"ormander topology for each $k$.
\end{lemma}

\begin{proof}
We first show that any point in $N^*\iota$ has an open conic neighborhood disjoint from the wavefront sets of $K_\tau^{(k)}$ for all $\tau\in I$. To make use of compactness, we identify conic sets with subsets of the cosphere bundle $S^*(\R_+ \times M \times M) = (T^*(\R_+ \times M \times M) \setminus\{0\}) / \R_+$, where we removed the zero section and quotiented out the action of $\R_+$ by dilation in the fibers. We define the cosphere bundle $S^*M$ similarly. Denote by $\pi_{\R_+}$ the projection $\pi_{\R_+}: S^*(\R_+ \times M \times M) \to \R_+.$

Given any point $p_0 = (t_0,0,x_0,\xi_0,x_0,-\xi_0) \in N^*\iota$
and some small $\epsilon>0$, consider the map, 
$$F: I\times [t_0-\epsilon,t_0+\epsilon]\times S^*M \to S^*(\R_+ \times M \times M),$$
defined by
$$F(\tau,t,x,\xi) = (t, -\xi\cdot X_\tau(x), x,\xi, \phi_{-t}^\tau(x), -d\phi_{t}^\tau(\phi_{-t}^\tau(x))^\top\cdot\xi).$$
By \eqref{wavefront_anosov}, the image of $F$ contains the wavefront set of $K_\tau^{(k)}$ over $[t_0-\epsilon,t_0+\epsilon]$ for each $\tau$, i.e.
$$\bigcup_{\tau\in I}\WF(K_\tau^{(k)})\cap\pi_{\R_+}^{-1}([t_0-\epsilon,t_0+\epsilon]) \subset \im(F).$$
Furthermore, $F$ is disjoint from $N^*\iota$ as we saw in the discussion following \eqref{wavefront_condition}. $F$ is continuous by the continuity of $X_\tau(x)$ and $\phi_{t}^\tau(x)$ with respect to $(\tau,t,x)$, so the image of $F$ is a compact subset of $S^*(\R_+\times M\times M)$, which does not intersect $N^*\iota$. Thus, we can find a neighborhood $U \subset S^*(\R_+ \times M \times M)$ of $p_0$ disjoint from $\im(F)$ and contained in $\pi_{\R_+}^{-1}((t_0-\epsilon,t_0+\epsilon))$.
It follows that the conic set generated by $U$ is a conic neighborhood of $p_0$ disjoint from $\bigcup_{\tau\in I}\WF(K_\tau^{(k)})$. Since $p_0$ was arbitrary, this shows that there is an open conic set,
$$\Lambda \subset T^*(\R_+ \times M \times M), \qquad\mathrm{with}\quad N^*\iota \subset \Lambda,\quad\mathrm{and}\quad \Lambda \cap \bigcup_{\tau\in I}\WF(K_\tau^{(k)})= \emptyset.$$
Taking the complement $\Gamma = \Lambda^c$ provides a closed cone with $K_\tau^{(k)} \in \D_\Gamma(\R_+\times M\times M)$ for all $\tau\in I$.

Note that $\Lambda = \Gamma^c$ can be chosen arbitrarily small around $N^*\iota$. As the projection of $N^*\iota$ onto $M\times M$ is the diagonal, we can choose $\Lambda$ small enough, so that $\pi_{M\times M}(\Lambda) \subset \bigcup_i(U_i\times U_i)$, where the $U_i$ are coordinate neighborhoods.

We must now show that for $\tau$ in the interior of $I$ the difference quotients
$$\tfrac{1}{h}(K^{(k)}_{\tau+h} - K^{(k)}_\tau) \in \D_\Gamma(\R_+\times M\times M)$$
converge in the H\"ormander topology as $h\to 0$. Since the pullback maps $(\phi_{-t}^\tau)^*$ are smooth on $\Omega^k(M,E)$ with respect to $\tau$, we certainly have weak convergence in $\D(\R_+\times M\times M)$. It remains to show that the difference quotients are bounded with respect to the H\"ormander seminorms $||\cdot||_{N, \chi, V}$ as $h\to 0$, see \cite[Definition 8.2.2]{hormander1}.
Here, the seminorms are given in local coordinates by
$$\|u\|_{N,\chi,V} \doteq \sup_{\xi\in V} |\xi|^N |\mathcal{F}(\chi u)(\xi)|,$$
where $\mathcal{F}$ denotes the Fourier transform. They range over all $N\in\N$ and all $\chi\in\C_c(\R_+\times M\times M)$, $V\subset\R^{2n+1}$ closed conic such that $\supp(\chi)\times V \subset \Lambda = \Gamma^c$. In fact, it is enough to check boundedness for $\chi$ and $V$ with $\supp(\chi)\times V$ forming a cover of $\Lambda$, see \cite[p. 80]{grigis_sjostrand}.

By our choice of $\Lambda$, we must only consider $\chi$ with $\supp(\chi) \subset \R_+ \times U \times U$ for some coordinate neighborhood $U \subset M$. Let $\{\nu_1,\dots,\nu_r\}$ be a local frame for the bundle $\wedge^kT^*M\otimes E$ over $U$. The components of the Schwartz kernel $K^{(k)}_\tau$ with respect to this frame can be written as
\begin{equation}
\label{kernel_components}
    K^{(k)}_\tau(t,x,y)_{ij} = b_{ij}^\tau(t,x) K_\tau(t,x,y),
\end{equation}
where the $b_{ij}^\tau$ are smooth functions and $K_\tau(t,x,y)$ is the Schwartz kernel of $(\phi_{-t}^\tau)^*$ acting on $\C(M)$, i.e. a delta function on $\{y=\phi_{-t}(x)\}$.

We first consider the H\"ormander seminorms for $K_\tau(t,x,y) \in \D_\Gamma(\R_+ \times M\times M)$, since the case of the full Schwartz kernel $K^{(k)}_\tau(t,x,y)$ will follow easily from this. Taking the Fourier transform in local coordinates of $\chi$ times the difference quotient gives
\begin{equation}
\label{hormander_differentiability}
\begin{split}
    &\mathcal{F}\big(\chi\frac{1}{h}(K_{\tau+h}-K_\tau)\big)(\omega,\xi,\eta) = \frac{1}{h}\pair{K_{\tau+h}(t,x,y)-K_\tau(t,x,y)}{\chi(t,x,y) e^{-i(\omega t + \xi x + \eta y)}} \\
    &= \frac{1}{h}\int_0^\infty\int_{\R^n} \big( \chi(t,x,\phi_{-t}^{\tau+h}(x))e^{-i(\omega t + \xi x + \eta \phi_{-t}^{\tau+h}(x))} - \chi(t,x,\phi_{-t}^{\tau}(x))e^{-i(\omega t + \xi x + \eta\phi_{-t}^{\tau}(x))}\big)\,dx\,dt \\
    &= \frac{1}{h}\int_\tau^{\tau+h}\int_0^\infty\int_{\R^n} \frac{d}{d\sigma}\big( \chi(t,x,\phi_{-t}^\sigma(x))e^{-i(\omega t + \xi x + \eta \phi_{-t}^\sigma(x))}\big)\,dx\,dt\,d\sigma \\
    &= \frac{1}{h}\int_\tau^{\tau+h}\int_0^\infty\int_{\R^n} (\frac{d}{d\sigma}\phi_{-t}^\sigma(x)\cdot\nabla_y\chi(t,x,y)|_{y=\phi_{-t}^\sigma(x)})e^{-i(\omega t + \xi x + \eta \phi_{-t}^\sigma(x))}\,dx\,dt\,d\sigma \\
    &- i\frac{1}{h}\int_\tau^{\tau+h}\int_0^\infty\int_{\R^n} (\eta\cdot\frac{d}{d\sigma}\phi_{-t}^\sigma(x))\chi(t,x,\phi_{-t}^\sigma(x)))e^{-i(\omega t + \xi x + \eta \phi_{-t}^\sigma(x))}\,dx\,dt\,d\sigma.
\end{split}
\end{equation}
Write $(\omega,\xi,\eta) \in \R^{2n+1}$ as $\lambda (\hat{\omega},\hat{\xi},\hat{\eta})$, where $\lambda = (|\omega|^2+|\xi|^2+|\eta|^2)^{\frac{1}{2}}$ and $(\hat{\omega},\hat{\xi},\hat{\eta}) \in S^{2n}$ is the corresponding unit vector. Then we can write the two integrals over $(t,x)$ in the last line of equation \eqref{hormander_differentiability}, as
\begin{equation*}
    \int_{\R^{n+1}} u_j^\sigma(t,x) e^{-i\lambda f^\sigma(t,x)}\,dx\,dt, \quad j \in \{1,2\}
\end{equation*}
for
\begin{equation*}
    f^\sigma(t,x) = \hat{\omega}t + \hat{\xi}\cdot x + \hat{\eta}\cdot \phi_{-t}^\sigma(x)
\end{equation*}
and
\begin{equation}
\label{u_formula}
    u_1^\sigma(t,x) = \frac{d}{d\sigma}\phi_{-t}^\sigma(x)\cdot\nabla_y\chi(t,x,y)|_{y=\phi_{-t}^\sigma(x)}, \quad
    u_2^\sigma(t,x) = \frac{\eta}{|\eta|}\cdot\frac{d}{d\sigma}\phi_{-t}^\sigma(x) \chi(t,x,\phi_{-t}^\sigma(x)),
\end{equation}
where we factored out the norm of $\eta$ in the last term of \eqref{hormander_differentiability}.

The gradient of the phase factor satisfies
\begin{equation}
\label{grad_f_formula}
    \nabla_{(t,x)} f^\sigma(t,x) = (\hat{\omega}-\hat{\eta}\cdot X_\sigma(\phi_{-t}^\sigma(x)),\,\, \hat{\xi} + d\phi_{-t}^\sigma(x)^\top \cdot \hat{\eta}).
\end{equation}
Inspecting the definition of the $u_j^\sigma$ in \eqref{u_formula}, we see that 
$$\supp(u_j^\sigma) \subset \{(t,x)\,|\, (t,x,\phi_{-t}^\sigma(x)) \in \supp(\chi)\}.$$
Recall that $\chi\in \C_c(\R_+\times M\times M)$ has support close to the diagonal in $M\times M$. At the cost of shrinking $\Lambda$ and choosing $h$ small enough, we can assume without loss of generality that there is a point $(t_0,x_0) \in \supp(u_j^\tau)$ satisfying $\phi_{-t_0}^\tau(x_0)=x_0$. Otherwise, we can choose $\supp(\chi)$ small enough so that it does not contain a point of the form $(t,x,\phi_{-t}^\sigma(x))$ for any $\sigma\in[\tau,\tau+h]$ and hence $u_j^\sigma=0$ for each $\sigma$.

Thus, the $u_j^\sigma$ are supported around a periodic point $(t_0,x_0)$ of the flow of $X_\tau$. We must check the H\"ormander seminorms for a closed cone $V \subset \R^{2n+1}$, such that when $\Lambda = \Gamma^c$ is viewed in local coordinates, we have $\supp(\chi)\times V \subset \Lambda$. Since $\Lambda$ is an arbitrarily small conic neighborhood of $N^*\iota$, the unit vectors $(\hat{\omega},\hat{\xi},\hat{\eta}) \in V$ are close to $(0,\hat{\xi}_0,-\hat{\xi}_0)$ for some $\xi_0\in\R^n$. As we saw in the arguments following \eqref{wavefront_condition}, for $(t_0, x_0)$ as above, we have
$$(-\hat{\xi}_0\cdot X_\tau(x_0), (1-d\phi_{t_0}^\tau(x_0)^\top)\cdot \hat{\xi}_0) \neq 0.$$
Inspecting equation \eqref{grad_f_formula}, we see that this is nothing but $\nabla f^\tau(t_0,x_0)$ for $(\hat{\omega},\hat{\xi},\hat{\eta}) = (0,\hat{\xi}_0,-\hat{\xi}_0)$. Since $f^\sigma$ depends continuously on $\tau$, $(t,x)$ and $(\hat{\omega},\hat{\xi},\hat{\eta})$, we can choose $h$, $\supp(\chi)$ and $V$ small enough so that for some $\epsilon>0$ and every $\sigma \in [\tau,\tau+h]$,
$$|\nabla f^\sigma(t,x)| > \epsilon$$
for all $(t,x)\in\supp(u_j^\sigma)$ and all unit vectors $(\hat{\omega},\hat{\xi},\hat{\eta}) \in V$.

We are now in a position to apply non-stationary phase estimates, see \cite[Theorem 7.7.1]{hormander1}.
For every $N\in\N$, we find for some $C_N>0$:
\begin{equation}
\label{non_stationary_phase}
    |\lambda|^N\,\Big|\int_{\R^{n+1}} u_j^\sigma(t,x) e^{-i\lambda f^\sigma(t,x)}\,dx\,dt\Big| \le C_N \sum_{|\alpha|\le N}\frac{\sup_{(t,x)}|\partial^\alpha u_j^\sigma(t,x)|}{\inf_{(t,x)\in\supp(u_j)}|\nabla f^\sigma(t,x)|^{2N-|\alpha|}}.
\end{equation}
As we have shown, the denominator satisfies
$$\inf_{(t,x)\in\supp(u_j)}|\nabla f^\sigma(t,x)|^{2N-|\alpha|} \ge \epsilon^{2N-|\alpha|}$$
uniformly for $\sigma \in [\tau,\tau+h]$.
The functions $u_j^\sigma(t,x)$ involve the flow $\phi_{t}^\sigma(x)$ and its derivative $\frac{d}{d\sigma}\phi_{t}^\sigma(x)$. Since $\phi_{t}^\sigma(x)$ depends smoothly on $\sigma$ we can bound $u_j^\sigma(t,x)$ on the compact set 
$$(\sigma,t,x) \in [\tau,\tau+h]\times\pi_{\R_+\times M}(\supp(\chi)),$$
to get an estimate
$$\sup_{(t,x)}|\partial^\alpha u_j^\sigma(t,x)| < C$$
uniformly for $\sigma \in [\tau,\tau+h]$. Furthermore, the constants $C_N$ in \eqref{non_stationary_phase} are polynomials in 
$$\sup_{(t,x)\in\supp(u_j)}|\partial^\beta f^\sigma(t,x)|,$$
for $2 \le |\beta| \le N+1$, see \cite[Theorem 7.7.1]{hormander1}.
Again, smoothness of $\phi_{t}^\sigma(x)$ implies that these constants can be chosen uniformly for $\sigma\in[\tau,\tau+h]$.

Recalling that $\lambda$ is the norm of $(\omega,\xi,\eta)$, we have found that the two oscillatory integrals can be estimated for all $(\omega,\xi,\eta) \in V$ by
\begin{equation*}
    |(\omega,\xi,\eta)|^N\,\Big|\int_{\R^{n+1}} u_j^\sigma(t,x) e^{-i\lambda f^\sigma(t,x)}\,dx\,dt\Big| \le C_j^{(N)}
\end{equation*}
for constants $C_j^{(N)}$ independent of $\sigma \in [\tau,\tau+h]$. Using these estimates in equation \eqref{hormander_differentiability} and denoting the two oscillatory integrals over the $u_j$ by $I_j(\sigma)$, we find
\begin{equation*}
\begin{split}
    |(\omega,\xi,\eta)|^N &\big|\mathcal{F}\big(\chi\frac{1}{h}(K_{\tau+h}-K_\tau)\big)(\omega,\xi,\eta)\big| \le |(\omega,\xi,\eta)|^N\frac{1}{h}\Big(\Big|\int_\tau^{\tau+h} I_1(\sigma) \,d\sigma\Big| + |\eta|\Big|\int_\tau^{\tau+h} I_2(\sigma) \,d\sigma\Big|\Big) \\
    &\le \frac{1}{h}\int_\tau^{\tau+h} |(\omega,\xi,\eta)|^N |I_1(\sigma)| \,d\sigma + \frac{1}{h}\int_\tau^{\tau+h} |(\omega,\xi,\eta)|^{N+1} |I_2(\sigma)| \,d\sigma \le C_1^{(N)} + C_2^{(N+1)},
\end{split}
\end{equation*}
for all $(\omega,\xi,\eta) \in V$. Thus, the H\"ormander seminorm $\|\frac{1}{h}(K_{\tau+h}-K_\tau)\|_{N,\chi,V}$ remains bounded as $h\to 0$.

To show the same result for the full Schwartz kernel $K_\tau^{(k)}$, note that by \eqref{kernel_components} the components of $K_\tau^{(k)}$ with respect to the local frame are given by multiplying the kernel $K_\tau$ with the smooth functions $b_{ij}^\tau(t,x)$. The sections in the local frame $\{\nu_1,\dots,\nu_r\}$ are tensor products of local sections of $\wedge^k T^*M$ and $E$. Note that the pullback $(\phi_{-t}^\tau)^*$ acts in the fibers of $\wedge^k T^*M$ by $\wedge^k d\phi_{-t}^\tau(x)^\top$ and in the fibers of $E$ by the parallel transport map $\mathbb{P}^\tau$ along the path $s\in [0,t] \to \phi_{-s}^\tau(x)$. The $b_{ij}^\tau(t,x)$ are just products of the components of these two maps. Since the flow $\phi_{-t}^\tau(x)$ depends smoothly on $\tau$ and the parallel transport map depends smoothly on the path, the component functions $b_{ij}^\tau(t,x)$ are smooth in $\tau$. Now, computing the Fourier transform of the components of $K_\tau^{(k)}$ as in \eqref{hormander_differentiability} leads to the same oscillatory integrals as above, except that the functions $u_j^\sigma(t,x)$ in \eqref{u_formula} now involve factors of $b_{ij}^\sigma(t,x)$ and $\frac{d}{d\sigma}b_{ij}^\sigma(t,x)$. Since these factors depend smoothly on $\sigma$, the oscillatory integrals can still be estimated uniformly in $\sigma$.
\end{proof}

We now address the remaining assumptions of Theorem \ref{thm_invariance}

\begin{proof}[Proof of Proposition \ref{invariance_ruelle}]
    By Lemma \ref{contact_variation}, the family of general codifferentials $\delta_\tau = \iota_{X_\tau}$ forms an inner variation. By Proposition \ref{prop_ruelle_analytic} each $\delta_\tau$ is regular. We have also seen that the semigroup $e^{-t\L_{X_\tau}}=(\phi_{-t}^\tau)^*:\Omega^\bullet(M,E)\to\Omega^\bullet(M,E)$ depends smoothly on $\tau$, so Assumption \ref{assumption_diff_semigroup} is satisfied. Lemma \ref{WF_uniform} shows that Assumption \ref{assumption_diff_hormander} is satisfied.
    
    We consider Assumption \ref{assumption_smaller_wavefront}. By \eqref{wavefront_anosov}, for each $\tau$ the wavefront set of the Schwartz kernel of $e^{-t\L_{x_\tau}}$ satisfies
    \begin{equation}
    \label{wavefront_anosov2}
    \begin{split}
        \WF(K_\tau) &\subseteq \bigl\{(t,-\xi\cdot X_\tau(x), x, \xi, \phi^\tau_{-t}(x), -d\phi^\tau_{t}(\phi^\tau_{-t}(x))^\top\cdot \xi) \,\,|\,\, t\in\R_+,\, (x,\xi)\in T^*M\bigr\} \\
        &= \bigl\{(t,-\xi\cdot X_\tau(\phi^\tau_{t}(y)), \phi^\tau_{t}(y), \xi, y, -d\phi^\tau_{t}(y)^\top\cdot \xi) \,\,|\,\, t\in\R_+,\, y\in M,\, \xi\in T_{\phi^\tau_{t}(y)}^*M\bigr\},
    \end{split}
    \end{equation}
    where in the second line we rewrote the set using $y=\phi^\tau_{-t}(x)$. Proceeding by contradiction, assume there exists some
    $$(t,0,x,\xi,y,-\eta)\in\WF(K_\tau) \quad\text{with}\quad (s,0,y,\eta,x,-\xi)\in\WF(K_\tau) \quad\text{for some } s\in \R_+.$$
    Then by \eqref{wavefront_anosov2}, we must have
    \begin{equation*}
        \xi\cdot X_\tau(x) = 0, \quad x = \phi^\tau_t(y) \,\,\text{ and }\,\, y = \phi^\tau_s(x), \quad \xi = d\phi_s(x)^\top\cdot\eta \,\,\text{ and }\,\, \eta = d\phi_t(x)^\top\cdot\xi.
    \end{equation*}
    In other words
    $$\xi \in E_s^*(x)\oplus E_u^*(x), \quad x = \phi^\tau_{t+s}(x), \quad \xi = d\phi^\tau_{t+s}(x)^\top\cdot\xi.$$
    But by the arguments following \eqref{wavefront_condition} this is impossible for an Anosov flow.

    Turning to Assumption \ref{assumption_G_convergencce}, we can obtain a type of Guillemin trace formula, see Lemma \ref{guillemin_thm}, for the flat trace of the composition $\theta_\tau e^{-t\L_{X_\tau}}$ acting on $E$-valued $k$-forms. Note that the operator $\theta_\tau:\wedge^k T^*M\otimes E \to \wedge^k T^*M\otimes E$ of Lemma \ref{contact_variation} is just a bundle endomorphism depending smoothly on $\tau$. Thus, for a coordinate neighborhood $U$ around $x_0\in M$, where $x_0=\phi^\tau_{t_0}(x)$ for some $t_0$, and a local frame on $U$, the components of the Schwartz kernel of the composition are given by
    $$\bigl(\theta_\tau K_\tau^{(k)}\bigr)_{ij}(t,x,y) = \sum_{k}\theta^\tau_{ik}(y)b^\tau_{kj}(t,y)K_\tau(t,x,y),$$
    where $K_\tau$ is the Schwartz kernel of $(\phi^\tau_{-t})^*$ acting on $\C(M)$ and $\theta^\tau_{ik}$ respectively $b^\tau_{kj}$ are the components of $\theta_\tau$ and of $(\phi^\tau_{-t})^*$ in $\wedge^k T^*M\otimes E$. The latter are just smooth functions. Pulling back by $\iota:(t,x)\to(t,x,x)$, taking the trace in the fibers of the bundle and pairing with a function $\chi\in\C_c((t_0-\epsilon,t_0+\epsilon)\times U)$, we find the following local formula in the exact same way as in \cite[Appendix B]{zworski}:
    \begin{equation*}
        \pair{\mathrm{Tr}(\theta_\tau K_\tau^{(k)})}{\chi} = \frac{1}{|\mathrm{det}(1-P_\gamma)|}\int_{-\epsilon}^\epsilon \chi(t_0,\phi^\tau_s(x_0))\sum_{j,k}\theta^\tau_{jk}(\phi^\tau_s(x_0))b^\tau_{kj}(t,\phi^\tau_s(x_0))\,ds
    \end{equation*}
    Using a partition of unity leads to a global Guillemin-like trace formula of the form
    \begin{equation*}
        \tr\bigl(\theta_\tau e^{-t\L_{X_\tau}}\big|_{\Omega^k(M,E)}\bigr) = \sum_{\gamma\in\P_\tau} \frac{\mathrm{tr}(\rho([\gamma]))}{|\mathrm{det}(1-P_\gamma)|} \int_0^{T_\gamma^\#} \mathrm{tr}\bigl(\theta_\tau(\phi^\tau_s(x_0))\cdot\wedge^kd\phi^\tau_{-T_\gamma}(\phi^\tau_s(x_0))\bigr)\,ds\cdot \delta(t-T_\gamma),
    \end{equation*}
    where $x_0$ is any point on the closed orbit $\gamma$.

    Assumption \ref{assumption_G_convergencce} is now equivalent to the locally uniform convergence of the sum
    \begin{equation}
    \label{ruelle_sum}
        \sum_{\gamma\in\P_\tau} \frac{|\mathrm{tr}(\rho([\gamma]))|}{|\mathrm{det}(1-P_\gamma)|} \Bigl|\int_0^{T_\gamma^\#} \mathrm{tr}\bigl(\theta_\tau(\phi^\tau_s(x_0))\cdot\wedge^kd\phi^\tau_{-T_\gamma}(\phi^\tau_s(x_0))\bigr)\,ds\Bigr|\,T_\gamma^{\Re(s)-1}e^{-\Re(\lambda) T_\gamma}.
    \end{equation}
    Note first that for any $\tau_0\in(-1,1)$ we can choose $T_0>0$ such that $T_\gamma > T_0$ for every $\gamma\in\P_\tau$ and $\tau$ near $\tau_0$, see \cite[Remark 4]{viet_dang_fried_conjecture} together with \eqref{min_period}. Thus, there is no issue of convergence as the periods of $\gamma$ become small. Since $\rho$ is a unitary representation, $\mathrm{tr}(\rho([\gamma]))$ is bounded by some constant depending only on the rank of $E$. Moreover, by the arguments following \eqref{aux_func_ruelle}, $|\mathrm{det}(1-P_\gamma)|$ can be bounded by a constant uniformly for $\tau$ near $\tau_0$. We can estimate the integral in \eqref{ruelle_sum} by
    \begin{equation}
    \label{flow_integral_estimate}
        \Bigl|\int_0^{T_\gamma^\#} \mathrm{tr}\bigl(\theta_\tau(\phi^\tau_s(x_0))\cdot\wedge^kd\phi^\tau_{-T_\gamma}(\phi^\tau_s(x_0))\bigr)\,ds\Bigr| \leq T_\gamma^\#\sup_{x\in M}\|\theta_\tau(x)\|\cdot\sup_{x\in M}\|d\phi^\tau_{-T_\gamma(x)}\|^k
    \end{equation}
    where we denoted by $\|\cdot\|$ the operator norm of the finite-dimensional linear operators for a choice of norm on the fibers of $T^*M$. Since $\theta_\tau$ depends smoothly on $\tau$, $\|\theta_\tau(x)\|$ is bounded on the compact manifold $M$, locally uniformly in $\tau$. Furthermore, by the smoothness of $(\tau,t,x) \to d\phi^\tau_t(x)$, we have
    $$\sup\bigl\{\bigr\|d\phi^\tau_{-t}(x)\| \,\,|\,\, t\in[0,1],\,x\in M,\, \tau \in [\tau_0-\epsilon,\tau_0+\epsilon]\} = \beta <\infty$$
    for some $\beta\geq 1$.
    Writing $t\in\R_+$ as $t=m+s$ for $m\in\N$ and $s\in[0,1)$, we find
    $$\|d\phi^\tau_{-t}(x)\| = \|d\phi^\tau_{-s}\circ d\phi^\tau_{-1}\circ\cdots\circ d\phi^\tau_{-1}(x)\| \leq \beta^{m+1} \leq \beta e^{\log(\beta)t}$$
    uniformly for $\tau\in [\tau_0-\epsilon,\tau_0+\epsilon]$. Thus, we have an exponential bound on \eqref{flow_integral_estimate} that is locally uniform in $\tau$. By \cite[Remark 4]{viet_dang_fried_conjecture}, we also obtain a bound on the growth of periods as in \eqref{period_growth} uniformly for $\tau\in [\tau_0-\epsilon,\tau_0+\epsilon]$. Therefore, we see that there is some $C>0$ large enough such that the sum in \eqref{ruelle_sum} converges locally uniformly in $(\tau,\lambda,s)$ for all $\tau\in(-1,1)$, $s\in\mathbb{C}$ and $\Re(\lambda)>C$, proving Assumption \ref{assumption_G_convergencce}.

    Note that the integral, or rather distributional pairing,
    $$G(\tau,\lambda,s) = \int_0^\infty \str(\theta_\tau e^{-t\L_{X_\tau}})t^{s-1}e^{-\lambda t}\,dt,$$
    defining the function $G$ which appears in Assumption \ref{assumption_analytic_continuation}, is convergent for all $s\in\mathbb{C}$. Thus, we can directly set $s=1$ and only need to find an analytic continuation of $G(\tau,\lambda,1)$ to $\lambda=0$, see Remark \ref{rmk_eliminating_lambda_s}. Note from the proof of Theorem \ref{thm_invariance} that $\lambda G(\tau,\lambda,1)$ is precisely the $\tau$-derivative of $\log\sdet((\L_{X_\tau}+\lambda)|_{\im(\iota_{X_\tau})}\bigr)$. As in \eqref{trace_resolvent}, we now find for $\Re(\lambda)>C$:
    \begin{equation*}
        G(\tau,\lambda,1) = \int_0^\infty \str(\theta_\tau e^{-t\L_{X_\tau}})e^{-\lambda t}\,dt = e^{-\lambda t_0}\str(\theta_{\tau}e^{-t_0\L_{{X_\tau}}}R_{{X_\tau}}(\lambda)),
    \end{equation*}
    where $R_{{X_\tau}}(\lambda) = (\L_{{X_\tau}}+\lambda)^{-1}$ is the resolvent of $\L_{{X_\tau}}$. The analytic continuation in $\lambda$ of this expression to a locally bounded function of $(\tau,\lambda) \in (-1,1)\times \tilde{Z}$ for some domain $\tilde{Z}\subset\mathbb{C}$ with $0\in\tilde{Z}$ is precisely  \cite[Theorem 4]{viet_dang_fried_conjecture}.
\end{proof}

\appendix

\section{Construction of the heat kernel for a smooth family of elliptic operators}
\label{appendix_heat_kernel}

In this appendix we provide a self-contained construction of the heat kernel for a smooth family of positive definite elliptic operators $\tau\to D_\tau$, paying special attention to the smooth dependence on the parameter $\tau$.

The heat operator $e^{-tD_\tau}$ can be obtained from the spectral theory of the positive operator $D_\tau$. For concreteness, define the open set
$$\Lambda=\{\lambda\in\mathbb{C} \,|\, \Re(\lambda)+1<2|\Im(\lambda)|\}.$$
Note that $\Lambda$ is disjoint from the spectrum of $D_\tau$ and $|\lambda|$ is bounded away from zero on $\Lambda$.
Let $t\in\R \to \gamma(t)\in\mathbb{C}$ be a smooth contour contained in $\Lambda$ such that $\Re(\gamma(t))\to\infty$ as $t\to\pm\infty$ and for some $C>0$, we have
$$\Im(\gamma(t)) = \Re(\gamma(t)) \text{ for } t>C \quad\text{ and }\quad \Im(\gamma(t)) = -\Re(\gamma(t)) \text{ for } t<-C.$$
Then the holomorphic functional calculus gives
\begin{equation}
\label{contour_integral}
    e^{-tD_\tau} = \frac{1}{2\pi i}\int_\gamma e^{-t\lambda}(D_\tau-\lambda)^{-1}\,d\lambda
\end{equation}
and the heat kernel can be obtained as the Schwartz kernel of this operator. However, in order to control the $\tau$-dependence of the heat kernel, we will follow a different route. We first construct an approximate heat kernel, which depends in a straightforward way on the symbol of $D_\tau$. This is obtained by replacing the resolvent in \eqref{contour_integral} with a parametrix, i.e. an approximate inverse, for $(D_\tau-\lambda)$. The heat kernel can then be obtained from the approximate heat kernel by a converging Volterra series. Following this procedure, we can show that the heat kernel depends smoothly on $\tau$. We mainly follow \cite{gilkey} for the construction of an approximate heat kernel from a parametrix and \cite{berline_getzler} for the relation between approximate and actual heat kernel.

For the elliptic parametrix construction, we must introduce an appropriate symbol calculus.
\begin{defn}
\label{def_symbol_class}
    Let $U \subset \R^n$ be open and take $\Lambda\subset\mathbb{C}$ as above. Fix $m,r\in\N$. For $k\in\R$, we say that $a\in S^k_{\Lambda,m}(U;\mathbb{C}^{r\times r})$ is a $k$-th order symbol depending on the complex parameter $\lambda$ (with values in $\mathbb{C}^{r\times r}$) if the following hold:
    \begin{itemize}
        \item $a \in \C(U\times\R^n\times\Lambda;\mathbb{C}^{r\times r})$ and $a(x,\xi,\lambda)$ is an analytic function of $\lambda\in\Lambda$ for fixed $(x,\xi)$.
        \item For any multi-indices $\alpha,\beta,\gamma$ and any compact $K\subset U$, there is $C=C_{\alpha,\beta,\gamma,K}>0$ such that
        \begin{equation}
        \label{symbolic_estimates}
            \bigl|\partial_x^\alpha \partial_\xi^\beta \partial_\lambda^\gamma a(x,\xi,\lambda)\bigr| \leq C(1+|\xi|+|\lambda|^{\frac{1}{m}})^{k-|\beta|-m|\gamma|}, \quad \forall\,x\in K,\,\,\xi\in\R^n,\,\,\lambda\in\Lambda.
        \end{equation}
    \end{itemize}
    We equip $S^k_{\Lambda,m}(U;\mathbb{C}^{r\times r})$ with a topology using as seminorms the best constants in \eqref{symbolic_estimates}, that is for each $K\subset U$ compact and $j\in\N$, we define the seminorm
    \begin{equation}
    \label{symbol_seminorms}
        \|a\|_{k,K,j} = \sum_{|\alpha|+|\beta|+|\gamma|\leq j}\sup_{(x,\xi,\lambda)\in K\times\R^n\times\Lambda} (1+|\xi|+|\lambda|^{\frac{1}{m}})^{|\beta|+m|\gamma|-k}\bigl|\partial_x^\alpha \partial_\xi^\beta \partial_\lambda^\gamma a(x,\xi,\lambda)\bigr|.
    \end{equation}
    We say that $a\in S^k_{\Lambda,m}(U;\mathbb{C}^{r\times r})$ is homogeneous of order $k$ if
    $$a(x,t\xi,t^m\lambda) = t^k a(x,\xi,\lambda), \qquad \forall\,x\in U,\,\, \xi\in\R^n,\,\, \lambda\in\Lambda,\,\, t\geq 1.$$
\end{defn}

Note that at fixed $\lambda$ these parameter-dependent symbols just give rise to elements of the standard symbol calculus and can be quantized by the usual procedure. As shown in the following lemma, the symbols of Definition \ref{def_symbol_class} exhibit nice behavior when used in contour integrals of the form \eqref{contour_integral}.

\begin{lemma}
\label{lemma_kernel_local}
    Let $a \in S^{-m-N}_{\Lambda,m}(U;\mathbb{C}^{r\times r})$ be homogeneous of order $-m-N$ for some $N\geq 0$ and define
    $$K_a(t,x,y) = \frac{1}{(2\pi)^n}\frac{1}{2\pi i}\int_{\R^n}\int_\gamma e^{i(x-y)\xi-t\lambda}a(x,\xi,\lambda)\,d\lambda d\xi.$$
    Then the following holds:
    \begin{itemize}
        \item $K_a(t,x,y) \in \C(\R_+\times U\times U;\mathbb{C}^{r\times r})$,
        \item for any $\alpha,\beta,\gamma$ and $K\subset U$ compact, there exists $M\in\N$ such that
        \begin{equation}
        \label{heat_kernel_bound}
            |\partial_x^\alpha \partial_y^\beta \partial_t^\gamma K_a(t,x,y)| \leq C t^{\frac{N-n-|\alpha|-|\beta|-m|\gamma|}{m}}\|a\|_{-m-N,K,M} \quad \forall\, x\in K,\,\, y\in U,\,\, t\in (0,1),
        \end{equation}
        \item for any $l\in\N$, integration against $K_a$ defines an operator
        $$E_a(t):C_c^l(U;\mathbb{C}^r) \to C^l(U;\mathbb{C}^r), \quad E_a(t)f(x) = \int_U K_a(t,x,y)f(y)\,dy$$
        satisfying the following bound for any $K\subset U$ compact and some $M\in\N$:
        \begin{equation}
        \label{heat_op_bound}
            \|E_a(t)f\|_{C^l(K)} \leq C t^{\frac{N}{m}}\|a\|_{-m-N,K,M}\|f\|_{C^l(U)}, \quad \forall\, t\in (0,1).
        \end{equation}
    \end{itemize}
\end{lemma}
\begin{proof}
    For $t>0$ we can write $e^{-t\lambda} = t^{-M}(-\partial_\lambda)^Me^{-t\lambda}$. Integration by parts then gives
    \begin{equation*}
    \begin{split}
        \partial_x^\alpha \partial_y^\beta \partial_t^\gamma K_a(t,x,y) &= \frac{1}{(2\pi)^n}\frac{1}{2\pi i}\int_{\R^n}\int_\gamma e^{i(x-y)\xi-t\lambda}(\partial_x+i\xi)^\alpha(-i\xi)^\beta(-\lambda)^\gamma a(x,\xi,\lambda)\,d\lambda d\xi \\
        &= \frac{1}{(2\pi)^n}\frac{1}{2\pi i}t^{-M}\int_{\R^n}\int_\gamma e^{i(x-y)\xi-t\lambda}\partial_\lambda^M\bigl((\partial_x+i\xi)^\alpha(-i\xi)^\beta(-\lambda)^\gamma a(x,\xi,\lambda)\bigr)\,d\lambda d\xi.
    \end{split}
    \end{equation*}
    The integrand is bounded by $Ce^{-t\Re(\lambda)}(1+|\xi|+|\lambda|^{\frac{1}{m}}\bigr)^{-m-N+|\alpha|+|\beta|+m|\gamma|-M}$ for $x$ in a compact subset of $U$, by the symbolic estimates for $a$. So choosing $M$ large enough, we see that the integral converges absolutely. Hence, $K_a$ defines a smooth function for $t>0$. 
    
    For the small $t$ asymptotics, we use the homogeneity of $a$. Making the change of variables $\xi\to t^{-\frac{1}{m}}\xi$, $\lambda \to t^{-1}\lambda$ and shifting the contour $\gamma$ within $\Lambda$, we find for $t\in(0,1)$ and $x\in K$:
    \begin{equation*}
    \begin{split}
        |K_a(t,x,y)| &= t^{-\frac{n}{m}-1}\Bigl|\frac{1}{(2\pi)^n}\frac{1}{2\pi i}\int_{\R^n}\int_\gamma e^{it^{-\frac{1}{m}}(x-y)\xi-\lambda}a(x,t^{-\frac{1}{m}}\xi,t^{-1}\lambda)\,d\lambda d\xi\Bigr| \\
        &= t^{\frac{N-n}{m}}\Bigl|\frac{1}{(2\pi)^n}\frac{1}{2\pi i}\int_{\R^n}\int_\gamma e^{it^{-\frac{1}{m}}(x-y)\xi-\lambda}a(x,\xi,\lambda)\,d\lambda d\xi\Bigr| \\
        &= t^{\frac{N-n}{m}}\Bigl|\frac{1}{(2\pi)^n}\frac{1}{2\pi i}\int_{\R^n}\int_\gamma e^{it^{-\frac{1}{m}}(x-y)\xi-\lambda}\partial_\lambda^M a(x,\xi,\lambda)\,d\lambda d\xi\Bigr| \\
        &\leq C t^{\frac{N-n}{m}}\|a\|_{-m-N,K,M}\int_{\R^n}\int_\gamma e^{-\Re(\lambda)}(1+|\xi|+|\lambda|)^{-m-N-M}\,d\lambda d\xi,
    \end{split}
    \end{equation*}
    which gives the desired bound when $M$ is chosen large enough for the integral over $\xi$ to converge. The bound in \eqref{heat_kernel_bound} follows similarly. Indeed, taking derivatives in $x,y,t$ results in extra factors of $\xi$ and $\lambda$ in the integrand, which upon changing variables as above leads to extra powers of $t^{-\frac{1}{m}}$.

    Taking $f\in C^l_c(U,\mathbb{C}^r)$ and again using the homogeneity of $a$ and a change of variables, we find for $x\in K$, $t\in(0,1)$:
    \begin{equation*}
    \begin{split}
        |E_a(t)f(x)| &= \Bigl|\frac{1}{(2\pi)^n}\frac{1}{2\pi i}\int_U\int_{\R^n}\int_\gamma e^{i(x-y)\xi-t\lambda}a(x,\xi,\lambda)f(y)\,d\lambda d\xi dy\Bigr| \\
        &= t^{-\frac{n}{m}-1}\Bigl|\frac{1}{(2\pi)^n}\frac{1}{2\pi i}\int_U\int_{\R^n}\int_\gamma e^{it^{-\frac{1}{m}}(x-y)\xi-\lambda}a(x,t^{-\frac{1}{m}}\xi,t^{-1}\lambda)f(y)\,d\lambda d\xi dy\Bigr| \\
        &= t^{\frac{N}{m}}\Bigl|\frac{1}{(2\pi)^n}\frac{1}{2\pi i}\int_U\int_{\R^n}\int_\gamma e^{i(t^{-\frac{1}{m}}x-y)\xi-\lambda}a(x,\xi,\lambda)f(t^{\frac{1}{m}}y)\,d\lambda d\xi dy\Bigr| \\
        &\leq C t^{\frac{N}{m}} \|a\|_{-m-N,K,M}\sup_{y\in U}|f(y)|\int_{\R^n}\int_\gamma e^{-\Re(\lambda)}(1+|\xi|+|\lambda|)^{-m-N-M}\,d\lambda d\xi.
    \end{split}
    \end{equation*}
    Derivatives with respect to $x$ can be handled by noting that
    $$\partial_x^\alpha\int_U e^{i(x-y)\xi}a(x,\xi,\lambda)f(y)\,dy = \sum_{|\beta|+|\gamma|=|\alpha|}\frac{\alpha!}{\beta!\gamma!}\int_U e^{i(x-y)\xi}\partial_x^\beta a(x,\xi,\lambda)\partial_y^\gamma f(y)\,dy$$
    where we used that $\partial_x^\alpha e^{i(x-y)\xi} = (-\partial_y)^\alpha e^{i(x-y)\xi}$ and partial integration with respect to $y$.
\end{proof}

We will now construct an approximate heat kernel for $D_\tau$ by using a parametrix for $(D_\tau-\lambda)$ in the contour integral of \eqref{contour_integral}. The parametrix is obtained as the quantization of a finite sum of homogeneous symbols in the symbol class of Definition \ref{def_symbol_class}. For the sake of simplicity, in the following we will denote by $\V=\wedge^\bullet T^*M\otimes E$ the corresponding vector bundle.
\begin{lemma}
\label{lemma_approx_heat_kernel}
    Let $\tau\in(-1,1)\to D_\tau$ be a smooth family of positive definite elliptic operators, as in Proposition \ref{invariance_elliptic}. Then for every $N\in\N$ there is a smooth family of approximate heat kernels
    $$K_N(\tau,t,x,y) \in \C\bigl((-1,1)\times\R_+\times M\times M;\V\boxtimes (\V^*\otimes \wedge^nT^*M)\bigr)$$
    satisfying
    \begin{itemize}
        \item for any $T>0$, $I\subset (-1,1)$ compact and $i,l\in\N$ the operator defined by the kernel $\partial_\tau^iK_N$:
        $$\partial_\tau^iE_N(\tau,t):C^l(M,\V)\to C^l(M,\V), \quad \partial_\tau^iE_N(\tau,t)f(x) = \int_M \partial_\tau^iK_N(\tau,t,x,y)f(y)$$
        is bounded (w.r.t.\ the $C^l$-norm) uniformly for $\tau\in I$ and $t\in(0,T]$,
        \item for every $\tau$, we have $E_N(\tau,t)f \to f \in C^l(M,\V)$ as $t\to 0$,
        \item denoting $$S_N(\tau,t,x,y) = (\partial_t+D_\tau)K_N(\tau,t,x,y),$$ we have for any $T>0$, $I\subset (-1,1)$ compact:
        \begin{equation}
        \label{remainder_bound}
            \|\partial_\tau^i\partial_t^jS_N(\tau,t,x,y)\|_{C^l(M\times M)} \leq C t^{\frac{N-n-l-jm}{m}}
        \end{equation}
        uniformly for $\tau\in I$ and $t\in(0,T]$.
    \end{itemize}
\end{lemma}
\begin{proof}
    We will first work locally and construct an approximate heat kernel in local coordinates. We then patch together the local constructions using a partition of unity. Thus, let $U\subset\R^n$ be the image of a coordinate chart over which the bundle $\V$ is trivial. In these local coordinates we can write
    $$D_\tau = \sum_{|\alpha|\leq m} a_\alpha(\tau,t)\partial_x^\alpha , \quad\text{for some } a_\alpha(\tau,x) \in \C((-1,1)\times U;\mathbb{C}^{r\times r}),$$
    where $r$ is the rank of the bundle $\V$. The full symbol of $D_\tau$ is given by
    $$d(\tau,x,\xi) = \sum_{j\leq m} d_j(\tau,x,\xi), \quad\text{where } d_j(\tau,x,\xi) = \sum_{|\alpha|=m-j} a_\alpha(\tau,t)(i\xi)^\alpha,$$
    is a homogeneous polynomial of degree $m-j$ in $\xi.$ Note that $d_0(\tau,x,\xi)$ is the principle symbol of $D_\tau$.
    We will apply the elliptic parametrix construction and obtain a local version of the approximate heat kernel as
    \begin{equation}
    \label{approx_heat_kernel}
        K_N(\tau,t,x,y) = \frac{1}{(2\pi)^n}\frac{1}{2\pi i}\int_{\R^n}\int_\gamma e^{i(x-y)\xi-t\lambda}q^N(\tau,x,\xi,\lambda)\,d\lambda d\xi,
    \end{equation}
    where for each $\tau$, $q^N(\tau,x,\xi,\lambda) \in S^{-m}_{\Lambda,m}(U;\mathbb{C}^{r\times r})$ will be the symbol of a parametrix for $(D_\tau-\lambda)$, modulo error terms of order $-m-N$.
    Note that
    $$(\partial_t + D_\tau) K_N(\tau,t,x,y) = \frac{1}{(2\pi)^n}\frac{1}{2\pi i}\int_{\R^n}\int_\gamma e^{i(x-y)\xi-t\lambda}\Bigl(\sum_{|\alpha|\leq m}\frac{(-i)^{|\alpha|}}{\alpha!}\partial_\xi^\alpha \bigl(d(\tau,x,\xi)-\lambda\bigr) \partial_x^\alpha q^N(\tau,x,\xi,\lambda)\Bigr)\,d\lambda d\xi,$$
    as follows from the composition formula for pseudodifferential operators or a direct computation. We make the ansatz
    $$q^N(\tau,x,\xi,\lambda) = \sum_{0\leq k<m+N}q_k(\tau,x,\xi,\lambda), \quad\text{with } q_k \in S^{-m-k}_{\Lambda,m}(U;\mathbb{C}^{r\times r})$$
    and attempt to solve for the $q_k$ so that
    \begin{equation}
    \label{parametrix_formula}
        \sum_{|\alpha|\leq m}\frac{(-i)^{|\alpha|}}{\alpha!}\partial_\xi^\alpha \bigl(d(\tau,x,\xi)-\lambda\bigr) \partial_x^\alpha q^N(\tau,x,\xi,\lambda) = 1 + r^N(\tau,x,\xi,\lambda), \quad\text{with } r^N \in S^{-m-N}_{\Lambda,m}(U;\mathbb{C}^{r\times r}).
    \end{equation}
    Defining $\tilde{d_0}=d_0-\lambda$ and $\tilde{d_j}=d_j,$ for $1\leq j\leq m$, so that $\tilde{d_j} \in S^{m-j}_{\Lambda,m}(U;\mathbb{C}^{r\times r})$ for each $j$, \eqref{parametrix_formula} becomes
    \begin{equation}
    \label{parametrix_by_order}
    \begin{split}
        \sum_{\substack{k<m+N \\ j\leq m \\ |\alpha|\leq m}} \frac{(-i)^{|\alpha|}}{\alpha!}\partial_\xi^\alpha \tilde{d}_j \partial_x^\alpha q_k = \sum_{M}\,\sum_{\substack{j+k+|\alpha|=M \\ 0\leq k < m+N \\ 0\leq j,|\alpha| \leq m}}\frac{(-i)^{|\alpha|}}{\alpha!}\partial_\xi^\alpha \tilde{d}_j \partial_x^\alpha q_k = 1 \quad\text{modulo } S^{-m-N}_{\Lambda,m}(U;\mathbb{C}^{r\times r}).
    \end{split}
    \end{equation}
    Note that $\partial_\xi^\alpha \tilde{d}_j \partial_x^\alpha q_k$ is a symbol of order $-M=-j-k-|\alpha|$. Thus, at order zero \eqref{parametrix_formula} is satisfied by setting $q_0=\tilde{d_0}^{-1} = (d_0-\lambda)^{-1}$.
    Since $D_\tau$ is positive definite, the principle symbol $d_0(\tau,x,\xi)$ is a positive definite matrix for all $\xi\neq 0$. Thus, we indeed have
    $$q_0(\tau,x,\xi,\lambda) = (d_0(\tau,x,\xi) - \lambda)^{-1} \in S^{-m}_{\Lambda,m}(U;\mathbb{C}^{r\times r}).$$
    Moreover, $q_0$ is homogeneous of order $-m$ and from the smooth $\tau$ dependence of $d_0$, it follows that $q_0$ depends smoothly on $\tau$ with respect to the seminorms defined in \eqref{symbol_seminorms}.
    We can now solve for the $q_k$ inductively. At order $-m-N<-M<0$ the finite sum in \eqref{parametrix_by_order} has only one term involving $q_M$ and we find the solution
    $$q_M(\tau,x,\xi,\lambda) = - (d_0(\tau,x,\xi) - \lambda)^{-1}\sum_{\substack{j+k+|\alpha|=M \\ 0\leq k < M \\ 0\leq j,|\alpha| \leq m}}\frac{(-i)^{|\alpha|}}{\alpha!}\partial_\xi^\alpha d_j(\tau,x,\xi) \partial_x^\alpha q_k(\tau,x,\xi,\lambda),$$
    which is a homogeneous symbol in $S^{-m-M}_{\Lambda,m}(U;\mathbb{C}^{r\times r})$. Note that each $q_M$ is the finite sum of terms involving derivatives of $(d_0-\lambda)^{-1}$ and the $d_j$, and hence depends continuously on $\tau$ in the symbol topology. Setting $q^N = \sum_{M<m+N}q_M$, equation \eqref{parametrix_formula} is satisfied with
    $$r^N(\tau,x,\xi,\lambda) = \sum_{\substack{j+k+|\alpha|\geq m+N \\ 0\leq k < m+N \\ 0\leq j,|\alpha| \leq m}}\frac{(-i)^{|\alpha|}}{\alpha!}\partial_\xi^\alpha d_j(\tau,x,\xi) \partial_x^\alpha q_k(\tau,x,\xi,\lambda)$$
    a finite sum of homogeneous symbols of order $\leq -m-N$ depending smoothly on $\tau$ in the respective symbol topology.

    The local versions of the statements in Lemma \ref{lemma_approx_heat_kernel} now follow by applying Lemma \ref{lemma_kernel_local} to the kernel in \eqref{approx_heat_kernel} and using the smoothness of $\tau \to q^N(\tau,x,\xi,\lambda)$ with respect to the symbol seminorms. The first point in Lemma \ref{lemma_kernel_local} shows that $K_N(\tau,t,x,y) \in \C((-1,1)\times\R_+\times U\times U)$. The third point shows that for any $i,l\in\N$ and $K\subset U$ compact, we have
    $$\Bigl\|\int_U \partial_\tau^i K_N(\tau,t,x,y)f(y)\,dy\Bigr\|_{C^l(K)} \leq C\|f\|_{C^l(U)} \quad \forall\, f\in C^l_c(U)$$
    uniformly for $\tau\in I$ and $t\in(0,T]$, where $I\subset(-1,1)$ compact and $T>0$ are arbitrary. Note that uniformity for $t\in [1,T]$ follows simply from the smoothness of $K_N$. Moreover, in the limit $t\to 0$, we have for each $f\in C^l_c(U)$:
    \begin{equation*}
    \begin{split}
        \int_U K_N(\tau,t,x,y)f(y)\,dy &= \frac{1}{(2\pi)^n}\frac{1}{2\pi i}\int_U\int_{\R^n}\int_\gamma e^{i(x-y)\xi-t\lambda}q_0(\tau,x,\xi,\lambda)f(y)\,d\lambda d\xi dy + \mathcal{O}(t)\\
        &= \frac{1}{(2\pi)^n}\frac{1}{2\pi i}\int_{\R^n}\int_U\int_\gamma e^{i(x-y)\xi-t\lambda}\bigl(d_0(\tau,x,\xi)-\lambda\bigr)^{-1}f(y)\,d\lambda d\xi dy + \mathcal{O}(t) \\
        &= \frac{1}{(2\pi)^n}\int_{\R^n}\int_U e^{i(x-y)\xi-td_0(\tau,x,\xi)}f(y)\,d\xi dy + \mathcal{O}(t) \xrightarrow{t\to 0} f(x),
    \end{split}
    \end{equation*}
    where we used Cauchy's integral theorem and the fact that the eigenvalues of the matrix $d_0(\tau,x,\xi)$ lie on the positive real axis. Finally, we have
    $$S_N(\tau,t,x,y) = (\partial_t+D_\tau)K_N(\tau,t,x,y) = \frac{1}{(2\pi)^n}\frac{1}{2\pi i}\int_{\R^n}\int_\gamma e^{i(x-y)\xi-t\lambda}r^N(\tau,x,\xi,\lambda)\,d\lambda d\xi,$$
    where $r_N \in S^{-m-N}_{\Lambda,m}(U;\mathbb{C}^{r\times r})$ depends smoothly on $\tau$ with respect to the symbol seminorms. Thus, the second point in Lemma \ref{lemma_kernel_local} shows that for any $i,j,l\in\N$ and $K\subset U$ compact, we have
    \begin{equation}
    \label{loc_remainder_est}
        \|\partial_\tau^i\partial_t^jS_N(\tau,t,x,y)\|_{C^l(K\times K)} \leq C t^{\frac{N-n-l-jm}{m}}
    \end{equation}
    locally uniformly in $\tau$ and $t$.

    We can now patch together the local construction to a define an approximate heat kernel globally on $M$. Cover $M$ by finitely many coordinate charts $U_\nu$ and let $\chi_\nu$ be a partition of unity subordinate to the cover. Let further $\tilde{\chi}_\nu \in \C_c(U_\nu)$ satisfy $\tilde{\chi}_\nu = 1$ on $\supp(\chi_\nu)$. We set
    $$K_N(\tau,t,x,y) = \sum_\nu \tilde{\chi}_\nu(x)K_N^{(\nu)}(\tau,t,x,y)\chi_\nu(y),$$
    where $K_N^{(\nu)}$ is the pullback to $M$ of the local approximate heat kernel constructed above. Note that $K_N$ depends smoothly on $(\tau,t,x,y)$ and if $(x,y)$ lies in the support of $K_N$ then $x,y$ must be contained in a coordinate chart $U_\nu$ for some $\nu$. The first two properties of Lemma \ref{lemma_approx_heat_kernel} now follow immediately from the local versions. For the third property, note that
    $$(\partial_t+D_\tau)K_N(\tau,t,x,y) = \sum_\nu \tilde{\chi}_\nu(x)S_N^{(\nu)}(\tau,t,x,y)\chi_\nu(y) + \sum_\nu [D_\tau, \tilde{\chi}_\nu(x)]K_N^{(\nu)}(\tau,t,x,y)\chi_\nu(y).$$
    The first term can be handled by the local estimates in \eqref{loc_remainder_est}. For the second term, notice that $[D_\tau, \tilde{\chi}_\nu(x)]$ is a differential operator supported away from $\supp(\chi_\nu)$. Thus, working in local coordinates we see that $[D_\tau, \tilde{\chi}_\nu(x)]e^{i(x-y)}q_\nu^N(\tau,x,\xi,\lambda)\chi_\nu(y)$ is supported away from $x=y$. Using $e^{i(x-y)\xi} = (x-y)^{-\alpha}(-i\partial_\xi)^\alpha e^{i(x-y\xi)}$ for $x\neq y$ and performing partial integration with respect to $\xi$, we find that
    \begin{equation*}
        [D_\tau, \tilde{\chi}_\nu(x)]K_N^{(\nu)}(\tau,t,x,y)\chi_\nu(y) = \frac{1}{(2\pi)^n}\frac{1}{2\pi i}\int_{\R^n}\int_\gamma e^{i(x-y)\xi-t\lambda}\,\tilde{r}_\nu^N(\tau,x,\xi,\lambda)\,d\lambda d\xi\,\chi_\nu(y),
    \end{equation*}
    where $\tilde{r}^N_\nu \in S^{-m-N}_{\Lambda,m}(U;\mathbb{C}^{r\times r})$ has compact $x$-support and depends smoothly on $\tau$ in the symbol topology. The estimate in \eqref{remainder_bound} now follows from the local estimates \eqref{heat_kernel_bound} of Lemma \ref{lemma_kernel_local}.
\end{proof}

We can now construct the actual heat kernel $K$ from the approximate heat kernel obtained in Lemma \ref{lemma_approx_heat_kernel}.

\begin{proof}[Proof of Lemma \ref{lemma_smooth_heat_kernel}]
    We will show that the heat kernel $K(\tau,t,x,y)$, i.e. the Schwartz kernel of $e^{-tD_\tau}$, depends smoothly on $\tau$. To this end, we follow \cite[Section 2.4]{berline_getzler} and obtain $K$ from the approximate heat kernel $K_N$ and remainder $S_N$ of Lemma \ref{lemma_approx_heat_kernel} by a Volterra series:
    \begin{equation}
    \label{volterra_series}
        K = \sum_{k=0}^\infty (-1)^k K_N*(S_N)^{*k},
    \end{equation}
    where for $A,B \in \C\bigl(\R_+\times M\times M;\V\boxtimes (\V^*\otimes \wedge^nT^*M)\bigr)$, we define the convolution product by
    $$A*B(t,x,y) = \int_0^t\int_M A(t-u,x,z)B(u,z,y)\,du$$
    whenever the integrals over $z$ and $u$ converge. 
    Notice that
    \begin{equation*}
        (S_N)^{*(k+1)}(\tau,t,x,y) = \int_{t\Delta_k}\int_{M^{k}}S_N(\tau,t-t_{k},x,z_{k})S_N(\tau,t_{k}-t_{k-1},z_{k},z_{k-1})\cdots S_N(\tau,t_{1},z_{1},y),
    \end{equation*}
    where the first integral is over the rescaled $k$-simplex
    $$t\Delta_k = \{0\leq t_1 \leq t_2 \leq \cdots \leq t_k \leq t\} \subset \R^k,$$
    Fix $J\in\N$, $T>0$ and $I\subset (-1,1)$ compact. Choosing $N$ large enough, Lemma \ref{lemma_approx_heat_kernel} shows that $S_N$ and its derivatives in $(\tau,t,x,y)$ up to order $J$ extend continuously to $t=0$ for $\tau\in I$. Thus, the integral defining $(S_N)^{*(k+1)}$ converges to an element of $C^J(I\times[0,T]\times M\times M)$. Moreover, \eqref{remainder_bound} leads to the estimate
    \begin{equation}
    \label{convolution_estimate}
        \bigl\|\partial_\tau^i\partial_t^j(S_N)^{*(k+1)}(\tau,t,x,y)\bigr\|_{C^l(M\times M)} \leq C^{(k+1)}\frac{t^k}{k!}\,t^{(k+1)\frac{N-n}{m}-\frac{l}{m}-j}
    \end{equation}
    uniformly for $\tau\in I$, $t\in [0,T]$, where we used that the Lebesgue measure of the rescaled $k$-simplex is $|t\Delta_k| = \frac{t^k}{k!}$. By Lemma \ref{lemma_approx_heat_kernel}, integration against $\partial_\tau^iK_N(\tau,t,x,y)$ defines a uniformly bounded operator on $C^l(M)$ for each $i$. Since $\partial_tK_N = S_N - D_\tau K_N$ by definition, we see that the kernel $\partial_\tau^i\partial_t^jK_N(\tau,t,x,y)$ defines a uniformly bounded operator from $C^{l+jm}(M)$ to $C^l(M)$. Thus,
    $$\Bigl\|\int_M \partial_\tau^i\partial_t^jK_N(\tau,t-u,x,z)(S_N)^{*k}(\tau,u,z,y)\Bigr\|_{C^l(M\times M)} \leq C_0\Bigl\|(S_N)^{*k}(\tau,u,z,y)\Bigr\|_{C^{l+jm}(M\times M)}$$
    uniformly for $0\leq u\leq t\leq T$ and $\tau\in I$. Using the estimate for $(S_N)^{*k}$ and integrating over $u$, we find for every $k\geq 1$:
    $$\bigl\|\partial_\tau^i\partial_t^j\bigl(K_N*(S_N)^{*k}\bigr)(\tau,t,x,y)\bigr\|_{C^l(M\times M)} \leq C_0C^k\frac{t^k}{(k-1)!}t^{\frac{k(N-n)}{m}-\frac{l}{m}-j}$$
    uniformly for $\tau\in I$, $t\in [0,T]$. Choosing $N$ large enough, we see that the sum in \eqref{volterra_series} converges in $C^J(I\times[0,T]\times M\times M)$. 
    Furthermore, using $(\partial_t-D_\tau)K_N=S_N$ and the definition of the convolution product, we see that
    $$(\partial_t-D_\tau)K_N*(S_N)^{*k} = (S_N)^{*(k+1)} + (S_N)^{*k}, \quad\text{for } k\geq 1.$$
    Thus, upon application of the operator $(\partial_t-D_\tau)$ the sum in \eqref{volterra_series} telescopes. We find that the left hand side of \eqref{volterra_series} satisfies
    $$(\partial_t-D_\tau)K = 0, \quad\text{and}\quad \int_M K(\tau,t,x,y)f(y) \xrightarrow{t\to 0} f(x) \quad \forall\, f\in\Omega^\bullet(M,E),$$
    where the second property follows thanks to the second point in Lemma \ref{lemma_approx_heat_kernel} and the estimates on $K_N*(S_N)^{*k}$ for $k\geq 1$. By uniqueness of the heat kernel for an elliptic operator on a compact manifold, see for instance \cite[Proposition 2.17]{berline_getzler}, the left hand side of \eqref{volterra_series} is independent of $N$ and constitutes the Schwartz kernel of $e^{-tD_\tau}$. Since $I,T$ and $J$ were arbitrary, we see that 
    \begin{equation}
        K \in \C\bigl((-1,1)\times \R_+\times M\times M;\V\boxtimes(\V^*\otimes\wedge^nT^*M)\bigr).
    \end{equation}
\end{proof}

\printbibliography

\end{document}